\documentclass[11pt,reqno]{amsart}

\usepackage[margin=1.1in]{geometry}

\usepackage{amssymb,cite}
\usepackage[mathscr]{eucal}
\usepackage{eufrak}
\usepackage{xspace}
\usepackage{mathtools}
\usepackage{color}

\usepackage{hyperref}
\hypersetup{hidelinks}

\usepackage[us,12hr]{datetime}    
 
\renewcommand{\div}{\text{\rm div}\,}
\newcommand{\ve}{\varepsilon}
\newcommand{\medint}{{\mbox{\vrule height3.5pt depth-2.8pt
          width4pt}\mkern-13mu\int\nolimits}}
\newcommand{\Medint}{\mkern12mu\mbox{\vrule height4pt
         depth-3.2pt
          width5pt}\mkern-16.5mu\int\nolimits}
\newcommand{\curl}{\text{\rm curl}\,}

\theoremstyle{plain}
\newtheorem{theorem}{Theorem}[section]

\newtheorem{lemma}{Lemma}[section]
\newtheorem{proposition}{Proposition}[section]
\theoremstyle{definition}

\theoremstyle{remark}

\newtheorem{remark}{Remark}[section]
\numberwithin{equation}{section}

\begin{document}

\title[Weak solutions for the compressible PNPNS system]{Global weak solutions for the compressible Poisson-Nernst-Planck-Navier-Stokes System}


\author{Daniel Marroquin}
\address{Instituto de Matem\'{a}tica - Universidade Federal do Rio de Janeiro\\
Av. Athos da Silveira Ramos, 149, Cidade Universit\'{a}ria, Rio de Janeiro, RJ, 21945-970, Brazil}
\curraddr{}
\email{marroquin@im.ufrj.br}

\author{Dehua Wang} 
\address{Department of Mathematics, University of Pittsburgh, Pittsburgh, PA, 15260, USA}
\curraddr{}
\email{dwang@math.pitt.edu}
\thanks{}




\date{\today}
 

\dedicatory{}

\begin{abstract}

We consider the compressible Poisson-Nernst-Planck-Navier-Stokes (PNPNS) system of equations, governing the transport of charged particles under the influence of the self-consistent electrostatic potential, in a three-dimensional bounded domain. We prove the existence of global weak solutions for the initial-boundary value problem with no-slip boundary condition for the fluid's velocity, blocking boundary condition for the ionic concentrations and inhomogeneous Robin boundary condition for the electrostatic potential, without restrictions on the size of the initial data. We derive the crucial energy dissipation of the system and prove the weak sequential stability of solutions of the Poisson-Nernst-Planck subsystem with respect to the velocity field of the fluid, which enables the proof of the existence of global weak solutions for the PNPNS system. We also study the large-time behavior of the solutions and justify the incompressible limit of the compressible PNPNS system as applications of the weak sequential stability of the solutions. New techniques and estimates are developed to overcome the difficulties from the strong interaction of the fluid with the ion particles and the physical boundary conditions.
\end{abstract}

\maketitle

\tableofcontents

\section{Introduction}\label{S1}

We are concerned with the compressible Poisson-Nernst-Planck-Navier-Stokes (PNPNS) system of equations, which models the transport of charged particles under the influence of the self-consistent electrostatic potential in a compressible fluid. 
The PNPNS system in the three-dimensional space is as follows: 
\begin{align}
&\partial_t\rho+\div(\rho u)=0,\label{e1.1}\\
&\partial_t(\rho u) + \div(\rho u\otimes u)+\nabla p = \div\mathbb{S}-\sum_{j=1}^n\nabla\varphi_j(c_j)+\kappa\Delta\Phi\nabla\Phi,\label{e1.2}\\
&\partial_t c_{j}+\div(c_ju)=\div\left( D_j\nabla\varphi_j(c_j)+D_jz_jc_j\nabla\Phi \right),\quad j=1,...,N,\label{e1.3}\\
&-\kappa \Delta\Phi=\sum_{j=1}^Nz_jc_j,\label{e1.5}
\end{align}  
where $\rho$, $u$ and $p=p(\rho)$ denote the fluid's density, velocity field and pressure; $\mathbb{S}$ denotes the viscous stress tensor given by
\begin{equation}\label{e1.6}
\mathbb{S}=\lambda(\div u)I + \mu(\nabla u + (\nabla u)^\perp),
\end{equation}
where the constants $\lambda$ and $\mu$ are the viscosity coefficients satisfying
\[
\mu >0,\quad \lambda + \frac{2}{3}\mu \ge 0.
\]
Regarding the pressure, we assume the following constitutive relation
\begin{equation}\label{e1.7}
p(\rho)=a\rho^\gamma,
\end{equation}
where $a>0$ and $\gamma>\frac{3}{2}$.
Moreover, the nonnegative functions $c_j$ are the ion concentrations and $\Phi$ is the electrostatic potential. Also, $\kappa$ is the dielectric permittivity, $D_j$ are positive diffusion coefficients and $z_j\in\mathbb{R}$ is the (constant) charge of the $j$-th ion species. Finally, the functions $\varphi_j(c_j)$ are strictly increasing functions defined in terms of the entropy densities $\sigma_j(c_j)$ by relations \eqref{e.entropyv}. 

The evolution of the ion concentrations is described by the Poisson-Nernst-Planck equations \eqref{e1.3}-\eqref{e1.5}. The mixture of the ion species give rise to a self-consistent electrostatic potential $\Phi$, given by the Poisson equation \eqref{e1.5}. Accordingly, $-\nabla\Phi$ is the electric field associated to the charge distribution $\sum_{j=1}^Nz_jc_j$ induced by the mixture. Equation \eqref{e1.3} models the mass balance of each ion species. The mass flux is decomposed into a diffusion term, $D_i\nabla\varphi_j(c_j)$, and an electromigration term, $D_jz_j  c_j\nabla\Phi$. In the case of Fickian diffusion, we have $\varphi_j(c_j)=c_j$, according to Fick's law of diffusion, which states that the flux of the ions should go from regions of high concentration to regions of low concentration, that is, in a direction proportional (and opposite) to the gradient of the concentration. Equation \eqref{e1.3} also contains a convective term $\div(c_j u)$, due to the fluid's movement. 
The evolution of the ion concentrations also exert forcing on the fluid, as accounted by the terms $-\sum\nabla\varphi_j(c_j)$ and $\kappa\Delta\Phi\nabla\Phi$ in the momentum equation \eqref{e1.2}, where the former is associated to the diffusion of ion particles and the latter corresponds to the Coulomb force produced by the charge distribution.
We refer to \cite{Ru,CI1,WLT1} and the references therein for an introduction to the physical and mathematical issues regarding the model.

We take the Fickian diffusion $\varphi_j(c_j)=c_j$ and consider the initial-boundary value problem of equations \eqref{e1.1}-\eqref{e1.5}   in a smooth bounded spatial domain $\Omega\subseteq \mathbb{R}^3$, subject to the following initial and boundary conditions:
\begin{equation}\label{e1.9}
(\rho,\rho u,c_j)(0,x)=(\rho_0,m_0,c_{j}^0)(x),\quad x\in\Omega;
\end{equation}
and for $t>0$,
\begin{align}
&u|_{\partial \Omega}=0,\label{e1.10}\\
&D_j\left(\partial_\nu c_j + z_jc_j\partial_\nu\Phi\right)|_{\partial \Omega} = 0,\quad j=1,...,N,\label{e1.11}\\
&(\partial_\nu\Phi + \tau\Phi)|_{\partial \Omega}=V,\label{e1.13}
\end{align}
where $\nu$ is the outer normal vector to the boundary $\partial\Omega$ of the domain, $\partial_\nu$ is the normal derivative at the boundary, $\tau>0$ is the (constant) capacity of the boundary and $V$ is a given (smooth) datum connected with an external electrical field.
Condition \eqref{e1.10} is a no-slip boundary condition for the velocity field of the fluid. In turn, the blocking boundary conditions \eqref{e1.11} model impermeable walls and yield the conservation of the averages of concentrations (cf. \cite{CI1}). Moreover, the Robin boundary condition \eqref{e1.13} accounts for electrochemical double layers at the boundary, which, in general, is expected to be charged (cf. \cite{BFS,FS}).
Our goal is to establish the existence, large-time behavior, and the incompressible limit of global finite energy weak solutions in the sense of Lions-Feireisl \cite{L2,F} to the above initial-boundary value problem \eqref{e1.1}-\eqref{e1.13}.

In the case of an incompressible fluid, the fluid's density is constant (scaled to be equal to $1$), the continuity equation \eqref{e1.1} reduces to the incompressibility condition $\div u=0$ and the constitutive relation for the pressure \eqref{e1.7} is dropped. Usually, in the incompressible setting, the resulting equations considered in the literature do not include the term $-\sum\nabla\varphi_j(c_j)$ in the momentum equation. However, being a gradient term, it can be incorporated into the pressure and the model is, thus, consistent with the above formulation. In the compressible case, this term is very important as, without it, the energy of the system is unbalanced. This has been observed in \cite{WLT1,WLT2}, where system \eqref{e1.1}-\eqref{e1.5}, for the case of two ionic species, has been derived by an energetic variational approach.

There is a lot of literature on the incompressible version of the PNPNS equations. The local existence of solutions in the whole space was proved by Jerome in \cite{Je}. Global existence of solutions for small data was obtained by Ryham \cite{Ry}. The Cauchy problem in dimension $2$ was considered by Zhang and Yin in \cite{ZhJ},  and in higher dimensions by Ma in \cite{Ma} and by Liu and Wang \cite{LW} (see also the works by Zhao, 
Deng and Cui \cite{ZDC1,ZDC2}). In bounded domains, global weak solutions have been shown to exist by Jerome and Sacco in \cite{JeS} and by Schmuck in \cite{Sch} with blocking boundary conditions for the ion concentrations and Neumann boundary conditions for the potential. In \cite{FaG} Fan and Gao prove the uniqueness of weak solutions in critical spaces. The existence of global weak solutions with blocking boundary conditions for the ion concentrations and Robin boundary conditions for the potential, which is more physically relevant than the Neumann one, was proved by Fischer and Saal in \cite{FS}. The global existence and stability of strong solutions for the two-dimensional system with blocking boundary conditions for the ions and Robin boundary conditions for the potential has been proved by Bothe, Fischer and Saal in \cite{BFS}. More recently, some other physically meaningful boundary conditions, namely blocking or Dirichlet boundary conditions for the ions and Dirichlet boundary conditions for the potential, were considered in a series of works by Constantin, Ignatova and Lee \cite{CI1,CIL2,CIL3,CIL4,Le} where existence, stability, long time behavior and regularity of global strong solutions is investigated in dimensions $2$ and $3$. Further regularity and long-time behaviour results have been investigated for periodic solutions by Abdo and Ignatova in \cite{AI1,AI2,AI3}. The quasi-neutral limit of the solutions has been studied by Li \cite{FCLi}, by Wang, Jiang and Liu \cite{WJL} and by Constantin, Ignatova and Lee \cite{CIL1}. Furthermore, the case of an inviscid fluid has been considered by Ignatova and Shu in \cite{ISh}.

Regarding the compressible case of the system, the literature is more limited, mostly focusing on smooth solutions in the whole space, which are either local or small; 
see the works in \cite{WLT1,WLT2,ToTa,SW}. 
%
In this paper we investigate the 
global weak solutions to the initial-boundary value problem of the compressible PNPNS system with no-slip boundary condition for the velocity, blocking boundary conditions for the ion concentrations and Robin boundary conditions for the electrostatic potential, as described above, without restrictions on the size of the initial data, nor on the number of ionic species. 

The analysis of system \eqref{e1.1}-\eqref{e1.5} relies on its dissipative structure. It is expected that the energy of the system dissipates as time evolves. Mathematically, the energy dissipation gives rise to the natural {\it a priori} estimates for the solutions and yield the natural function spaces that provide a consistent notion of solution. The energy dissipation is sensitive with respect to the boundary conditions for the ion concentrations and the electrostatic potential. More physically meaningful boundary conditions usually lead to mathematical challenges due to the loss of conservation of certain quantities, which makes the analysis more intricate.



\subsection{Main results}

As mentioned above, we focus on the case of Fickian diffusion where $\varphi_j(c_j)=c_j$, $j=1,...,N$. Also, for simplicity, we consider the case where the diffusivities $D_j$ are all positive constants; although the results below hold also for the case where the diffusivities depend on $x$, with minor and straightforward modifications, as long as they are strictly positive and bounded.  
%
Then, we consider problem \eqref{e1.1}-\eqref{e1.13} posed on a smooth bounded spatial domain $\Omega\subseteq \mathbb{R}^3$ and assume that the initial data satisfy the following conditions:
\begin{equation}\label{e1.initial}
\begin{cases}
\displaystyle \rho_0\in L^\gamma(\Omega),\, \rho_0\ge 0,\\
\displaystyle m_0\in L^1(\Omega), \text{ with $m_0(x)=0$ if $\rho_0(x)=0$, } \frac{|m_0|^2}{\rho_0}\in L^1(\Omega),\\
\displaystyle c_j^0\ge 0,\,  c_j^0\in L^1(\Omega), \, c_j\ln c_j  \in L^1(\Omega).
\end{cases}
\end{equation}
We say that $(\rho, u, c_1,...,c_N, \Phi)(t,x)$ is a {\em finite energy weak solution} of \eqref{e1.1}-\eqref{e1.13} for $t\in[0,T]$ with $T>0$ a given constant and $x\in\Omega$  if
\begin{itemize}
\item The density $\rho$ is nonnegative and 
\[
\rho\in C([0,T];L^1(\Omega))\cap L^\infty(0,T; L^\gamma(\Omega)),\, \rho(0,\cdot)=\rho_0;
\]
\item The velocity field $u$ satisfies
\[
\begin{split}
u\in L^2(0,T;H_0^1(\Omega)), \, \rho u\otimes u\in L^1((0,T)\times\Omega),\\
\rho u(0,\cdot) = m_0\, \text{ in the sense of distributions};
\end{split}
\]
\item The ion densities $c_j$, $j=1,...,N$, are nonnegative and
\[
c_j\in L^\infty(0,T;L^1(\Omega))\cap L^1(0,T;W^{1,\frac{3}{2}}(\Omega)),
\]
with $\sqrt{c_j}\in L^2(0,T;H^1(\Omega))$, $j=1,...,N$;
\item The electrostatic potential $\Phi$ satisfies $\Phi(t,x)=\Phi_1(x)+\Phi_2(t,x)$ with $\Phi_1\in W^{2,r}(\Omega)$ for some $r>3$ and
\[
\Phi_2\in L^\infty(0,T;H^1(\Omega))\cap L^1(0,T,W^{3,\frac{3}{2}}(\Omega))\cap C([0,T];L^p(\Omega)),\text{ for $p\in [1,6)$};
\]
\item The continuity and the momentum equations \eqref{e1.1} and \eqref{e1.2} are satisfied in the sense of distributions (with test functions which do not necessarily vanish at the boundary of $\Omega$, in accordance with \eqref{e1.10});
\item The continuity equation \eqref{e1.1} is satisfied in the sense of renormalized solutions, that is, the following equation 
\begin{equation}\label{e1.renorm}
b(\rho)_t + \div(b(\rho u) + \Big(b'(\rho)\rho - b(\rho)\Big)\div u = 0,
\end{equation}
is satisfied in the sense of distributions, for any $b\in C^1([0,\infty))$ such that 
\begin{equation}
b'(z) = 0 \text{ for all $z$ large enough, say, $z\ge z_0$}, 
\end{equation}
for some constant $z_0$ which depends on $b$;
\item The equation \eqref{e1.3} with \eqref{e1.11} and the initial condition $c_j(0,\cdot)=c_j^0$ is satisfied in the sense of distributions with test functions which do not necessarily vanish at the boundary of $\Omega$ or at $t=0$;
\item The function $\Phi$ is a strong solution of \eqref{e1.5}, and \eqref{e1.13} is satisfied in the sense of traces;
\item The solution satisfies the following energy inequality:
\begin{equation}\label{e1.energy}
E(t) + \int_0^t\int_\Omega \left(\mu|\nabla u|^2 + (\lambda+\mu)(\div u)^2 + \sum_{j=1}^ND_jc_j\Big|\frac{\nabla c_j}{c_j}+ z_j\nabla\Phi\Big|^2\right)dx\, ds
\le E(0),
\end{equation}
 where,
\begin{equation*}
E(t):=\int_\Omega\left( \rho\left(\frac{1}{2} |u|^2 + \frac{a}{\gamma-1}\rho^{\gamma-1} \right)  
+ \sum_{j=1}^N(c_j\ln c_j - c_j + 1) + \frac{\kappa}{2}|\nabla\Phi|^2\right) dx +\frac{\kappa\tau}{2}\int_{\partial\Omega}|\Phi|^2dS.
\end{equation*}
\end{itemize}

\begin{remark}
Note that the electrostatic potential $\Phi$ may always be decomposed as $\Phi(t,x)=\Phi_1(x)+\Phi_2(t,x)$, where $\Phi_1$ is the unique solution of
\[
\begin{cases}
\displaystyle -\kappa\Delta \Phi_1 = 0,&\text{in $\Omega$},\\
\displaystyle \partial_\nu \Phi_1 + \tau\Phi_1 = V,&\text{on $\partial\Omega$},
\end{cases}
\]
and $\Phi_2$ solves
\[
\begin{cases}
\displaystyle -\kappa\Delta \Phi_2 = \sum_{j=1}^Nz_jc_j,&\text{in $\Omega$},\\
\displaystyle \partial_\nu \Phi_1 + \tau\Phi_1 = 0,&\text{on $\partial\Omega$}.
\end{cases}
\]
\end{remark}

\begin{remark}\label{r1.2}
By the dominated convergence theorem, if $(\rho,u)$ is a renormalized solution of the continuity equation, then \eqref{e1.renorm} holds for any $b\in C^1(0,\infty)\cap C[0,\infty)$, such that
\begin{equation}
|b'(z)z|\le cz^{\frac{\gamma}{2}},\text{ for $z$ larger that some positive constant $z_0$.}
\end{equation}
\end{remark}

Our first  result is on the existence of global weak solutions and reads as follows.
{\it {\it a priori}}
\begin{theorem}[Existence of global weak solutions]\label{T1.1}
Let $\gamma>\frac{3}{2}$ and let the initial data satisfy \eqref{e1.initial}. Then, for any given $T>0$, there is a finite energy weak solution of \eqref{e1.1}-\eqref{e1.13}.
Moreover, there is a positive constant $C$, which depends only on $E(0)$, $V$, $\tau$ and $\kappa$, such that
\begin{multline}\label{e1.T_Phi}
D\int_{T_1}^{T_2}\int_\Omega \sum_{j=1}^N \left(\left|\nabla \sqrt{c_j}\right|^2+z_j^2c_j|\nabla \Phi|^2\right)dx\, dt+2\kappa D\int_{T_1}^{T_2}\int_\Omega|\Delta \Phi|^2 dx\, dt\\ 
\le E(0) + C(T_2-T_1)\sum_{j=1}^N|z_j|\,\| c_{j,0} \|_{L^1(\Omega)},
\end{multline}
for any $0\le T_1\le T_2$, where $\displaystyle D=\min_j D_j$.
\end{theorem}

We prove the existence of global finite energy weak solutions as a limit of solutions of a regularized system of equations by combining the theory in \cite{L2,F} for the Navier-Stokes equations with several {\it {\it {\it a priori}}} estimates on system \eqref{e1.1}-\eqref{e1.5} and weak stability results regarding the Poisson-Nernst-Planck (PNP) subsystem \eqref{e1.3}-\eqref{e1.5}. The {\it a priori} estimates, in particular the energy inequality \eqref{e1.energy}, are a consequence of the underlying dissipative structure of the equations. The dissipative structure  is crucial in the analysis  and we dedicate a whole section (see Section~\ref{S:2} below) to deducing an energy equation for the PNPNS system. The resulting energy equation allows for the analysis of the dissipation of energy with respect to the different physically meaningful boundary conditions that may be imposed on the ion densities and on the electrostatic potential.

Our second result is on the weak sequential stability of the solutions to the PNP subsystem.  
\begin{theorem}[Weak sequential stability for the PNP subsystem]\label{T1.2}
Let	$\{ u_n\}_{n\in\mathbb{N}}$ be a sequence in $L^2(0,T;H_0^1(\Omega))$ and let $(c^{(n)},\Phi^{(n)})$ be a weak solution of the PNP sub-system \eqref{e1.3}-\eqref{e1.5}, \eqref{e1.11}-\eqref{e1.13}, with $u=u_n$. 
Suppose that
\begin{equation}\label{e1.est}
\begin{cases}
\sqrt{c_j^{(n)}} \text{ is bounded in $L^2(0,T;H^1(\Omega))$,}\\
\Phi^{(n)} \text{ is bounded in $L^\infty(0,T;H^1(\Omega))$,}\\
\sqrt{c_j^{(n)}}\nabla\Phi^{(n)} \text{ is bounded in $L^2((0,T)\times\Omega)$.}
\end{cases}
\end{equation}
Assume also that
\begin{equation}\label{e1.estu}
u_n\rightharpoonup u \text{ weakly in $L^2(0,T;H_0^1(\Omega))$.}
\end{equation}
Then, there are $c_j\in L^\infty(0,T;L^1(\Omega))\cap L^1(0,T; W^{1,3/2}(\Omega))$, $j=1,...,N$, $\Phi\in L^\infty(0,T;H^1(\Omega))\cap C([0,T];L^p(\Omega))$, for all $p\in [1,6)$, and a subsequence of $(c^{(n)},\Phi^{(n)})$ (not relabeled) such that
\begin{align*}
&c_j^{(n)}\to c_j  \text{ strongly in $L^1(0,T;L^p(\Omega))$ for $1\le p< 3$,}\\
&\nabla c_j^{(n)}\rightharpoonup \nabla c_j \text{ weakly in $L^2(0,T;L^1(\Omega))\cap L^1(0,T;L^q(\Omega))$, for $1\le q< 3/2$,}\\
&\nabla \Phi^{(n)}\rightharpoonup \nabla \Phi \text{ weakly-* in $L^\infty(0,T; L^2(\Omega))$,}\\
&\Phi^{(n)}\to \Phi \text{ strongly in $C([0,T];L^p(\Omega))$ for $1\le p < 6$.}\\
\end{align*}
Moreover, there are $r_1,r_2  > 1$ such that
\begin{align*}
&c_j^{(n)}\nabla\Phi^{(n)}\rightharpoonup c_j\nabla\Phi \text{ weakly in $L^{r_1}((0,T)\times\Omega)$,}\\
&c_j^{(n)}u_n\rightharpoonup c_j u \text{ weakly in $L^{r_2}((0,T)\times\Omega)$,}
\end{align*}
and the limit functions $u$, $(c_1,...,c_N,\Phi)$ are a weak solution of \eqref{e1.3}, \eqref{e1.5}, \eqref{e1.11}-\eqref{e1.13}.
\end{theorem}

\begin{remark}
We remark that the weak sequential stability of the solutions to the PNP subsystem in Theorem \ref{T1.2} is very useful not only in the proof of Theorem \ref{T1.1} but also in showing the large-time behavior of weak solutions and the incompressible limits. 
\end{remark}

Our third  result concerns the large-time behavior of the finite energy weak solutions of the PNPNS system as an application of the weak sequential stability in Theorem \ref{T1.2}.

\begin{theorem}[Large-time behavior]\label{T1.ltb}
Let $(\rho, u, c_1,...,c_N,\Phi)$ be a finite energy weak solution of \eqref{e1.1}-\eqref{e1.13}. Then, there is a stationary state of density $\rho_s$, which is a positive constant, a stationary velocity field $u_s=0$ and a stationary state of ion densities $c_j^{(s)}$, $j=1,...,n$, with a corresponding stationary self-consistent potential $\Phi^{(s)}$, such that, as $t\to\infty$,
\begin{equation}
\begin{cases}
\rho(t)\to \rho_s \text{ strongly in $L^\gamma(\Omega)$},\\
u(t)\to 0 \text{ strongly in $L^2(\Omega)$},\\
c_j(t)\to c_j^{(s)} \text{ strongly in $L^1(\Omega)$},\\ 
\Phi(t)\to \Phi^{(s)} \text{ strongly in $H^1(\Omega)$}.
\end{cases}
\end{equation}
Moreover, the limiting ion densities are given by
\[
c_j^{(s)}(x)=I_j^{(0)}\frac{e^{-z_j\Phi^{(s)}(x)}}{\int_\Omega e^{-z_j\Phi^{(s)}(x)} dx},
\]
where the constants $I_j^{(0)}$ are the initial masses
\[
I_j^0=\int_\Omega c_j^{(0)}(x) dx,
\]
and the potential $\Phi^{(s)}$ is the unique solution to the following Poisson-Boltzmann equation:
\begin{equation}
-\kappa\Delta \Phi^{(s)}=\sum_{j=1}^Nz_j I_j^{(0)}\frac{e^{-z_j\Phi^{(s)}}}{\int_\Omega e^{-z_j\Phi^{(s)}}(x) dx},
\end{equation}
with
\begin{equation}
(\partial_\nu\Phi^{(s)}+\tau\Phi^{(s)})|_{\partial\Omega}=V.
\end{equation}
\end{theorem}

As another application of the weak sequential stability in Theorem \ref{T1.2}, we also justify the incompressible limit of the PNPNS system in Theorem \ref{T1.3}, which will be discussed in detail in Section 6.

\subsection{Difficulties and strategy of proofs}
Although the Poisson-Nernst-Planck equations are a semilinear parabolic system, the natural energy estimates only provide time uniform boundedness of the ion concentrations in $L^1(\Omega) \ln(L^1(\Omega))$. Moreover, the weak solutions for the ionic concentrations are non-negative, but may vanish and the are not necessarily (uniformly) bounded. Then, it is not clear from the energy dissipation that (e.g) the $H^1(\Omega)$ norm of the ion concentrations is finite or even (square) integrable with respect to time. Thus, the existence of global weak solutions is not straightforward. In addition, the strong coupling with the Navier-Stokes equations also adds additional difficulties. Indeed, the flux of ion particles is convected by the fluid motion. Furthermore, the forcing terms in the momentum equation depend explicitly on the ion concentrations, and the energy inequality also does not provide a straightforward bound for them, due to the lack of boundedness form above and away from zero of the ion concentrations.  

Compared to the incompressible version of the equations, we also face additional challenges, due to the possibility of vacuum. Indeed, in the compressible case, we do not expect that the velocity field is bounded in $L^2(\Omega)$, uniformly in time. Thus, the convective term in the Nersnt-Planck equation is delicate and the estimates to ensure stability of the solutions are complicated. Moreover, the analysis of the pressure $p(\rho)$, which is key in the theory of weak solutions for the compressible Navier-Stokes equations, is sensitive to the behavior of the forcing terms in the momentum equation. 

It is also worth mentioning that the boundary conditions, which are motivated by physical considerations, also play an important role in the energy estimates and make the stability analysis more intricate.

In order to overcome these difficulties we show that the solutions of the Poisson-Nerst-Placnk subsystem are stable with respect to a given (compressible) velocity field in the sense of Theorem~\ref{T1.2}. Let us point out that the hypotheses \eqref{e1.est} and \eqref{e1.estu} are compatible with the natural energy estimates \eqref{e1.energy} and \eqref{e1.T_Phi} (which is a consequence of  \eqref{e1.energy}). 

The weak sequential stability of the PNP subsystem allows for the adaptation of the established theory for the Navier-Stokes equations contained in \cite{L2,F,FKP}. Indeed, based on the {\it a priori} estimates provided by the energy inequality, it follows that, up to a subsequence, the approximate solutions converge to the limit PNPNS system with $p(\rho)$ replaced by a weak limit of the sequence of pressures, denoted by $\overline{p(\rho)}$. Then, the key point is to show that the fluid's densities converge strongly, so that $\overline{p(\rho)}=p(\rho)$. As in the case of the Navier-Stokes equations, the convergence and consistency of a regularized system of equations is enabled by the weak continuity of the effective viscous flux $p-(\lambda+2\mu)\div u$ (see also \cite{S,H1,H2}), which guarantees the stability of renormalized solutions of the continuity equation and which, in turn, provide the strong convergence of the densities. With Theorem~\ref{T1.2} at hand, the proof of the weak continuity of the effective viscous flux follows almost directly from the arguments in \cite{F}, once we realize that the forcing terms in the continuity equation, 
 depending explicitly on the ion densities and  the self-consistent potential, converge nicely to their corresponding limits.

The regularized system that we are considering consists of  an artificial viscosity in the continuity equation and an artificial pressure term in the momentum equation, as in \cite{FNP}. We also regularize the Poisson equation \eqref{e1.5}, motivated by a similar approach adopted in \cite{FS} in the incompressible case of the equations. Next, we solve the regularized system through a Faedo-Galerkin scheme, where a thorough analysis of the PNP subsystem in terms of a given velocity field is performed. More precisely, we first assume that the velocity field $u$ is given and solve the regularized continuity equation, as well as the Poisson-Nernst-Planck subsystem in terms of $u$. Then, we substitute these solutions into a Galerkin approximation of the momentum equation and solve it locally in time via Schauder's fixed point theorem. After this, we prove several {\it a priori} estimates independent of the regularizing parameters, based on the dissipative structure of the system, which serve the purpose of extending the local approximate solutions to global ones and also to take the limit along a subsequence in order to show convergence of the Galerkin scheme. 

The analysis of the convergence of the scheme requires attention. Indeed, the energy estimates for the ion densities are not suitable to perform usual weak convergence arguments, based on Sobolev embeddings. Instead, motivated by some arguments in \cite{FS} on the incompressible PNPNS equations, we find that it is better to work with the square root of the ion densities for which it is possible to deduce $L^\infty(0,T;L^2(\Omega))$ and $L^2(0,T;\dot{H}^1(\Omega))$ estimates. The first one of them is provided by the conservation of the ion masses, related to equation \eqref{e1.3}, which, together with no-slip boundary condition \eqref{e1.10} and the blocking boundary conditions \eqref{e1.11}, implies that the $L^1$ norm of the ion densities is preserved in time. The second estimate corresponds to \eqref{e1.T_Phi}, which follows from the energy inequality. Then, we are able to show strong convergence of the square roots of the approximate ion densities and, consequently, on the densities. This is, roughly, the key to the proof of Theorem~\ref{T1.2}.

Theorem~\ref{T1.ltb} essentially states that the weak solutions of the PNPNS system converge to a stationary solution, which is uniquely determined by the initial data. We mention that an analogous result was obtained by Constantin and Ignatova \cite{CI1} in the two-dimensional and incompressible case for the strong solutions to the boundary value problems of the PNPNS system. The proof of the result of large-time behavior in \cite{CI1} relies on the boundedness of the solutions in higher-order norms, which we do not expect to control for weak solutions. Thus, we introduce a different approach, inspired by the result on large time behavior for weak solutions of the compressible Navier-Stokes equations by Feireisl and Petzeltov\'{a} \cite{FP}. In other words, we arrive at the same result, assuming only the minimum regularity of the solutions provided by the energy. The proof of Theorem~\ref{T1.ltb} also uses the weak sequential stability of weak solutions from Theorem~\ref{T1.2}.

As a final comment, we point that Theorem~\ref{T1.2} is general and may be applied in other contexts. To illustrate this, we include one extra Section (Section~\ref{S5} below) where we discuss further results on the PNPNS system that can be deduced as applications of Theorem~\ref{T1.2}, namely, the incompressible limit of the PNPNS equations proved in Theorem \ref{T1.3}.


\subsection{Organization of the paper}
The rest of the paper is organized as follows. In Section~\ref{S:2} we analyze the dissipative structure of the system by deducing a general energy equation, independent of any boundary conditions. In Section~3 we introduce and prove the existence of solutions to a regularized PNPNS system. In the process we deduce the weak sequential stability of solutions to the PNP subsystem, i.e., Theorem~\ref{T1.2}. More precisely, in Subsection~\ref{S3.5} we prove a version of Theorem~\ref{T1.2} for the regularized PNP subsystem and whose proof contains the case of the original PNP subsystem. In Section~\ref{S4} we consider the limit as the regularizing parameters vanish in order to find a solution of the original PNPNS system in the limit, which completes the proof of Theorem~\ref{T1.1}. In Section~\ref{S50} we discuss the long time behavior of the weak solutions provided by Theorem~\ref{T1.1} and in Section~\ref{S5} we discuss the incompressible limit of the equations, as applications of Theorem~\ref{T1.2}. Lastly, we include an Appendix where we discuss the energy dissipation subject to another physically meaningful set of boundary conditions.



\section{Energy Equation}\label{S:2}

As aforementioned, the system \eqref{e1.1}-\eqref{e1.5} has been derived in \cite{WLT1,WLT2}  for the case of two ionic species by an energetic variational approach in the spatial domain $\mathbb{R}^3$, 
and was shown to have a dissipative structure under certain far-field conditions. In this section, we discuss the dissipative structure of the above initial-boundary value problem by (formally) deriving an energy identity for the model.


Considering the term $-\sum\nabla\varphi_j(c_j) + \kappa\nabla\Phi\Delta\Phi$ in the momentum equation as an external force,  we can multiply equation \eqref{e1.2} by $u$ and perform standard calculations, using also the continuity equation \eqref{e1.1}, to obtain
\begin{multline}\label{e.u}
\partial_t\left( \rho\left(\frac{1}{2} |u|^2 + \frak{e}(\rho)\right) \right) + \mu|\nabla u|^2 + (\lambda+\mu)(\div u)^2\\
+\div\left( u\left(\rho\left(\frac{1}{2} |u|^2 + \frak{e}(\rho)\right) + p(\rho) \right)\right) + \div(\mathbb{S}\cdot u) \\
= u\cdot\left(-\sum_{j=1}^N\nabla\varphi_j(c_j)+ \kappa\nabla\Phi\Delta\Phi\right),
\end{multline}
where $\frak{e}(\rho)$ is the internal energy given by
\begin{equation}\label{e.imte}
\frak{e}:=\int^\rho \frac{p(s)}{s^2}ds.
\end{equation}
Note that, from equation \eqref{e1.5}, we have that $u\cdot\kappa\nabla\Phi\Delta\Phi=-u\cdot\nabla\Phi \sum z_j c_j$. 

Next, we take a function $\sigma_j$, called entropy density, that solves the equation 
\begin{equation}\label{e.entropyv}
s\sigma_j'(s)-\sigma_j(s)=\varphi_j(s),\quad s>0,
\end{equation}
and multiply equation \eqref{e1.3} by $\sigma_j'(c_j)+z_j\Phi$ to obtain
\begin{align*}
\partial_t\sigma_j(c_j)+ z_j\partial_tc_j \Phi &= -\div(c_ju)(\sigma_j'(c_j)+z_j\Phi) + \div\left( D_j \Big(\nabla\varphi_j(c_j)+z_jc_j\nabla\Phi\Big) \right)(\sigma_j'(c_j)+z_j\Phi)\\
&:= A+B.
\end{align*}
Noting that \eqref{e.entropyv} implies $\varphi_j'(s)=s\sigma_j''(s)$, we see that $\nabla\varphi_j(c_j)=c_j\nabla\sigma_j'(c_j)$. From this identity, we have that
\begin{align*}
B&=\div \left(D_j c_j\nabla(\sigma_j'(c_j)+z_j\Phi)\right)(\sigma_j'(c_j)+z_j\Phi)\\
 &=\div\Big(D_j c_j(\sigma_j'(c_j)+z_j\Phi) \nabla(\sigma_j'(c_j)+z_j\Phi) \Big) - D_jc_j|\nabla (\sigma_j'(c_j)+z_j\Phi)|^2.
\end{align*}
We also see that
\begin{align*}
A&= -\div\Big( c_j(\sigma_j'(c_j)+z_j\Phi)u\Big) + c_j\nabla(\sigma_j'(c_j)+z_j\Phi)\cdot u\\
 &= -\div\Big( c_j(\sigma_j'(c_j)+z_j\Phi)u\Big) + \nabla\varphi_j(c_j)\cdot u +z_jc_j\nabla\Phi\cdot u.
\end{align*}

Thus we, get that
\begin{multline}\label{e.vt2}
\partial_t\sigma_j(c_j)+ z_j\partial_tc_j \Phi = -\div\Big( c_j(\sigma_j'(c_j)+ z_j\Phi)u\Big) + \nabla\varphi_j(c_j)\cdot u + z_jc_j\nabla\Phi\cdot u\\
  +  \div\Big(D_j c_j(\sigma_j'(c_j)+ z_j\Phi) \nabla(\sigma_j'(c_j)+ z_j\Phi) \Big) - D_jc_j|\nabla (\sigma_j'(c_j)+ z_j\Phi)|^2.
\end{multline}


Summing in $j$  and recalling \eqref{e1.5} yields
\begin{multline}\label{e.vw3}
\partial_t\Big(\sum_{j=1}^N\sigma_j(c_j)\Big) - \kappa\Delta \Phi_t \Phi + \sum_{j=1}^N D_j c_j|\nabla (\sigma_j'(c_j)+ z_j\Phi)|^2\\
 =-\div\Big( \sum_{j=1}^Nc_j(\sigma_j'(c_j)+ z_j\Phi)u\Big) +u\cdot\sum_{j=1}^N\nabla\varphi_j(c_j)   +\Big(\sum_{j=1}^N z_j c_j\Big)\nabla\Phi\cdot u  \\
  +  \div\Big(\sum_{j=1}^N D_jc_j(\sigma_j'(c_j)+z_j\Phi) \nabla(\sigma_j'(c_j)+z_j\Phi) \Big).
\end{multline}

Then, writing $$-\kappa\Delta \Phi_t \Phi = \frac{\kappa}{2}(|\nabla\Phi|^2)_t - \kappa\div(\Phi\nabla\Phi_{t}),$$ we have
\begin{multline}\label{e.vw}
\partial_t\Big(\sum_{j=1}^N\sigma_j(c_j)\Big)+\frac{\kappa}{2}\partial_t(|\nabla\Phi|^2) + \sum_{j=1}^ND_jc_j|\nabla (\sigma_j'(c_j)+ z_j\Phi)|^2\\
 =\div(\kappa\Phi\nabla\Phi_{t})-\div\Big( \sum_{j=1}^Nc_j(\sigma_j'(c_j)+ z_j\Phi)u\Big) +u\cdot\sum_{j=1}^N\nabla\varphi_j(c_j)   +\Big(\sum_{j=1}^N z_j c_j\Big)\nabla\Phi\cdot u  \\
  +  \div\Big(\sum_{j=1}^N D_jc_j(\sigma_j'(c_j)+z_j\Phi) \nabla(\sigma_j'(c_j)+z_j\Phi) \Big).
\end{multline}

Now, we add equations \eqref{e.u} and \eqref{e.vw} to obtain
\begin{align}\label{e.energy0}
&\frac{d}{dt}\left( \rho\left(\frac{1}{2} |u|^2 + \frak{e}(\rho)\right) + \sum_{j=1}^N\sigma_j(c_j) + \frac{\kappa}{2}|\nabla\Phi|^2 \right)\\
&\quad + \mu|\nabla u|^2 + (\lambda+\mu)(\div u)^2 + \sum_{j=1}^ND_jc_j|\nabla (\sigma_j'(c_j)+ z_j\Phi)|^2\nonumber\\
&= -\div\left( u\left(\rho\left(\frac{1}{2} |u|^2 + \frak{e}(\rho)\right) + p(\rho)\right) \right) - \div(\mathbb{S}\cdot u)+ \div(\kappa\Phi\nabla\Phi_{t})\nonumber \\
&\quad -\div\Big( \sum_{j=1}^Nc_j(\sigma_j'(c_j)+ z_j\Phi)u\Big) +  \div\Big(\sum_{j=1}^N D_jc_j(\sigma_j'(c_j)+z_j\Phi) \nabla(\sigma_j'(c_j)+z_j\Phi) \Big).\nonumber\nonumber
\end{align}

To conclude, we use, once again, the identities $c_j\nabla\sigma_j'(c_j)=\nabla\varphi_j(c_j)$ in order to rewrite the last term on the right-hand-side of \eqref{e.energy0} to finally obtain the following energy equation
\begin{align}\label{e.energydif}
&\frac{d}{dt}\left( \rho\left(\frac{1}{2} |u|^2 + \frak{e}(\rho)\right) + \sum_{j=1}^N\sigma_j(c_j) + \frac{\kappa}{2}|\nabla\Phi|^2 \right)\\
&\quad + \mu|\nabla u|^2 + (\lambda+\mu)(\div u)^2 + \sum_{j=1}^ND_jc_j|\nabla (\sigma_j'(c_j)+ z_j\Phi)|^2\nonumber\\
&= -\div\left( u\left(\rho\left(\frac{1}{2} |u|^2 + \frak{e}(\rho)\right) + p(\rho)\right) \right) - \div(\mathbb{S}\cdot u)+ \div(\kappa\Phi\nabla\Phi_{t})\nonumber \\
&\quad -\div\Big( \sum_{j=1}^Nc_j(\sigma_j'(c_j)+ z_j\Phi)u\Big) +  \div\Big(\sum_{j=1}^N D_j(\sigma_j'(c_j)+z_j\Phi) (\nabla\varphi_j(c_j)+z_jc_j\nabla\Phi) \Big).\nonumber\nonumber
\end{align}

We emphasize that up to this point we have not used the boundary conditions. However, taking them  into account, we can integrate \eqref{e.energydif} in order to obtain an integral energy identity. Indeed, except for the term $\div(\kappa\Phi\nabla\Phi_{t})$, all of the terms on the right-hand-side of \eqref{e.energydif} vanish upon integration over $\Omega$,  in accordance to the no slip boundary condition \eqref{e1.10} and the blocking boundary conditions \eqref{e1.11}. In order to deal with the term $\div(\kappa\Phi\nabla\Phi_{t})$, we take into account condition \eqref{e1.13} and notice that, since $V$ does not depend on t, we have 
\begin{align*}
\int_\Omega \div(\kappa\Phi\nabla\Phi_t)dx &=  \kappa\int_{\partial\Omega} \Phi\partial_\nu\Phi_t\, dS\\
    &=\kappa\tau\int_{\partial\Omega} \Phi\Phi_t\, dS\\
    &=\frac{\kappa\tau}{2}\frac{d}{dt}\int_{\partial\Omega}\Phi^2dS.
\end{align*}
Thus, we conclude that
\begin{equation}\label{e.energyint}
\frac{d}{dt}E(t) + \int_\Omega \Big(\mu|\nabla u|^2 + (\lambda+\mu)(\div u)^2 + \sum_{j=1}^ND_jc_j|\nabla (\sigma_j'(c_j)+ z_j\Phi)|^2\Big)dx = 0,
\end{equation}
where,
\begin{equation}
E(t)=\int_\Omega\left( \rho\left(\frac{1}{2} |u|^2 + \frak{e}(\rho)\right) + \sum_{j=1}^N\sigma_j(c_j) + \frac{\kappa}{2}|\nabla\Phi|^2  \right)dx+\frac{\kappa\tau}{2}\int_{\partial\Omega}\Phi^2dS.
\end{equation}

%

\begin{remark}\label{R.1}
Usually, the entropy densities $\sigma_j(c_j)$ are $\sigma_j(c_j)=c_j\ln c_j-c_j+1$, corresponding to $\varphi_j(c_j)=(c_j-1)$, according to \eqref{e.entropyv}. Indeed, as already mentioned, this is in accordance to Fick's law of diffusion by which the flux of ions should go from regions of high concentration to regions of low concentrations, that is, in a direction proportional (and opposite) to the gradient, resulting in the choice (up to an additive constant) of $\varphi_j(c_j)=c_j$ in the equation \eqref{e1.3}. In this case, we note that $\sigma_j(c_j)$ is not actually differentiable at $c_j=0$. Therefore, we cannot just multiply equation \eqref{e1.3} by $\sigma_j'(c_j)+z_j\Phi$, as we did above. So, in order to obtain an estimate like \eqref{e.energyint} rigorously, we would have to replace $c_j\ln(c_j)$ by $(c_j+\delta)\ln(c_j+\delta)$ and then let $\delta\to 0$ (cf.  \cite[Proof of Lemma 3.7]{BFS}).
\end{remark}

\begin{remark}
Identity \eqref{e.energyint} yields the inherent dissipative structure of system \eqref{e1.1}-\eqref{e1.5}, under the boundary conditions \eqref{e1.10}-\eqref{e1.13}. The calculations performed so far are formal, as they were developed under the assumption of smoothness of the solutions. We will show existence of weak solutions of the system as limits of a sequence of solutions to a regularized system, where the calculations can be made rigorous. In the limit, \eqref{e.energyint} is shown to be satisfied as an inequality by the weak solutions.
\end{remark}

\begin{remark}
We point out that we have not used the constitutive relation \eqref{e1.7} on the pressure in order to deduce \eqref{e.energydif} and \eqref{e.energyint}, where, as is usual for the fluid equations, the contribution of the pressure to the energy is accounted by the internal energy given by \eqref{e.imte}.
\end{remark}

\begin{remark}
Note that we did not use the assumption that the diffusion coefficients $D_j$ are constant in the calculations above. Thus, the energy equation deduced is still valid in the case of nonconstant diffusion coefficients. 
\end{remark}

To finish this section, let us point out that other boundary conditions for the model might be shown to provide good estimates starting from \eqref{e.energydif}, by adapting the arguments above accordingly. See the Appendix for an example.





\section{Approximation Scheme}\label{S3}

We now move on to the prove of Theorem~\ref{T1.1}. In this section and in what follows we take $\varphi_j(c_j)=c_j$, so that $\sigma_j(c_j)=c_j\ln c_j-c_j+1$ (cf. Remark~\ref{R.1} above). We also assume, without loss of generality, that $\kappa=D_j=1$.

We will first prove the existence of solutions to a regularization of system \eqref{e1.1}-\eqref{e1.5} with initial and boundary conditions \eqref{e1.9}-\eqref{e1.13}. First, we write $\Phi(t,x)=\Phi_1(x)+\Phi_2(t,x)$ where $\Phi_1$ is the unique (smooth) solution to
\begin{equation}\label{e3.0}
\begin{cases}
-\Delta \Phi_1 = 0,&\text{in $\Omega$},\\
\partial_\nu \Phi_1 + \tau\Phi_1 = V,&\text{on $\partial\Omega$}.\\
\end{cases}
\end{equation}
Then, for the given small positive constants $\vartheta$ and $\delta$,  we consider the regularized system:  
\begin{align}
&\partial_t\rho+\div(\rho u)=\vartheta \Delta\rho,\label{e3.1}\\ 
&\partial_t(\rho u) + \div(\rho u\otimes u)+\nabla (a\rho^\gamma + \delta \rho^\beta) + \vartheta\nabla u\cdot\nabla \rho= \div\mathbb{S}-\sum_{j=1}^N\nabla c_j-\sum_{j=1}^N z_j c_j \nabla\Phi,\label{e3.2}\\
&\partial_tc_j+\div(c_j(u-z_j\Phi_1))=\div\left( \nabla c_j+z_jc_j\nabla\Phi_2 \right),\label{e3.3}\\[4mm]
&-\Delta\Phi_2=\Psi ,\label{e3.5}\\
&(1-\vartheta\Delta)\Psi = \sum_{j=1}^Nz_j c_j.\label{e3.5'}
\end{align}  
Here, the of artificial viscosity on the right hand side of \eqref{e3.1} is intended to regularize the continuity equation \eqref{e1.1}. The introduction of this term causes an unbalance in the energy of the system which is equated by the term $\vartheta \nabla u\cdot\nabla\rho$ in the regularized momentum equation \eqref{e3.2}. Moreover, the term $\delta \rho^\beta$ with $\beta>1$ large enough (but fixed) acts as an artificial pressure and provides better estimates on the density. The introduction of these regularizing terms is motivated by the analogues in \cite{FNP}.

Note also that we introduced a new variable, namely $\Psi$, and added an extra equation in the system. We point out that equations \eqref{e3.5} and \eqref{e3.5'} are equivalent to
\begin{equation*}
-\Delta\Phi_2 + \vartheta\Delta^2\Phi_2 = \sum_{j=1}^Nz_j c_j,
\end{equation*}
which, together with \eqref{e3.0}, regularizes and formally approximates \eqref{e1.5} as $\vartheta\to0$. Let us mention that the decomposition $\Phi=\Phi_1+\Phi_2$ is motivated by the linearity of the Poisson equation \eqref{e1.5}, wherein $\Phi_1$ depends only on the function $V$ (cf. \eqref{e1.13}), which is a given data of the problem.

We consider the initial-boundary value problem for the regularized system with the following initial data:
\begin{equation}\label{e3.6}
(\rho,\rho u, c_1,...,c_N)(0,x)=(\rho_0, m_0,c_1^{(0)},...,c_N^{(0)})(x), \quad x\in\Omega,
\end{equation}
and the following boundary conditions:
\begin{align}
&\partial_\nu\rho|_{\partial\Omega}=0,\label{e3.7}\\
&u|_{\partial \Omega}=0,\label{e3.8}\\
&\left(\partial_\nu c_j  - c_j\partial_\nu\Phi\right)|_{\partial \Omega} = 0,\label{e3.9}\\
&(\partial_\nu\Phi_2 + \tau\Phi_2)|_{\partial \Omega}=0,\label{e3.11}\\
&(\partial_\nu\Psi + \tau\Psi)|_{\partial \Omega}=0.\label{e3.12}
\end{align}
Note that a Newmann boundary condition has been added for the density of the fluid in accordance with the introduction of the artificial viscosity. We also added a Robin boundary condition for $\Psi$. 

Our first   result of this section concerns the solvability of the regularized system and reads as follows.

\begin{proposition}\label{P:3.1}
Let $T>0$ be given. Suppose that the initial data is smooth and satisfies 
\begin{equation}
M_0^{-1}\le \rho_0\le M_0
\end{equation}
and
\begin{equation}
0\le c_j^0\le M_0,
\end{equation}
for some positive constant $M_0>0$. Assume further that $\partial_\nu\rho_0|_{\partial \Omega}=0$ and that $\beta>\max\{4,\gamma\}$.
Then, there exists a weak solution $(\rho,u,c_j,\Phi,\Psi)$ of system \eqref{e3.1}-\eqref{e3.12}, satisfying that, for some constant $1<r<2$ independent of $\vartheta$ and $\delta$,
\begin{enumerate}
\item $\rho$ is nonnegative and
\[
\rho\in L^r(0,T;W^{2,r}(\Omega))\cap L^{\beta+1}((0,T)\times\Omega),\quad \rho_t\in L^r((0,T)\times\Omega);
\]
\item $u\in L^2(0,T;H_0^1(\Omega))$;
\item $c_j$ are nonnegative and
\[
c_j\in L^\infty(0,T;L^1(\Omega))\cap L^1(0,T;W^{1,\frac{3}{2}}(\Omega)),
\]
with $\sqrt{c_j}\in L^2(0,T;H^1(\Omega))$, $j=1,...,N$;
\item $\Phi = \Phi_1+\Phi_2$, where $\Phi_1$ is the unique solution of \eqref{e3.0} and 
$\Phi_2\in L^\infty(0,T;H^1(\Omega))\cap L^1(0,T,W^{3,\frac{3}{2}}(\Omega))\cap C([0,T];L^p(\Omega))$, for $p\in [1,6)$.
\end{enumerate}
Moreover, the solution satisfies the following energy inequality
\begin{multline}\label{e3.15}
E_{\vartheta,\delta}(t) 	+ \int_0^t\int_\Omega \Big(\mu|\nabla u|^2 + (\lambda+\mu)(\div u)^2 + \sum_{j=1}^Nc_j\Big|\frac{\nabla c_j}{c_j}+ z_j\nabla\Phi\Big|^2\Big)dx\, ds\\
+\vartheta\int_0^t\int_\Omega(a\gamma\rho^{\gamma-2}+\delta\beta\rho^{\beta-2} )|\nabla\rho|^2dx\, ds\le E_{\vartheta,\delta}(0),
\end{multline}
where,
\begin{multline}\label{e3.energy-aprox}
E_{\vartheta,\delta}(t)=\int_\Omega\Bigg( \rho\left(\frac{1}{2} |u|^2 + \frac{a}{\gamma-1}\rho^{\gamma-1} + \frac{\delta}{\beta-1}\rho^{\beta-1}\right) +\frac{\vartheta}{2} |\Delta\Phi_2|^2 \\
+ \sum_{j=1}^N(c_j\ln c_j - c_j + 1) + \frac{1}{2}|\nabla\Phi_2|^2+(\sum_{j=1}^Nz_jc_j)\Phi_1 \Bigg)dx +\frac{\tau}{2}\int_{\partial\Omega}|\Phi_2|^2dS.
\end{multline}
Furthermore, there is a positive constant $C$ independent of $\vartheta$ and $\delta$ such that
\begin{equation}\label{e3.T_Phi}
\int_0^T\int_\Omega\left( \sum_{j=1}^N\frac{|\nabla c_j|^2}{c_j}+c_j|\nabla \Phi|^2+|\Delta \Phi|^2 \right)dx\, dt\le C.
\end{equation}
\end{proposition}

We will solve this regularized system as follows. First, we solve the regularized continuity equation as well as the Poisson-Nernst-Planck subsystem in terms of the velocity field. Then, we substitute these solutions into the Galerkin approximation of the regularized momentum equation and solve it through a Faedo-Galerkin scheme, following the ideas in \cite{FNP}. In order to show the convergence of the scheme, we employ several {\it a priori} estimates based on the basic energy estimates delineated in the last section, which can be performed rigorously for the Galerkin approximations.

\subsection{The regularized continuity equation}
We begin by stating a result on the well posedness of the regularized continuity equation in terms of a given smooth velocity field. For the proof, we refer to \cite[Proposition~7.1]{F} (cf. \cite[Lemma~2.2]{FNP}).

Consider the problem
\begin{equation}\label{e3.16}
\begin{cases}
\rho_t + \div(\rho u) = \vartheta \Delta\rho, &\text{on $(0,T)\times\Omega$,}\\
\partial_\nu \rho = 0, &\text{on $\partial\Omega$, $t>0$},\\
\rho=\rho_0, &\text{on $\{t=0\}\times\Omega$},
\end{cases}
\end{equation}
where $u$ is a given velocity field.

\begin{lemma}\label{L:3.2}
Let $\rho_0\in C^{2+\zeta}(\Omega)$, $\zeta>0$ and $u\in C([0,T];C_0^2(\overline{\Omega}))$ be given. Assume further that $\partial_\nu\rho_0=0$.
Then, problem \eqref{e3.16} has a unique classical solution $\rho$ such that
\begin{equation}
\rho_t\in C([0,T];C^\zeta(\overline{\Omega})),\quad \rho\in C([0.T];C^{2+\zeta}(\Omega)).
\end{equation}
Moreover, suppose that the initial density is positive and let  $u\mapsto\rho[u]$ be the solution mapping which assigns to any $u\in C([0,T];C_0^2(\overline{\Omega}))$ the unique solution $\rho$ of \eqref{e3.16}.
Then, this mapping takes bounded sets in the space $C([0,T];C_0^2(\overline{\Omega}))$ into bounded sets in the space 
\[
Y:=\{\partial_t\rho \in C([0,T];C^\zeta(\overline{\Omega}):\rho\in C([0,T];C^\zeta(\overline{\Omega}))  \},
\]
and
\[
u\in C([0,T];C_0^2(\overline{\Omega})) \mapsto \rho[u]\in C^1([0,T]\times\overline{\Omega})
\]
is continuous.
\end{lemma}

\subsection{The regularized Poisson-Nersnt-Planck system}

Next, we consider the Poisson-Nersnt-Planck system. Namely, we are going to solve for $c_j$, $j=1,...,N$,  and $\Phi$ in the equations \eqref{e3.3}-\eqref{e3.5} in terms of a given velocity field $u$. More precisely, suppose that $u$ is a given smooth velocity field, which vanishes in the boundary of $\Omega$ (in accordance with the no-slip boundary condition \eqref{e3.8}), and consider the following problem:
\begin{align}
&\begin{cases}\label{e3.NP}
\partial_tc_j+\div(c_ju)=\div\left( \nabla c_j+z_jc_j\nabla\Phi \right), &\text{on $(0,T)\times\Omega$,\, $j=1,...,N$},\\
\partial_\nu c_j  - c_j\partial_\nu\Phi = 0, & \text{on $\partial\Omega$, $t>0$},\\
c_j=c_j^0 & \text{on $\{t=0\}\times\Omega$,\, $j=1,...,N$, }
\end{cases}\\
&\begin{cases}\label{e3.P}
-\Delta\Phi_2=\Psi,\qquad\qquad\qquad\qquad\quad\quad\qquad &\text{on $(0,T)\times\Omega$},\\
(1-\vartheta\Delta)\Psi = \sum_{j=1}^Nz_j c_j &\text{on $(0,T)\times\Omega$},\\
(\partial_\nu\Phi_2 + \tau\Phi_2)|_{\partial \Omega}=0,& \text{on $\partial\Omega$, $t>0$},\\
(\partial_\nu\Psi + \tau\Psi)|_{\partial \Omega}=0& \text{on $\partial\Omega$, $t>0$}.
\end{cases}
\end{align} 

Here, as before, $\Phi=\Phi_1+\Phi_2$, where $\Phi_1$ is the solution of \eqref{e3.0}.

\begin{lemma}\label{l3.PNP[u]}
Let $T>0$ be given and suppose that $u\in C([0,T];C_0^2(\overline{\Omega}))$. Suppose also that $c_j^0$ is nonnegative and bounded, $j=1,...,N$. Then, there is a unique solution $(c,\Phi)$, $c=(c_1,...,c_N)$, of \eqref{e3.NP}-\eqref{e3.P} such that 
\begin{enumerate}
\item $c_j\in  L^2(0,T;H^1(\Omega))$ with $\partial_t c_j\in L^2(0,T;H^{-1}(\Omega))$,
\item $\Phi\in C([0,T];H^4(\Omega))$ is a strong solution of \eqref{e3.P},
\item \eqref{e3.NP} is satisfied in the sense that for any $\eta\in C^\infty([0,T]\times\overline{\Omega})$ such that $\eta(T,\cdot)=0$,
\begin{equation}\label{e3.NPweak}
\int_0^T\int_\Omega \Big(-c_j\eta_t +( -c_ju + \nabla c_j + z_jc_j\nabla\Phi)\cdot\nabla\eta \Big)dx\, ds = \int_\Omega c_j^0 \eta(0)dx.
\end{equation}
\end{enumerate}  
Moreover, let $u\to (c,\Phi)[u]$ be the solution mapping  assigning to each $u\in C([0,T];C_0^2(\overline{\Omega}))$ the unique solution $(c,\Phi)$ of \eqref{e3.NP}-\eqref{e3.P}.
Then, this mapping takes bounded sets in $C([0,T];C_0^2(\overline{\Omega}))$ into bounded sets of $Z=Z_c^N\times Z_\Phi$, where
\[
Z_c := L^2(0,T;H^1(\Omega))\cap L^\infty(0,T;L^2(\Omega))
\]
and
\[
Z_\Phi := C([0,T];H^2).
\]
Furthermore, 
\[
u\in C([0,T];C_0^2(\overline{\Omega}))\mapsto (c,\Phi)[u]\in Z 
\]
is continuous.
\end{lemma}

\begin{proof} The proof will be divided into three steps as follows.
\smallskip

\noindent {\sl Step 1:} 
The existence of solutions follows from Lemma~4.1 in \cite{FS}, whose proof is roughly as follows: Given $\varphi_2$ smooth enough (belonging to, say, $L^\infty(0,T;W^{1,r}(\Omega))$, for some $3<r<6$), then, equation \eqref{e3.NP}, with $\varphi_2$ instead of $\Phi_2$ on the right-hand-side, is a linear parabolic problem and can be solved by standard methods (see theorem~5.1 in \cite[pg. 170]{LSU}). Once we have the unique solution for this linear equation, we solve for $\Phi_2$ in equation \eqref{e3.P} and use elliptic regularity to show that $\Phi_2\in C([0,T];H^2(\Omega))\cap L^2(0,T;H^3(\Omega))$. In this way, we obtain an operator $\varphi_2 \mapsto \Phi_2$ which can be shown to be a compact operator in $W^{1,r}(\Omega)$ due to the embedding $H^2(\Omega)\hookrightarrow W^{1,r}(\Omega)$. Then, it is possible to use Schauder's fixed point theorem in order to find a fixed point of this operator, which yields the existence of solutions of \eqref{e3.NP}-\eqref{e3.P}. We omit the details.

Once $\Phi\in L^2(0,T;H^3(\Omega))$, we see that $t\to\|\nabla\Phi(t)\|_{L^\infty(\Omega)}^2$ is integrable in $[0,T]$. Using this fact, it may be shown, as in \cite{CI1}, that $c_j(t,x)$ is nonnegative, as long as $c_j^0$ is. Indeed, taking e.g.
\[
F(s)=\begin{cases}
s^2, &\text{if $s< 0$},\\
0, &\text{if $s\ge 0$},
\end{cases}
\] 
multiplying \eqref{e3.NP} by $F'(c_j)$ and integrating, we have
\begin{align*}
\frac{d}{dt}\int_\Omega F(c_j) dx &= -\int_{\Omega}F''(c_j)|\nabla c_j|^2 dx + \int_\Omega F''(c_j)\, c_j\, (u-z_j\nabla\Phi)\cdot\nabla c_j dx \\
&\le -\frac{1}{2}\int_{\Omega}F''(c_j)|\nabla c_j|^2 dx + C\left(\|u\|_{L^\infty(\Omega)}^2+\|\nabla \Phi\|_{L^\infty(\Omega)}^2\right)\int_\Omega F(c_j) dx\\
&\le  C\left(\|u\|_{L^\infty(\Omega)}^2+\|\nabla \Phi\|_{L^\infty(\Omega)}^2\right)\int_\Omega F(c_j) dx.
\end{align*}
Thus, since $F(c_j^0)\equiv 0$, Gronwall's inequality implies that $F(c_j)\equiv 0$, which means that $c_j\ge 0$. 
\smallskip

\noindent {\sl Step 2:}
Let us now prove that the solution operator $u\mapsto (c,\Phi)[u]\in Z$ takes bounded sets of $C([0,T];C_0^2(\overline{\Omega}))$ into bounded sets of $Z$. First, we rewrite equation \eqref{e3.NP} as
\begin{equation}\label{e3.NP'}
\partial_tc_j=\div\left(c_j(z_j\nabla\Phi_1-u)+ c_j\left(\frac{\nabla c_j}{c_j}+z_j\nabla\Phi_2\right) \right).
\end{equation}
Then, inspired by the computations from Section~\ref{S:2}, we multiply \eqref{e3.NP'} by $\sigma'( c_j) + z_j\Phi_2$, where $\sigma(s)=s\ln(s)-s+1$, and integrate to obtain
\begin{align*}
\frac{d}{dt}\int_\Omega \sigma(c_j) dx &+ \int_\Omega (z_j c_j)_t\Phi_2 dx + \int_\Omega c_j\left|\frac{\nabla c_j}{c_j}+z_j\nabla\Phi_2\right|^2 dx\\
	&=\int_\Omega c_j (z_j\nabla\Phi_1-u)\cdot \left(\frac{\nabla c_j}{c_j}+z_j\nabla \Phi_2\right) dx\\
	&\le \frac{1}{2}\int_\Omega c_j\Big|\frac{\nabla c_j}{c_j}+z_j\nabla\Phi_2\Big|^2 dx + C(\|\nabla\Phi_1\|_{L^\infty}+\|u\|_{L^\infty})\int_\Omega c_j dx\\
	&=\frac{1}{2}\int_\Omega c_j\Big|\frac{\nabla c_j}{c_j}+z_j\nabla\Phi_2\Big|^2 dx + C(\|\nabla\Phi_1\|_{L^\infty}^2+\|u\|_{L^\infty}^2)\int_\Omega c_j^{0} dx,
\end{align*}
where, we used the fact that the $c_j$ is a non-negative function whose integral is preserved in time, due to the boundary conditions of the problem \eqref{e3.NP}. Here, as usual, $C$ denotes a positive universal constant which may increase from line to line. Thus,
\begin{multline}\label{e3.23}
\frac{d}{dt}\int_\Omega \sigma(c_j) dx + \int_\Omega (z_j c_j)_t\Phi_2 dx + \frac{1}{2}\int_\Omega c_j\left|\frac{\nabla c_j}{c_j}+z_j\nabla\Phi_2\right|^2 dx\\
\le C(\|\nabla\Phi_1\|_{L^\infty}^2+\|u\|_{L^\infty}^2)\int_\Omega c_j^{0} dx.
\end{multline}

Next, from \eqref{e3.P} we see that
\begin{align*}
\sum_{j=1}^N\int_\Omega (z_j c_j)_t\Phi_2 dx &= \int_\Omega \left( (1-\vartheta\Delta)\Psi \right)_t\Phi_2 dx\\
	&=\int_\Omega (-\Delta\Phi_2)_t\Phi_2 dx - \vartheta\int_\Omega (\Delta\Psi)_t\Phi_2 dx\\
	&:= I_1 + I_2.
\end{align*}
Here we note that
\begin{align*}
I_1 &= \int_\Omega \nabla(\Phi_2)_t\cdot \nabla\Phi_2 dx - \int_{\partial\Omega}(\partial_\nu \Phi_2)_t\Phi_2 dS\\
	&= \int_\Omega (\nabla\Phi_2)_t\cdot \nabla \Phi_2 dx  + \tau \int_{\partial \Omega}(\Phi_2)_t\Phi_2 dS\\
	&=\frac{d}{dt}\left[\frac{1}{2}\int_\Omega |\nabla\Phi_2|^2 dx + \frac{\tau}{2}\int_{\partial\Omega}|\Phi_2|^2 dS \right].
\end{align*}
Regarding $I_2$, according to the boundary conditions in \eqref{e3.P} we have
\begin{align*}
I_2&=\vartheta\int_\Omega\nabla\Psi_t\cdot\nabla\Phi_2 dx - \vartheta\int_{\partial\Omega}\partial_\nu\Psi_t \Phi_2 dS\\
	&=-\vartheta\int_\Omega\Psi_t\Delta \Phi_2 dx + \vartheta\int_{\partial\Omega}\Psi_t\partial_\nu\Phi_2 dx -\vartheta\int_{\partial\Omega}\partial_\nu\Psi_t \Phi_2 dS\\
	&=-\vartheta\int_\Omega\Psi_t\Delta \Phi_2 dx + \vartheta\int_{\partial\Omega}\Psi_t(-\tau\Phi_2) dx -\vartheta\int_{\partial\Omega}(-\tau\Psi)_t \Phi_2 dS\\
	&=-\vartheta\int_\Omega\Psi_t\Delta \Phi_2 dx\\
	&=\vartheta\int_\Omega (\Delta\Phi_2)_t\Delta\Phi_2 dx\\
	&=\frac{d}{dt}\left(\frac{\vartheta}{2}\int_\Omega |\Delta\Phi_2|^2dx\right).
\end{align*}
Gathering these identities in \eqref{e3.23} we obtain
\begin{multline}
\frac{d}{dt}\left[\int_\Omega \sum_{j=1}^N\sigma(c_j) dx + \frac{1}{2}\int_\Omega (|\nabla\Phi_2|^2+\vartheta|\Delta\Phi_2|^2) dx + \int_{\partial\Omega}|\Phi|^2 dS \right] \\
+ \frac{1}{2}\sum_{j=1}^N\int_\Omega c_j\left|\frac{\nabla c_j}{c_j}+z_j\nabla\Phi_2\right|^2 dx\\
 \le C(\|\nabla\Phi_1\|_{L^\infty}^2+\|u\|_{L^\infty}^2)\sum_{j=1}^N\int_\Omega c_j^{0} dx.
\end{multline}
In particular, we have that $\Phi\in L^\infty(0,T;H^2(\Omega))$ and $$\sup_{0\le t\le T}\int_\Omega (|\nabla\Phi_2(t,x)|^2+\vartheta|\Delta\Phi_2(t,x)|^2 )dx$$ is bounded by a constant which depends only on $V$, $\|u\|_{C([0,T]\times\overline{\Omega})}$ and the initial data. Consequently, we also have that $\sup_{0\le t\le T}\|\nabla\Phi(t)\|_{L^r(\Omega)}$, $1\le r\le 6$ is also bounded by a constant which depends only on $V$, $\|u\|_{C([0,T]\times\overline{\Omega})}$, the initial data and $\vartheta$.

Now we multiply \eqref{e3.NP} by $c_j$ and integrate to obtain
\begin{align*}
\frac{d}{dt}\left(\frac{1}{2}\int_\Omega c_j^2 dx\right) + \int_\Omega |\nabla c_j|^2 dx
	&= \int_\Omega c_j (u-z_j\nabla\Phi)\cdot \nabla c_j dx\\
	&\le C\|c_j(t)\|_{L^4(\Omega)}(\|u\|_{L^4(\Omega)}+\|\nabla\Phi\|_{L^4(\Omega)})\|\nabla c_j\|_{L^2(\Omega)}.
\end{align*}
Here we use the Sobolev inequality for a given positive $\varepsilon$ to estimate 
$$\|c_j(t)\|_{L^4(\Omega)}\le C_\varepsilon \|c_j(t)\|_{L^2(\Omega)}+\varepsilon\|\nabla c_j\|_{L^2(\Omega)}$$ for some universal constant $C_\varepsilon$ (which holds true due to the compactness of the embedding $H^1(\Omega)\hookrightarrow L^4(\Omega)$). Thus, choosing $\varepsilon$ small enough, we obtain
\begin{align*}
\frac{d}{dt}\left(\frac{1}{2}\int_\Omega c_j^2 dx\right) &+ \int_\Omega |\nabla c_j|^2 dx\\
&\le C(\|u\|_{L^4(\Omega)}+\|\nabla\Phi\|_{L^4(\Omega)})^2\|c_j(t)\|_{L^2(\Omega)}^2 +  \frac{1}{2}\|\nabla c_j\|_{L^2(\Omega)},
\end{align*}
which means that
\[
\frac{d}{dt}\left(\frac{1}{2}\int_\Omega c_j^2 dx\right) + \frac{1}{2}\int_\Omega |\nabla c_j|^2 dx\le C(\|u\|_{L^4(\Omega)}+\|\nabla\Phi\|_{L^4(\Omega)})^2\|c_j(t)\|_{L^2(\Omega)}^2,
\]
and, therefore, Gronwall's inequality yields the uniform bounds on $c_1$,...,$c_N$ in $Z_c$.

Moreover, using the equation \eqref{e3.NP}, we see that $\partial_t c_j \in L^2(0,T,H^{-1}(\Omega))$, which implies that $c_j\in C([0,T];L^2(\Omega))$. This also implies the continuity of $t\mapsto \Phi(t)\in H^4(\Omega)$, by virtue of equation \eqref{e3.P}.
\smallskip

\noindent {\sl Step 3:}
Let us now prove the continuity of the solution operator $$u\in C([0,T];C(\overline{\Omega})) \mapsto (c,\Phi)\in Z.$$
Let $(u_n)_n$ be a sequence such that $u_n\to u$ in $C([0,T];C_0^2(\overline{\Omega}))$ as $n\to\infty$. Denote $(c^{(n)},\Phi^{(n)})=(c,\Phi)[u]$ and $(c,\Phi)=(c,\Phi)[u]$, where $c^{(n)}=(c_1^{(n)},...,c_N^{(n)})$ and $c=(c_1,...,c_N)$. Then, taking the difference of equations \eqref{e3.NP} for $c_j^{(n)}$ and for $c_j$, multiplying the resulting equation by $c_j^{(n)}-c_j$ and integrating over $\Omega$, we obtain
\begin{align*}
&\frac{1}{2}\frac{d}{dt}\int_\Omega (c_j^{(n)}-c_j)^2 dx + \int_\Omega |\nabla(c_j^{(n)}-c_j)|^2dx \\
&\qquad= \int_\Omega \Big((c_j^{(n)}-c_j)u_n + c_j(u_n - u)\Big)\cdot\nabla(c_j^{(n)}-c_j)dx\\
&\qquad\qquad + \int_\Omega\Big(z_j(c_j^{(n)}-c_j)\nabla\Phi^{(n)} + z_j c_j\nabla(\Phi^{(n)}-\Phi)\Big)\cdot\nabla(c_j^{(n)}-c_j)\, dx\\
&\qquad\le C\|c_j^{(n)}-c_j\|_{L^4(\Omega)}\|\nabla(c_j^{(n)}-c_j)\|_{L^2(\Omega)}\left(\|u_n\|_{L^4(\Omega)}+\|\nabla\Phi^{(n)}\|_{L^4(\Omega)} \right)\\[5pt]
&\qquad\qquad + C\|c_j\|_{L^4(\Omega)}\|\nabla(c_j^{(n)}-c_j)\|_{L^2(\Omega)}\left( \|u_n-u\|_{L^4(\Omega)}+\|\nabla(\Phi^{(n)}-\Phi)\|_{L^4(\Omega)} \right)\\
&\qquad\le \frac{1}{4}\|\nabla(c_j^{(n)}-c_j)\|_{L^2(\Omega)}^2 + C\|c_j^{(n)}-c_j\|_{L^4(\Omega)}^2\\
&\qquad\qquad + C\|c_j\|_{H^1(\Omega)}^2\left( \|u_n-u\|_{L^4(\Omega)}^2+\|\nabla(\Phi^{(n)}-\Phi)\|_{L^4(\Omega)}^2 \right)
\end{align*}
Using Sobolev's inequality with $\varepsilon$ as in Step 2 above and choosing $\varepsilon$ small enough, we have that
\begin{align*}
\|c_j^{(n)}-c_j\|_{L^4(\Omega)}^2&\le  \varepsilon\|\nabla(c_j^{(n)}-c_j)\|_{L^2(\Omega)}^2 + C_\varepsilon\|c_j^{(n)}-c_j\|_{L^2(\Omega)}^2\\
&\le \frac{1}{4} \|\nabla(c_j^{(n)}-c_j)\|_{L^2(\Omega)}^2 + C\|c_j^{(n)}-c_j\|_{L^2(\Omega)}^2.
\end{align*}
Thus,
\begin{multline}\label{e3.25}
\frac{1}{2}\frac{d}{dt}\int_\Omega (c_j^{(n)}-c_j)^2 dx + \int_\Omega |\nabla(c_j^{(n)}-c_j)|^2dx \\
\le \frac{1}{2}\int_\Omega |\nabla(c_j^{(n)}-c_j)|^2dx + C\int_\Omega (c_j^{(n)}-c_j)^2 dx\\
+C\|c_j\|_{H^1(\Omega)}^2\left( \|u_n-u\|_{L^4(\Omega)}^2+\|\nabla(\Phi^{(n)}-\Phi)\|_{L^4(\Omega)}^2 \right).
\end{multline}
To conclude, we estimate $\nabla(\Phi^{(n)}-\Phi)$. To this end, we take the difference of equations \eqref{e3.P} for $\Phi_2^{(n)}$ and $\Phi_2$ to obtain the equations
\[
\begin{cases}
-\Delta (\Phi_2^{(n)}-\Phi_2)=\Psi^{(n)}-\Psi,\\
(\Psi^{(n)}-\Psi)-\vartheta\Delta(\Psi^{(n)}-\Psi)=\sum_{j=1}^N z_j(c_j^{(n)}-c_j).
\end{cases}
\]
Multiplying the second equation above by $\Phi^{(n)}-\Phi$ and integrating we get
\begin{multline}\label{e3.26}
-\int_\Omega \Delta(\Phi_2^{(n)}-\Phi_2)(\Phi_2^{(n)}-\Phi)dx - \vartheta\int_\Omega\Delta(\Psi^{(n)}-\Psi)(\Phi_2^{(n)}-\Phi) dx \\
= \int_\Omega\sum_{n=1}^N z_j(c_j^{(n)}-c_j)(\Phi_2^{(n)}-\Phi_2)dx.
\end{multline}
Here, regarding the first integral on the left-hand-side of \eqref{e3.26}, using the boundary conditions we have that
\begin{align*}
&-\int_\Omega \Delta(\Phi_2^{(n)}-\Phi_2)(\Phi_2^{(n)}-\Phi)dx\\
&\qquad = \int_\Omega \nabla(\Phi_2^{(n)}-\Phi_2)\cdot\nabla (\Phi_2^{(n)}-\Phi)dx - \int_{\partial \Omega} \partial_\nu(\Phi_2^{(n)}-\Phi_2)(\Phi_2^{(n)}-\Phi) dS\\
&\qquad = \int_\Omega |\nabla(\Phi_2^{(n)}-\Phi_2)|^2 dx + \tau\int_{\partial \Omega} |\Phi_2^{(n)}-\Phi_2|^2 dS.
\end{align*}
Regarding the second integral on the left-hand-side of \eqref{e3.26}, we have that
\begin{align*}
&- \vartheta\int_\Omega\Delta(\Psi^{(n)}-\Psi)(\Phi_2^{(n)}-\Phi) dx\\
&\qquad = \vartheta\int_\Omega \nabla(\Psi^{(n)}-\Psi)\cdot\nabla (\Phi_2^{(n)}-\Phi)dx - \vartheta\int_{\partial \Omega} \partial_\nu(\Psi^{(n)}-\Psi)(\Phi_2^{(n)}-\Phi) dS\\
&\qquad = -\vartheta\int_\Omega (\Psi^{(n)}-\Psi)\Delta(\Phi_2^{(n)}-\Phi_2) dx + \vartheta\int_{\partial \Omega} (\Psi^{(n)}-\Psi)\partial_\nu(\Phi_2^{(n)}-\Phi) dS \\
&\qquad\qquad- \vartheta\int_{\partial \Omega} \partial_\nu(\Psi^{(n)}-\Psi)(\Phi_2^{(n)}-\Phi) dS\\
&\qquad = -\vartheta\int_\Omega (\Psi^{(n)}-\Psi)\Delta(\Phi_2^{(n)}-\Phi_2) dx + \vartheta\int_{\partial \Omega} (\Psi^{(n)}-\Psi)(-\tau)(\Phi_2^{(n)}-\Phi) dS \\
&\qquad\qquad- \vartheta\int_{\partial \Omega} (-\tau)(\Psi^{(n)}-\Psi)(\Phi_2^{(n)}-\Phi) dS\\
&\qquad=-\vartheta\int_\Omega (\Phi_2^{(n)}-\Phi_2)\Delta(\Phi_2^{(n)}-\Phi_2) dx\\
&\qquad=\vartheta\int_\Omega |\Delta(\Phi_2^{(n)}-\Phi_2)|^2 dx.
\end{align*}
Gathering these two equalities in \eqref{e3.26}, and using Young's inequality with $\varepsilon$, we have
\begin{align*}
\int_\Omega &\left(|\nabla(\Phi_2^{(n)}-\Phi_2)|^2+\vartheta|\Delta(\Phi_2^{(n)}-\Phi_2)|^2\right) dx + \tau\int_{\partial \Omega} |\Phi_2^{(n)}-\Phi_2|^2 dS \\
&\qquad\qquad\qquad\qquad\qquad\qquad= \int_\Omega\sum_{n=1}^N z_j(c_j^{(n)}-c_j)(\Phi_2^{(n)}-\Phi_2)dx.\\
&\qquad\qquad\qquad\qquad\qquad\qquad\le \varepsilon \|\Phi_2^{(n)}-\Phi_2\|_{L^2(\Omega)}^2 + C_\varepsilon \sum_{k=1}^N\|c_j^{(n)}-c_j\|_{L^2(\Omega)}^2,
\end{align*}
and, since $\|\cdot\|_{L^2(\partial\Omega)}+\| \nabla \cdot\|_{L^2(\Omega)}$ is an equivalent norm on $H^1(\Omega)$, due to Poincaré's inequality with boundary term (see e.g. \cite{BL,MV}), we may choose $\varepsilon>0$ small enough to conclude that 
\begin{multline}\label{e3.27}
\int_\Omega \left(|\nabla(\Phi_2^{(n)}-\Phi_2)|^2+\vartheta|\Delta(\Phi_2^{(n)}-\Phi_2)|^2\right) dx + \tau\int_{\partial \Omega} |\Phi_2^{(n)}-\Phi_2|^2 dS \\
\le C\sum_{k=1}^N\|c_j^{(n)}-c_j\|_{L^2(\Omega)}^2.
\end{multline}
In particular, $$\|\nabla(\Phi^{(n)}-\Phi)\|_{L^4(\Omega)}^2=\|\nabla(\Phi_2^{(n)}-\Phi_2)\|_{L^4(\Omega)}^2\le C\sum_{k=1}^N\|c_j^{(n)}-c_j\|_{L^2(\Omega)}^2.$$
Substituting this last estimate in \eqref{e3.25} we see that
\begin{multline}
\frac{1}{2}\frac{d}{dt}\sum_{j=1}^N\int_\Omega (c_j^{(n)}-c_j)^2 dx + \frac{1}{2}\sum_{j=1}^N\int_\Omega |\nabla(c_j^{(n)}-c_j)|^2dx \\
\le C(1+\|c_j\|_{H^1(\Omega)}^2)\sum_{j=1}^N\int_\Omega (c_j^{(n)}-c_j)^2 dx
+C\|c_j\|_{H^1(\Omega)}^2\|u_n-u\|_{L^4(\Omega)}^2.
\end{multline}
Hence, since we already know that $\int_0^T\|c_j(t)\|_{H^1(\Omega)}^2 dt$ is bounded, we may use Gronwall's inequality to conclude that $c_j^{(n)}\to c_j$ in $Z_c$, if $u_n\to u$ in $C([0,T];C_0^2(\overline{\Omega}))$; which by \eqref{e3.27}, readily implies that $\Phi^{(n)}\to \Phi$ in $Z_\Phi$.
\end{proof}


\subsection{Galerkin approximations for the regularized system}\label{S:3.3}

We now proceed with the proof of Proposition~\ref{P:3.1}. Following the ideas from \cite{F}, with the two preliminary results above, we may apply the Faedo-Galerkin method in order to find solutions to \eqref{e3.1}-\eqref{e3.11}. 

For $n\in\mathbb{N}$, let $X_n\subseteq L^2(\Omega)$ be defined as
\[
X_n:=[{\rm span}\,\{\eta_j\}_{j=1}^n]^3,
\]
where $\{\eta_1, \eta_2,... \}$  is the complete collection of normalized eigenvectors of the Laplacian with homogeneous Dirichlet boundary condition in $\Omega$.
 For each $n\in\mathbb{N}$, we look for a function $u_n\in C([0,T];X_n)$ that satisfies \eqref{e3.2} in an approximate way. More precisely, we demand that $u_n$ satisfies
\begin{multline}\label{e3.29}
\int_\Omega \partial_t(\rho_n u_n)\cdot\eta\, dx \\
    +\int_0^t\int_\Omega\Big(\div (\rho_n u_n\otimes u_n)+\nabla(a\rho_n^\gamma + \delta\rho_n^\beta)+\vartheta\nabla u_n\cdot\nabla\rho_n\Big)\cdot\eta \, dx\, ds\\
    =\int_0^t\int_\Omega\Big( \mu\Delta u_n + (\lambda+\mu)\nabla(\div u_n)-\sum_{j=1}^N\nabla c_j^{(n)} -\sum_{j=1}^Nz_j c_j\nabla\Phi^{(n)}  \Big)\cdot \eta dx\, ds,
\end{multline}
for any $t\in[0,T]$ and $\eta\in X_n$, where $\rho_n=\rho[u_n]$, $c_j^{(n)}=c_j[u_n]$, $j=1,...,N$, and $\Phi^{(n)}=\Phi[u_n]$ are the solutions to \eqref{e3.16} and \eqref{e3.NP}-\eqref{e3.P}, associated to the velocity field $u_n$, given by Lemmas \ref{L:3.2} and \ref{l3.PNP[u]}, respectively. Moreover, 
\begin{equation}
\int_\Omega \rho_n u_n(0)\cdot \eta dx=\int_\Omega m_0\cdot\eta dx, \qquad \eta\in X_n.
\end{equation}

Now, the momentum equation in the sense of \eqref{e3.29} can be solved locally in time by means of Schauder's fixed point theorem (see e.g. section~7.2 of \cite{F}). In order to obtain global solutions, it suffices to obtain uniform bounds independent of time on the solutions, which allows the iteration of the fixed point argument to extend the local solution $u_n$ to any given time interval $[0,T]$.
Taking $\eta=u_n$ as a test function in \eqref{e3.29} and proceeding in a standard way (cf. \cite{F}) we have 
\begin{multline}\label{e3.31}
\frac{d}{dt}\int_\Omega \left( \rho_n\left(\frac{1}{2} |u_n|^2 + \frac{a}{\gamma-1}\rho_n^{\gamma-1}+\frac{\delta}{\beta-1}\rho_n^{\beta-1}\right) \right)dx \\
+ \int_\Omega\left(\mu|\nabla u_n|^2 + (\lambda+\mu)(\div u_n)^2 \right) dx +\vartheta\int_\Omega (a\gamma \rho_n^{\gamma-2}+\delta \beta \rho_n^{\beta-2} )|\nabla\rho_n|^2 dx \\
=\int_\Omega u_n\cdot\Bigg( -\sum_{j=1}^N\nabla c_j^{(n)} - \sum_{j=1}^Nz_j c_j^{(n)}\nabla\Phi^{(n)}\Bigg) dx.
\end{multline}

In order to evaluate the right hand side of \eqref{e3.31}, we multiply \eqref{e3.NP} by $\sigma'( c_j^{(n)}) + z_j\Phi^{(n)}$, where $\sigma(s)=s\ln(s)-s+1$, and integrate to obtain
\begin{multline}\label{e3.32}
\frac{d}{dt}\int_\Omega \sigma(c_j^{(n)}) dx + \int_\Omega (z_j c_j^{(n)})_t\Phi^{(n)} dx + \int_\Omega c_j^{(n)}\Big|\frac{\nabla c_j^{(n)}}{c_j^{(n)}}+z_j\nabla\Phi^{(n)}\Big|^2 dx\\
	=\int_\Omega u_n\cdot \left(\nabla c_j^{(n)}+z_jc_j^{(n)}\nabla \Phi^{(n)}\right) dx.
\end{multline}

Decomposing $\Phi^{(n)}=\Phi_1 + \Phi_2^{(n)}$, as before, where $\Phi_1$ is the solution of $\eqref{e3.0}$ we have that
\[
\int_\Omega(\sum_{j=1}^N z_j c_j^{(n)})_t\Phi^{(n)} dx = \int_\Omega(\sum_{j=1}^N z_j c_j^{(n)})_t\Phi_1 dx + \int_\Omega(\sum_{j=1}^N z_j c_j^{(n)})_t\Phi_2^{(n)} dx.
\]
Now, on the one hand, proceeding similarly as in the previous subsection, we have that
\[
\int_\Omega(\sum z_j c_j^{(n)})_t\Phi_2^{(n)} dx = \frac{1}{2}\frac{d}{dt}\left[\int_\Omega \left(  |\nabla \Phi_2^{(n)}|^2 + \vartheta|\Delta\Phi_2^{(n)}|^2 \right) dx + \tau\int_{\partial\Omega}|\Phi_2^{(n)}|^2 dS\right].
\]
On the other hand, since $\Phi_1$ is independent of $t$, we see that
\[
\int_\Omega(\sum z_j c_j^{(n)})_t\Phi_1 dx = \frac{d}{dt}\int_\Omega (\sum_{j=1}^N z_jc_j^{(n)})\Phi_1 dx.
\]
Thus, taking the sum over $j=1,...,N$ in \eqref{e3.32} and replacing the last two identities yields
\begin{multline}\label{e3.32'}
\frac{d}{dt}\Bigg[\int_\Omega \Big(\sum_{j=1}^N\sigma(c_j^{(n)}) + \frac{1}{2}|\nabla\Phi_2^{(n)}|^2+\frac{\vartheta}{2}|\Delta\Phi_2^{(n)}|^2\Big) dx + \frac{\tau}{2}\int_{\partial\Omega}|\Phi_2^{(n)}|^2 dS  \Bigg]\\
+ \int_\Omega \sum_{j=1}^N c_j^{(n)}\Big|\frac{\nabla c_j^{(n)}}{c_j^{(n)}}+z_j\nabla\Phi^{(n)}\Big|^2 dx\\
	=-\int_\Omega u_n\cdot \Bigg(\sum_{j=1}^N\nabla c_j^{(n)}+\sum_{j=1}^Nz_jc_j^{(n)}\nabla \Phi^{(n)}\Bigg) dx - \frac{d}{dt}\int_\Omega (\sum_{j=1}^N z_jc_j^{(n)})\Phi_1 dx.
\end{multline}
Finally adding the resulting equation to \eqref{e3.31} we obtain the following energy identity
\begin{multline}\label{e3.33}
\frac{d}{dt}\Bigg[\int_\Omega \left( \rho_n\left(\frac{1}{2} |u_n|^2 + \frac{a}{\gamma-1}\rho_n^{\gamma-1}+\frac{\delta}{\beta-1}\rho_n^{\beta-1}\right) \right)dx \\
+ \int_\Omega \Big( \sum_{j=1}^N\sigma(c_j^{(n)})  + \frac{1}{2}|\nabla\Phi_2^{(n)}|^2 +\frac{\vartheta}{2}|\Delta\Phi_2^{(n)}|^2  \Big)dx  + \frac{\tau}{2}\int_{\partial\Omega}|\Phi_2^{(n)}|^2 dS \Bigg]\\
+ \int_\Omega\left(\mu|\nabla u_n|^2 + (\lambda+\mu)(\div u_n)^2 \right) dx +\vartheta\int_\Omega (a\gamma \rho_n^{\gamma-2}+\delta \beta \rho_n^{\beta-2} )|\nabla\rho_n|^2 dx \\
+\int_\Omega \sum_{j=1}^Nc_j^{(n)}\Big|\frac{\nabla c_j^{(n)}}{c_j^{(n)}}+z_j\nabla\Phi^{(n)}\Big|^2 dx=-\frac{d}{dt}\int_\Omega (\sum_{j=1}^Nz_j c_j^{(n)})\Phi_1 dx.
\end{multline}

At this point, we note that 
\begin{equation}\label{e3.mass_n}
\int_\Omega c_j^{(n)}(t,x)dx=\int_\Omega c_j^0(x)dx,
\end{equation}
due to the boundary conditions in \eqref{e3.NP} and consequently
\[
\left|\int_\Omega (\sum_{j=1}^Nz_j c_j^{(n)})\Phi_1 dx\right| \le N\|\Phi_1\|_{L^\infty(\Omega)}\max_j |z_j|\, \|c_j^0\|_{L^1(\Omega)}.
\]
Hence, integrating over $t\in [0,T]$ in \eqref{e3.33} we obtain, in particular, a uniform (in $n$ and $T$) bound for $u_n$ in $L^2(0,T;H_0^1(\Omega))$, which as in \cite{F}, is enough to guarantee the global existence  of solutions to the approximate momentum equation in sense of \eqref{e3.29}.

\subsection{Energy estimates}

The convergence of the Galerkin approximations to a solution of system \eqref{e3.1}-\eqref{e3.12} will rely on a few uniform estimates, which are based on the energy identity \eqref{e3.33}. More precisely, we have the following.

\begin{lemma}\label{l3.4}
Let $(\rho_n,u_n, c_j^{(n)}, \Phi^{(n)})$ be the Galerkin approximations constructed above, for each $n\in\mathbb{N}$, and let $E_{\vartheta,\delta}^{(n)}(t)$ be given by \eqref{e3.energy-aprox} with $(\rho,u, c_j, \Phi)=(\rho_n,u_n, c_j^{(n)}, \Phi^{(n)})$, i.e.
\begin{multline*}
E_{\vartheta,\delta}^{(n)}(t)=\int_\Omega\Bigg( \rho_n\left(\frac{1}{2} |u_n|^2 + \frac{a}{\gamma-1}\rho_n^{\gamma-1} + \frac{\delta}{\beta-1}\rho_n^{\beta-1}\right) + \frac{\vartheta}{2} |\Delta\Phi_2^{(n)}|^2 	\\
+ \sum_{j=1}^N(c_j^{(n)}\ln c_j^{(n)} - c_j^{(n)} + 1) + \frac{1}{2}|\nabla\Phi_2^{(n)}|^2+(\sum_{j=1}^Nz_jc_j^{(n)})\Phi_1^{(n)} \Bigg)dx 
+\frac{\tau}{2}\int_{\partial\Omega}|\Phi_2^{(n)}|^2dS.
\end{multline*}
 Then,
\begin{multline}\label{e3.35}
E_{\vartheta,\delta}^{(n)}(t) 
+ \int_0^t\int_\Omega \Big(\mu|\nabla u_n|^2 + (\lambda+\mu)(\div u_n)^2 + \sum_{j=1}^Nc_j^{(n)}\Big|\frac{\nabla c_j^{(n)}}{c_j^{(n)}}+ z_j\nabla\Phi^{(n)}\Big|^2\Big)dx\, ds\\
+\vartheta\int_0^t\int_\Omega(a\gamma\rho_n^{\gamma-2}+\delta\beta\rho_n^{\beta-2} )|\nabla\rho_n|^2dx\, ds= E_{\delta}^{(n)}(0).
\end{multline}
Also, there is a positive constant $C$, which depends only on $E_\delta^{(n)}(0)$, but is otherwise independent of $n$, $\vartheta$ and $\delta$, such that
\begin{equation}\label{e3.36}
\int_0^T\int_\Omega\left( \sum_{j=1}^N\frac{|\nabla c_j^{(n)}|^2}{c_j^{(n)}}+c_j^{(n)}|\nabla \Phi^{(n)}|^2+|\Delta \Phi^{(n)}|^2 \right)dx\, dt\le C.
\end{equation}
Moreover,
\begin{equation}\label{e3.psi}
\vartheta\int_0^T\int_\Omega |\nabla\Delta\Phi_2^{(n)}|^2dx\, dt+\tau\vartheta\int_0^T\int_{\partial\Omega} |\Delta\Phi_2^{(n)}|^2 dS\, dt\le C.
\end{equation}
\end{lemma}

\begin{remark}\label{r3.5}
Note that \eqref{e3.36} implies that the $\sqrt{c_j^{(n)}}$, $n=1,2,...$, are bounded in $L^2(0,T;H^1(\Omega))$. We also point out that the constant $C$ on the right-hand-side of \eqref{e3.36} is independent of $n$, $\vartheta$ and $\delta$ as long as $E_\delta^{(n)}(0)$ is.
\end{remark}

\begin{proof}[Proof of Lemma~\ref{l3.4}]
The energy identity \eqref{e3.35} follows directly from \eqref{e3.33} upon integrating over $[0,t]$. 

Regarding \eqref{e3.36}, we see that, since $\Phi=\Phi_1+\Phi_2$, where $\Phi_1$ is given by \eqref{e3.0}, it suffices to show that
\begin{equation}\label{e3.36'}
\int_0^T\int_\Omega\left( \sum_{j=1}^N\frac{|\nabla c_j|^2}{c_j}+c_j|\nabla \Phi_2|^2+|\Delta \Phi_2|^2 \right)dx\, dt\le C.
\end{equation}
As pointed out before, the mass conservation \eqref{e3.mass_n} implies that 
\[
\left|\int_\Omega (\sum_{j=1}^Nz_j c_j^{(n)})\Phi_1 dx\right| \le N\|\Phi_1\|_{L^\infty}\max_j |z_j|\, \|c_j^0\|_{L^1(\Omega)}.
\]
Therefore, from \eqref{e3.35} we have, in particular, that
\[
\int_0^T\int_\Omega \sum_{j=1}^Nc_j^{(n)}\Big|\frac{\nabla c_j^{(n)}}{c_j^{(n)}}+ z_j\nabla\Phi^{(n)}\Big|^2dx\, ds \le C,
\]
for some positive constant $C$ independent of $n$, $\vartheta$ and $\delta$.
Then, noting that $$\frac{|\nabla c_j^{(n)}|^2}{c_j^{(n)}}=4\Big|\nabla\sqrt{c_j^{(n)}}\Big|^2$$ and using once again the smoothness of $\Phi_1$, we have 
\begin{equation}\label{e3.38}
\int_0^T\int_\Omega \sum_{j=1}^N\Big( 4\Big|\nabla\sqrt{c_j^{(n)}}\Big|^2+ 2z_j\nabla c_j^{(n)}\cdot\nabla\Phi_2^{(n)}+ z_j^2c_j^{(n)}|\nabla\Phi_2^{(n)}|^2\Big)dx\, ds \le C.
\end{equation}
Thus, it is clear that we only have to deal with the second term in \eqref{e3.38}, which we estimate as follows.

Integrating by parts, we have
\begin{align}\label{e3.40}
\int_0^T\int_\Omega\sum_{j=1}^N & 2z_j\nabla c_j^{(n)}\cdot\nabla\Phi^{(n)} dx\, dt \nonumber\\
&= -2\int_0^T\int_\Omega\sum_{j=1}^N z_jc_j^{(n)}\Delta \Phi^{(n)}dx\, dt	+2\sum_{j=1}^N\int_0^T\int_{\partial\Omega}z_jc_j\partial_\nu\Phi_2 dS\, dt\nonumber\\
&=: -2I_0 + 2\sum_{j=1}^N I_j.
\end{align}
Using the boundary conditions in \eqref{e3.P}, for each $j=1,...,n$, we have that
\begin{align*}
I_j&= -\tau z_j\int_0^T\int_{\partial\Omega}c_j^{(n)}\Phi_2^{(n)} dS\, dt\\
 &\ge -C\|c_j^{(n)}\|_{L^1(0,T;L^{4/3}(\partial\Omega))}\|\Phi_2^{(n)}\|_{L^\infty(0,T;L^4(\partial\Omega))}\\
 &=-C\|\sqrt{c_j^{(n)}}\|_{L^2(0,T;L^{8/3}(\partial\Omega))}^2\|\Phi_2^{(n)}\|_{L^\infty(0,T;L^4(\partial\Omega))}.
\end{align*}
We now invoke once again Poincaré's inequality with boundary term (that is, the fact that $\|\cdot\|_{L^2(\partial\Omega)}+\| \nabla \cdot\|_{L^2(\Omega)}$ is an equivalent norm on $H^1(\Omega)$, see e.g. \cite{BL,MV}) to conclude from \eqref{e3.35} that $\|\Phi_2^{(n)}\|_{L^\infty(0,T;H^1(\Omega))}^2$, and therefore also $\|\Phi_2^{(n)}\|_{L^\infty(0,T;L^4(\partial\Omega))}$, is bounded by a constant independent of $n$, $\vartheta$ and $\delta$.
We also use the compactness of the embedding $H^1(\Omega)\hookrightarrow L^{8/3}(\partial\Omega)$ to estimate for a given $\varepsilon>0$
\[
\|\sqrt{c_j^{(n)}}\|_{L^2(0,T;L^{8/3}(\partial\Omega))}^2\le \varepsilon \|\nabla\sqrt{c_j^{(n)}}\|_{L^2(0,T;L^2(\Omega))}^2 + C_\varepsilon \|\sqrt{c_j^{(n)}}\|_{L^2(0,T;L^2(\Omega))}.
\] 
Then, choosing $\varepsilon>0$ small enough and using the conservation of mass \eqref{e3.mass_n}, we have
\begin{equation}\label{e3.41}
I_j\ge -\|\nabla\sqrt{c_j^{(n)}}\|_{L^2(0,T;L^2(\Omega))}^2 - C, \quad j=1,...,N.
\end{equation}

Regarding $I_0$ we use \eqref{e3.P} and integration by parts to obtain
\begin{align*}
I_0 &= -\int_0^T\int_\Omega \left((1-\vartheta\Delta)\Psi^{(n)}\right)\Psi^{(n)} dx\, dt\\
	&=-\int_0^T\int_\Omega|\Psi^{(n)}|^2dx\, dt - \vartheta\int_0^T\int_\Omega|\nabla\Psi^{(n)}|^2dx\, dt + \vartheta\int_0^T\int_{\partial\Omega} (\partial_\nu\Psi^{(n)})\Psi^{(n)} dS\, dt.
\end{align*}
Thus, using the boundary conditions for $\Psi^{(n)}$ we obtain
\begin{equation}\label{e3.42}
I_0=-\int_0^T\int_\Omega|\Psi^{(n)}|^2dx\, dt - \vartheta\int_0^T\int_\Omega|\nabla\Psi^{(n)}|^2dx\, dt - \tau\vartheta\int_0^T\int_\Omega |\Psi^{(n)}|^2 dS\, dt.
\end{equation}

Finally, gathering \eqref{e3.40}, \eqref{e3.41} and \eqref{e3.42} in \eqref{e3.38} we arrive at
\begin{multline}
\int_0^T\int_\Omega \sum_{j=1}^N\Big( 2\Big|\nabla\sqrt{c_j^{(n)}}\Big|^2+  z_j^2c_j^{(n)}|\nabla\Phi_2^{(n)}|^2\Big)dx\, ds \\
+ 2\int_0^T\int_\Omega|\Psi^{(n)}|^2dx\, dt + 2\vartheta\int_0^T\int_\Omega|\nabla\Psi^{(n)}|^2dx\, dt + 2\tau\vartheta\int_0^T\int_{\partial\Omega} |\Psi^{(n)}|^2 dS\, dt\le C,
\end{multline}
which readily implies \eqref{e3.36} and \eqref{e3.psi}.
\end{proof}

\subsection{Convergence of the Galerkin approximations}\label{S3.5}

Now, we want to take the limit as $n\to \infty$ in the sequence of Galerkin approximation $(\rho_n,u_n, c_j^{(n)}, \Phi^{(n)})$. We point out that the functions $u_n$ and $\rho_n$ may be treated exactly as in \cite[Section 7.3.6]{F}. Therefore, we focus only on the convergence of $(c_j^{(n)}, \Phi^{(n)})$, which follows from the following general result, regarding the weak stability of solutions of the regularized PNP system \eqref{e3.NP}-\eqref{e3.P}, combined with the energy estimates from Lemma~\ref{l3.4}.

\begin{lemma}\label{l3.6}
Let	$\{ u_n\}_{n\in\mathbb{N}}$ be a sequence in $L^2(0,T;H_0^1(\Omega))$ and let $(c^{(n)},\Phi^{(n)})$ be a solution of the (regularized) PNP system \eqref{e3.NP}-\eqref{e3.P} with $u=u_n$ and $\vartheta\ge 0$. 
Suppose that
\begin{equation}\label{e3.est}
\begin{cases}
\sqrt{c_j^{(n)}} \text{ is bounded in $L^2(0,T;H^1(\Omega))$,}\\
\Phi^{(n)} \text{ is bounded in $L^\infty(0,T;H^1(\Omega))$,}\\
\sqrt{c_j^{(n)}}\nabla\Phi^{(n)} \text{ is bounded in $L^2((0,T)\times\Omega)$.}
\end{cases}
\end{equation}
Assume also that
\begin{equation}\label{e3.estu}
u_n\rightharpoonup u \text{ weakly in $L^2(0,T;H_0^1(\Omega))$.}
\end{equation}
Then, there are $c_j\in L^\infty(0,T;L^1(\Omega))\cap L^1(0,T; W^{1,3/2}(\Omega))$, $j=1,...,N$, $\Phi\in L^\infty(0,T;H^1(\Omega))\cap C([0,T];L^p(\Omega))$, for all $p\in [1,6)$, and a subsequence of $(c^{(n)},\Phi^{(n)})$ (not relabeled) such that
\begin{align*}
&c_j^{(n)}\to c_j  \text{ strongly in $L^1(0,T;L^p(\Omega))$ for $1\le p< 3$,}\\
&\nabla c_j^{(n)}\rightharpoonup \nabla c_j \text{ weakly in $L^2(0,T;L^1(\Omega))\cap L^1(0,T;L^q(\Omega))$, for $1\le q< 3/2$,}\\
&\nabla \Phi^{(n)}\rightharpoonup \nabla \Phi \text{ weakly-* in $L^\infty(0,T; L^2(\Omega))$,}\\
&\Phi^{(n)}\to \Phi \text{ strongly in $C([0,T];L^p(\Omega))$ for $1\le p < 6$.}
\end{align*}
Moreover, there are $r_1,r_2 >1$ such that 
\begin{align*}
&c_j^{(n)}\nabla\Phi^{(n)}\rightharpoonup c_j\nabla\Phi \text{ weakly in $L^{r_1}((0,T)\times\Omega)$,}\\
&c_j^{(n)}u_n\rightharpoonup c_j u \text{ weakly in $L^{r_2}((0,T)\times\Omega)$,}
\end{align*}
and the limit functions $u$, $(c_1,...,c_N,\Phi)$ are a weak solution of \eqref{e3.NP}-\eqref{e3.P}.
\end{lemma}

\begin{remark}\label{R:3.7}
The result of Lemma~\ref{l3.6} implies Theorem~\ref{T1.2}. Indeed, the proof of Lemma~\ref{l3.6} below does not require $\vartheta$ to be strictly positive and all the arguments hold uniformly with respect to $\vartheta$, as long as the assumptions \eqref{e3.est} and \eqref{e3.estu} are uniform in $\vartheta$. In particular, the same proof may be carried out line by line for the original PNP equations which correspond to \eqref{e3.NP}-\eqref{e3.P} with $\vartheta=0$. This fact will justify the convergence and consistency of the sequence of the approximate ion densities and electrostatic potentials when we consider the limit as the regularizing parameters $\vartheta$ and $\delta$ vanish.

Note, also, that the assumptions \eqref{e3.est} and \eqref{e3.estu} are consistent with the natural energy estimates that arise from the dissipative structure of the equations, discussed in Section~\ref{S:2} (cf. Lemma~\ref{l3.4}).
\end{remark}

\begin{proof}[Proof of Lemma \ref{l3.6}]
First, we note that 
$\sqrt{c_j^{(n)}}$ and hence  $\sqrt{c_j^{(n)}+1}$ is bounded in $L^2(0,T;H^1(\Omega))$. 
Next, we see that
\begin{align*}
2\partial_t\sqrt{c_j^{(n)}+1} &= \frac{\partial_t c_j^{(n)}}{\sqrt{c_j^{(n)}+1}}\\
	&= -\frac{\div (c_j^{(n)} u_n)}{\sqrt{c_j^{(n)}+1}} + \frac{\div\Big( \nabla c_j^{(n)} + z_jc_j^{(n)}\nabla\Phi^{(n)}\Big)}{\sqrt{c_j^{(n)}+1}}\\
	&= -\div\left( \frac{c_j^{(n)} u_n}{\sqrt{c_j^{(n)}+1}} \right) - \frac{c_j^{(n)}u_n\cdot\nabla c_j^{(n)}}{2(c_j^{(n)}+1)^{3/2}}\\
	&\qquad + \div\left( \frac{\nabla c_j^{(n)} + z_j c_j^{(n)}\nabla\Phi^{(n)}}{\sqrt{c_j^{(n)}+1}} \right) +  \frac{\Big(\nabla c_j^{(n)} + z_j c_j^{(n)}\nabla\Phi^{(n)}\Big)\cdot\nabla c_j^{(n)}}{2(c_j^{(n)}+1)^{3/2}}\\
	&=-\div\left(\Bigg(\frac{c_j^{(n)}}{c_j^{(n)}+1}\Bigg)^{1/2} \sqrt{c_j^{(n)}} u_n \right) - \Bigg(\frac{c_j^{(n)}}{(c_j^{(n)}+1)}\Bigg)^{3/2} u_n\cdot\nabla \sqrt{c_j^{(n)}}\\
	&\qquad + \div\left( \Bigg(\frac{c_j^{(n)}}{c_j^{(n)}+1}\Bigg)^{1/2}\left(2\nabla \sqrt{c_j^{(n)}} + z_j \sqrt{c_j^{(n)}}\nabla\Phi^{(n)}\right) \right) \\
	&\qquad+  \frac{c_j^{(n)}}{(c_j^{(n)}+1)^{3/2}}|\nabla \sqrt{c_j^{(n)}}|^2 +  \Bigg(\frac{c_j^{(n)}}{c_j^{(n)}+1}\Bigg)^{3/2}z_j\nabla\sqrt{c_j^{(n)}}\cdot\nabla\Phi^{(n)}\\
	&:=J_1+J_2+J_3+J_4+J_5.
\end{align*}

From \eqref{e3.est}, we readily see that $J_1$ and $J_3$ are bounded in $L^1(0,T;H^{-1}(\Omega))$ and that $J_2$, $J_4$ and $J_5$ are bounded in $L^1((0,T)\times \Omega)$. Thus, using Aubin-Lions lemma we see that $\sqrt{c_j^{(n)}+1}$ is relatively compact in $L^2(0,T;L^p(\Omega))$ for any $1\le p <6$. Then, there is a function $c_j$ such that, up to a subsequence
\begin{equation}\label{e3.45}
c_j^{(n)}\to c_j \text{ in $L^1(0,T;L^p(\Omega))$, for any $1\le p <3$},
\end{equation}
and since $\sqrt{c_j^{(n)}}$ is bounded in $L^2(0,T;H^1(\Omega))$ we see that also up to a subsequence
\begin{equation}\label{e3.sqrtcj}
\sqrt{c_j^{(n)}}\rightharpoonup \sqrt{c_j} \text{ weakly in $L^2(0,T;H^1(\Omega))$}.
\end{equation}
In particular, the $c_j^{(n)}$ are bounded in $L^1(0,T;L^3(\Omega))$ and $c_j \in L^1(0,T;L^3(\Omega))$.

Next, using the interpolation inequality $\|\cdot\|_{L^r}\le \|\cdot\|_{L^3}^\theta\|\cdot\|_{L^1}^{1-\theta}$, where $0\le\theta\le 1$ and $\frac{1}{r}=\frac{\theta}{3}+\frac{1-\theta}{1}$, and taking into account the conservation of mass 
\begin{equation}\label{e3.mass}
\int_\Omega c_j^{(n)}(t,x)\, dx = \int_\Omega c_{j}^0(x)\, dx,
\end{equation}
 we see that
\begin{equation}\label{e3.48}
c_j^{(n)} \text{ are bounded in $L^{1/\theta}(0,T;L^r(\Omega)),$ where $r=\frac{3}{3-2\theta}$ and $0<\theta\le 1$.}
\end{equation}
 Also, by \eqref{e3.45}, \eqref{e3.sqrtcj} and \eqref{e3.mass}, one has
\begin{multline}\label{e3.47}
\nabla c_j^{(n)}=\sqrt{c_j^{(n)}}\, \nabla\sqrt{c_j^{(n)}}\rightharpoonup \nabla c_j \\
\text{ weakly in $L^2(0,T;L^1(\Omega))\cap L^1(0,T;L^q(\Omega))$, for $q\in [1,3/2)$.}
\end{multline}
Moreover, since 
\[
\|\nabla c_j^{(n)}\|_{L^1(0,T;L^{3/2}(\Omega))}\le \|\sqrt{c_j^{(n)}}\|_{L^2(0,T;L^6(\Omega))}\|\nabla\sqrt{c_j^{(n)}}\|_{L^2((0,T)\times\Omega)}\le C,
\]
we have that $\|\nabla c_j\|_{L^1(0,T;L^{3/2}(\Omega))}$.

Now we deal with the compactness of $\Phi^{(n)}$. From \eqref{e3.est} we have that $\Phi^{(n)}$ is bounded in $L^\infty(0,T;H^1(\Omega))$, so that, there is a function $\Phi\in L^\infty(0,T;H^1(\Omega))$ such that, up to a subsequence, 
\begin{equation}
\Phi^{(n)}\rightharpoonup \Phi \text{ weakly in $L^q(0,T;H^1(\Omega))$, for any $q\in[1,\infty)$.}
\end{equation}
In fact, $\nabla\Phi^{(n)}\rightharpoonup\nabla\Phi^{(n)}$ weakly-* in $L^\infty(0,T;L^2(\Omega))$.

Since $c_j^{(n)}$ and $\Phi^{(n)}$ solve equations \eqref{e3.NP} and \eqref{e3.P} with $u=u_n$, we have that
\[
\partial_t c_j^{(n)} = -\div(c_j^{(n)}u_n) + \div(V_j^{(n)}),
\] 
where $V_j^{(n)}=\nabla c_j^{(n)}+z_jc_j^{(n)}\nabla\Phi^{(n)}$. Here, taking $\theta=\frac{3}{5}$ in \eqref{e3.48}, we have that the $\sqrt{c_j^{(n)}}$ are bounded in $L^{10/3}((0,T)\times\Omega)$. Thus,
\[
\|V_j^{(n)}\|_{L^{5/4}((0,T)\times\Omega)}\le \|\sqrt{c_j^{(n)}}\|_{L^{10/3}((0,T)\times\Omega)}\|(c_j^{(n)})^{-1/2}V_j^{(n)}\|_{L^2((0,T)\times\Omega)}\le C.
\]
Furthermore, taking $\theta=\frac{3}{8}$, we have
\[
\|c_j^{(n)}u_n\|_{L^{8/7}(0,T;L^{12/11}(\Omega))}\le \|c_j^{(n)}\|_{L^{8/7}(0,T;L^{4/3}(\Omega))}\|u_n\|_{L^2(0,T;L^6(\Omega))}\le C.
\] 
As a consequence, the $\partial_t c_j^{(n)}$ are bounded in $L^{8/7}(0,T;W_0^{-1,12/11}(\Omega))$, which, from \eqref{e3.P}, implies that $\partial_t \Phi_n$ is bounded in $L^{8/7}(0,T;W^{1,\frac{12}{11}}(\Omega))$. Hence, the Aubin-Lions lemma implies that, up to a subsequence, we have that
\[
\Phi^{(n)}\to \Phi, \quad\text{in $C([0,T];L^p(\Omega)$, for $1\le p<6$.}
\]

 Now, taking into account \eqref{e3.45}, we conclude that
\begin{equation}\label{e3.49}
c_j^{(n)}\nabla\Phi^{(n)}\rightharpoonup c_j\nabla\Phi \text{ weakly in $L^1((0,T)\times\Omega)$.}
\end{equation}
Finally, since
\begin{equation}
u_n\rightharpoonup u \text{ weakly in $L^2(0,T;H_0^1(\Omega))$},
\end{equation}
we also have that
\begin{equation}
c_j^{(n)}u_n\rightharpoonup c_j u  \text{ weakly in $L^1((0,T)\times\Omega)$.}
\end{equation}

At this point, we recall that $c_j^{(n)}$ satisfies \eqref{e3.NP} in the sense that 
\begin{equation}\label{e3.NPweak_n}
\int_0^T\int_\Omega \Big(-c_j^{(n)}\eta_t +( -c_j^{(n)}u_n + \nabla c_j^{(n)} + z_jc_j^{(n)}\nabla\Phi^{(n)})\cdot\nabla\eta \Big)dx\, dt = \int_\Omega c_j^0 \eta(0)dx,
\end{equation}
for any $\eta\in C^\infty([0,T]\times\overline{\Omega})$ such that $\eta(T,\cdot)=0$. And we note that we can pass to the limit as $n\to\infty$ in each term to conclude that $c_j$ satisfies
\begin{equation}\label{e3.NPweak_infty}
\int_0^T\int_\Omega \Big(-c_j\eta_t +( -c_ju + \nabla c_j + z_jc_j\nabla\Phi)\cdot\nabla\eta \Big)dx\, dt = \int_\Omega c_j^0 \eta(0)dx,
\end{equation}
which is the weak formulation of equation \eqref{e3.NP}. We may also pass to the limit as $n\to \infty$ in order to conclude that $\Phi$ solves \eqref{e3.P}.
\end{proof}

Let us finally conclude the proof of Proposition~\ref{P:3.1}. To that end, let $(\rho_n,u_n, c_j^{(n)}, \Phi^{(n)})$ be the sequence of Galerkin approximations constructed in Subsection~\ref{S:3.3} above. Due to the estimates from Lemma~\ref{l3.4}, we can apply Lemma~\ref{l3.6} in order to conclude that there is a subsequence (not relabeled) such that the limit functions $u$, $c_1,...,c_N$ and $\Phi$ solve the limit PNP subsystem. 	

Since the estimates from Lemma~\ref{l3.4} are uniform with respect to $\vartheta$, then, by virtue of \eqref{e3.45}, it follows that $\Phi^{(n)}(t)\to \Phi(t)$ in $H^1(\Omega)$ for a.e. $t$, uniformly in $\vartheta$ (cf. estimate \eqref{e3.27}). Actually, a similar estimate to \eqref{e3.27} yields the fact that $\Psi_n(t)\to \Psi(t)$ in $L^2(\Omega)$, for a.e. $t$, uniformly in $\vartheta$, where $\Psi$ is the solution to \eqref{e3.5'} with \eqref{e3.12}. Hence, $\Phi^{(n)}(t)\to \Phi(t)$ in $H^2(\Omega)$ for a.e. $t$, also uniformly in $\vartheta$.

Finally, we observe that since the limit function $c_j$ belongs to $L^1(0,T;W^{1,3/2}(\Omega))$, we actually have that $\Phi_2$ is in $L^1(0,T;W^{3,3/2}(\Omega))$.

As already mentioned, the convergence of $(\rho_n, u_n)$ may be performed following the arguments in \cite[Section 7.3.6]{F} and it only remains to verify the convergence of the terms in the momentum equation related to the functions $c_j^{(n)}$ and $\Phi^n$. These are the last two terms in \eqref{e3.29}, whose convergence is guaranteed by \eqref{e3.47} and \eqref{e3.49}. Thus, we finally conclude that the limit functions $(\rho, u, c_j, \Phi, \Psi)$ are a solution of system \eqref{e3.1}-\eqref{e3.12}.

At last, we see that the inequalities \eqref{e3.15} and \eqref{e3.T_Phi} hold
 by the lower semi-continuity when letting $n\to\infty$ in \eqref{e3.35} and in \eqref{e3.36}, respectively, thus completing the the proof of Proposition~\ref{P:3.1}.




\section{Vanishing Viscosity and Artificial Pressure}\label{S4}

Let $(\rho_{\vartheta,\delta},u_{\vartheta,\delta},c_j^{({\vartheta,\delta})},\Phi^{(\vartheta,\delta)})$ be the solution of the approximate problem \eqref{e3.1}-\eqref{e3.12} given by Proposition~\ref{P:3.1}. In order to find solutions to the original PNPNS system, we consider the limit as $\vartheta\to 0$ first and then as $\delta\to 0$. 

Once we justify the convergence of the ion densities and the electrostatic potential, and guarantee that the forcing terms in the momentum equation converge to their counterparts in the limit, the theory from \cite{L2,F} may be applied almost directly in order to find the solutions to the original system in the limit. As usual, the key point is to show the strong convergence of the sequence $\{\rho_\delta\}$, which follows by the weak continuity of the effective viscous flux, together with the fact that the fluid's density satisfies the continuity equation in the sense of renormalized solutions. Since this procedure is already somewhat well understood in the literature for the case of the Navier-Stokes equations, we focus on the modifications that have to be made in order to accommodate the terms related to the ion densities.

We first consider the limit as $\vartheta\to 0$. When $\vartheta=0$, the approximate system reads as
\begin{align}
&\partial_t\rho+\div(\rho u)=0,\label{e4.1}\\ 
&\partial_t(\rho u) + \div(\rho u\otimes u)+\nabla (a\rho^\gamma + \delta \rho^\beta) = \div\mathbb{S}-\sum_{j=1}^N\nabla c_j-\sum_{j=1}^N z_j c_j \nabla\Phi,\label{e4.2}\\
&\partial_tc_j+\div(c_ju)=\div\left( \nabla c_j+z_jc_j\nabla\Phi \right),\label{e4.3}\\
&-\Delta\Phi=\sum_{j=1}^Nz_j c_j ,\label{e4.4}
\end{align}  
Accordingly, the initial and boundary conditions for system \eqref{e4.1}-\eqref{e4.5} are 
\begin{equation}\label{e4.5}
(\rho,\rho u, c_1,...,c_N)(0,x)=(\rho_{0,\delta}, m_{0,\delta},c_1^{(0,\delta)},...,c_N^{(0,\delta)})(x), \quad x\in\Omega,
\end{equation}
and 
\begin{align}
&u|_{\partial \Omega}=0,\label{e4.6}\\
&\left(\partial_\nu c_j  - c_j\partial_\nu\Phi\right)|_{\partial \Omega} = 0,\label{e4.7}\\
&(\partial_\nu\Phi + \tau\Phi)|_{\partial \Omega}=V.\label{e4.8}
\end{align}

\begin{proposition}\label{P:4.1}
Suppose that the initial data $(\rho_{0,\delta}, m_{0,\delta},c_j^{(0,\delta)})$ is smooth and satisfies 
\begin{equation}
\tilde{M}_{0,\delta}\le \rho_{0,\delta}\le M_{0,\delta}
\end{equation}
and
\begin{equation}
0\le c_j^{0,\delta}\le M_{0,\delta},
\end{equation}
for some positive constants $\tilde{M}_{0,\delta}$ and $M_{0,\delta}$. Assume further that $\partial_\nu\rho_0|_{\partial \Omega}=0$ and that $\beta>\max\{4,\frac{6\gamma}{2\gamma-3}\}$. 
Let $(\rho_{\vartheta,\delta},u_{\vartheta,\delta},c_j^{({\vartheta,\delta})},\Phi^{(\vartheta,\delta)})$ be the solution of the regularized problem \eqref{e3.1}-\eqref{e3.12} given by Proposition~\ref{P:3.1}. 
Then, there is a subsequence (not relabeled) that converges, as $\vartheta \to 0$, to a global finite energy weak solution $(\rho_{\delta},u_{\delta},c_j^{({\delta})},\Phi^{(\delta)})$ of \eqref{e4.1}-\eqref{e4.8}, where the initial and boundary conditions are satisfied in the sense of distributions.
Moreover, $\rho_\delta$ is nonnegative and satisfies the continuity equation \eqref{e4.1} in the sense of renormalized solutions. The ion densities $c_j^{(\delta)}$ are also nonnegative and have conserved mass:
\begin{equation}
\int_\Omega c_j^{(\delta)}(t) dx = \int_\Omega c_j^{(0,\delta)} dx,\quad \text{for a.e. $t\in [0,T]$, $j=1,...,N$}.
\end{equation}
Furthermore, the limit functions satisfy the following energy inequality
\begin{multline}\label{e4.9}
E_{\delta}(t) + \int_0^t\int_\Omega \Big(\mu|\nabla u_\delta|^2 + (\lambda+\mu)(\div u_\delta)^2 + \sum_{j=1}^Nc_j^{(\delta)}\Big|\frac{\nabla c_j^{(\delta)}}{c_j^{(\delta)}}+ z_j\nabla\Phi^{(\delta)}\Big|^2\Big)dx\, ds
 \le E_{\delta}(0),
\end{multline}
where,
\begin{multline}\label{e4.energy-delta}
E_{\delta}(t)=\int_\Omega\Bigg( \rho_\delta\left(\frac{1}{2} |u_\delta|^2 + \frac{a}{\gamma-1}\rho_\delta^{\gamma-1} + \frac{\delta}{\beta-1}\rho_\delta^{\beta-1}\right) \\
+ \sum_{j=1}^N(c_j^{(\delta)}\ln c_j^{(\delta)} - c_j^{(\delta)} + 1) + \frac{1}{2}|\nabla\Phi^{(\delta)}|^2 \Bigg)dx +\frac{\tau}{2}\int_{\partial\Omega}|\Phi^{(\delta)}|^2dS.
\end{multline}
Also, if $\rho_{0,\delta}^\gamma + \delta\rho_{0,\delta}^\beta$ is bounded in $L^1(\Omega)$, uniformly with respect to $\delta$, then $E_\delta(0)$ is also bounded uniformly, with respect to $\delta$, and there is a positive constant $C$, independent of $\delta$, such that
\begin{equation}\label{e4.T_Phi}
\int_0^T\int_\Omega\left( \sum_{j=1}^N\frac{|\nabla c_j^{(\delta)}|^2}{c_j^{(\delta)}}+c_j^{(\delta)}|\nabla \Phi^{(\delta)}|^2+|\Delta \Phi^{(\delta)}|^2 \right)dx\, dt\le C.
\end{equation}

\end{proposition}

Let us point out that the limit of $(c_j^{({\vartheta,\delta})},\Phi^{(\vartheta,\delta)}),$ as $\vartheta\to 0,$ follow from Lemma~\ref{l3.6}, based on the bounds independent of $\vartheta$ provided by the energy estimate \eqref{e3.35}, which corresponds to \eqref{e3.15} (cf. Remark~\ref{R:3.7}). Moreover, the limit of $(\rho_{\vartheta,\delta},u_{\vartheta,\delta})$ as $\vartheta\to 0$ may be carried out following \cite[Section 7.4]{F}, based on the energy inequality \eqref{e3.15} together with some additional pressure estimates, where, as usual, the key to ensure that limit functions are a solution of the momentum equation is the strong convergence of the densities. 

Since the convergence of the approximate solutions is more delicate and the key arguments are more general when taking the limit as the artificial pressure vanishes we omit the proof of Proposition~\ref{P:4.1} and focus on the analysis when $\delta\to 0$.

\begin{remark}
The energy inequality obtained from \eqref{e3.15}, as $\vartheta\to 0$, by lower semicontinuity, is \eqref{e4.9} but with $\tilde{E}_\delta(t)$ instead of $E_\delta(t)$, where
\begin{multline}\label{e4.deltatilde}
\tilde{E}_{\delta}(t)=\int_\Omega\Bigg( \rho\left(\frac{1}{2} |u|^2 + \frac{a}{\gamma-1}\rho^{\gamma-1} + \frac{\delta}{\beta-1}\rho^{\beta-1}\right) \\
+ \sum_{j=1}^N(c_j\ln c_j - c_j + 1) + \frac{1}{2}|\nabla\Phi_2|^2+(\sum_{j=1}^Nz_jc_j)\Phi_1 \Bigg)dx +\frac{\tau}{2}\int_{\partial\Omega}|\Phi_2|^2dS.
\end{multline}
Here, $\Phi_1$ is the unique solution to
\begin{equation}\label{e4.0}
\begin{cases}
-\Delta \Phi_1 = 0,&\text{in $\Omega$},\\
\partial_\nu \Phi_1 + \tau\Phi_1 = V,&\text{on $\partial\Omega$}.\\
\end{cases}
\end{equation}
and $\Phi_2(t,x)=\Phi(t,x)-\Phi_1(x)$. 

Noting that $-\Delta \Phi_2 = \sum_{j=1}^Nz_jc_j$, with $(\Phi_2+\tau\partial_\nu\Phi_2)|_{\partial \Omega}=0$, we see that
\[
\int_\Omega \left(\sum_{j=1}^Nz_jc_j\right)\Phi_1 dx = \int_\Omega \nabla\Phi_2\cdot\nabla \Phi_1 dx + \tau\int_{\partial\Omega}\Phi_2\Phi_1\, dS.
\]
Thus, we have that
\[
E_\delta(t)=\tilde{E}_\delta(t) + \frac{1}{2}\int_\Omega |\nabla\Phi_1|^2dx + \frac{\tau}{2}\int_{\partial \Omega}|\Phi_1|^2 dS,
\]
and we readily obtain \eqref{e4.9}.
\end{remark}

\begin{remark}
Note that \eqref{e4.9} implies that $\sqrt{c_j^{(\delta)}}$ is bounded in $L^2(0,T;H^1(\Omega))$ and, therefore, 
\begin{equation}\label{e4.17}
\|\nabla c_j^{(\delta)}\|_{L^1(0,T;L^{3/2}(\Omega))}\le \|\sqrt{c_j^{(\delta)}}\|_{L^2(0,T;L^6(\Omega))}\|\nabla\sqrt{c_j^{(\delta)}}\|_{L^2((0,T)\times\Omega)}\le C,
\end{equation}
for some positive constant $C$, independent of $\delta$.
Moreover, we have that $c_j^{(\delta)}$ are bounded in $L^1(0,T;L^3(\Omega))$,  uniformly in $\delta$, which implies that 
\begin{equation}\label{e4.17'}
\|c_j^{(\delta)}\nabla\Phi^{(\delta)}\|_{L^1(0,T;L^{6/5}(\Omega))}\le \|c_j^{(\delta)}\|_{L^1(0,T;L^3(\Omega))}\|\nabla\Phi^{(\delta)}\|_{L^\infty(0,T;L^2(\Omega))} \le C. 
\end{equation} 
Also, arguing as in Subsection~\ref{S3.5}, using equations \eqref{e4.3} and \eqref{e4.4} and the Aubin-Lions lemma, we may conclude that $\Phi^{(\delta)}$ is bounded in $C([0,T];L^q(\Omega))$ for any $1\le q <6$.
\end{remark}

\subsection{Integrability of the fluid density}\label{S4.1}

Let $(\rho_{\delta},u_{\delta},c_j^{({\delta})},\Phi^{(\delta)})$ be the solution of system \eqref{e4.1}-\eqref{e4.8} provided by Proposition~\ref{P:4.1}. We first derive an estimate on the fluid  density, uniform in $\delta$, that will justify the limit $\delta \rho_\delta^\beta\to 0$, as $\delta\to 0$. The proof is inspired by the analogue estimate in \cite{HW} (cf. \cite{FNP}).

\begin{lemma}\label{l4.1}
Let $\rho_{0,\delta}^\gamma + \delta\rho_{0,\delta}^\beta$ be bounded in $L^1(\Omega)$, uniformly in $\delta$.
Then, there is a positive constant $C$, independent of $\delta$, such that 
\begin{equation}\label{e4.16}
\int_0^T\int_\Omega (a\rho_\delta^\gamma+\delta\rho_\delta^\beta)\ln(1+\rho_\delta)dx\, dt \le C.
\end{equation}
\end{lemma}

\begin{proof}
Let $b(\rho)=\ln(1+\rho)$. Since $\rho_\delta$ is a renormalized solution of \eqref{e4.1}, we have that
\begin{equation}\label{e4.18}
\ln(1+\rho_\delta)_t+\div\Big(\ln(1+\rho_\delta)u_\delta\Big)+\left(\frac{\rho_\delta}{1+\rho_\delta}-\ln(1+\rho_\delta) \right)\div u_\delta = 0.
\end{equation}

Let us consider the Bogovskii operator, i.e., the bounded linear operator
\[
B:\left\{f\in L^p(\Omega):\int_\Omega f\, dx=0\right\}\to [W_0^{1,p}(\Omega)]^3,
\]
which satisfies the estimate
\[
\| B[f]\|_{W_0^{1,p}(\Omega)}\le C(p)\| f\|_{L^p(\Omega)}, \quad 1<p<\infty,
\]
such that the function $W=B[f]\in\mathbb{R}^3$ satisfies the equation
\[
\div W = f\text{ in $\Omega$,}\quad W|_{\partial\Omega}=0.
\] 
Moreover, if $f=\div g$ for some $g\in L^r(\Omega)$ with $\partial_\nu g|_{\partial \Omega} =0$, then
\[
\| B[f]\|_{L^r(\Omega)}\le C(r)\| g\|_{L^r(\Omega)}.
\]

We then, define the test function $\varphi$ by its coordinates
\[
\varphi_i = \psi(t)B_i\left[\ln(1+\rho_\delta)- \Medint_\Omega \ln(1+\rho_\delta) dx \right],\quad i=1,2,3,
\]
where, $\medint_\Omega \ln(1+\rho_\delta) dx=\frac{1}{|\Omega|}\int_\Omega \ln(1+\rho_\delta) dx$ and $\psi\in C_c^\infty(0,T)$.
In view of equation \eqref{e4.18} and the energy inequality \eqref{e4.9} we have that
\[
\ln(1+\rho_\delta)\in C([0,T]; L^p(\Omega),\, \text{for any finite $p>1$.}
\]
Then, we have that
\begin{equation}\label{e4.21}
\varphi_i\in C([0,T];W_0^{1,p}(\Omega)),\, \text{for any finite $p>1$,}
\end{equation}
in particular $\varphi\in C([0,T]\times\Omega)$. Note, also, that from equation \eqref{e4.18} we have
\begin{multline}\label{e4.22}
\partial_t B\left[ \ln(1+\rho_\delta)-\Medint_\Omega\ln(1+\rho_\delta)dx\right] = -B\left[\div\Big( \ln(1+\rho_\delta)u_\delta\Big)\right] \\
- B\left[\frac{\rho_\delta}{1+\rho_\delta} - \ln(1+\rho_\delta)-\Medint_\Omega\left( \frac{\rho_\delta}{1+\rho_\delta} - \ln(1+\rho_\delta) \right) dx\right].
\end{multline}

Now, by virtue of \eqref{e4.21}, we can use $\varphi$ as a test function in the momentum equation \eqref{e4.2}, to obtain, after a long, but straightforward calculation involving \eqref{e4.22}, the following identity
\begin{equation}
\int_0^T\int_\Omega \psi(a\rho_\delta^\gamma + \delta\rho_\delta^\beta)\ln(1+\rho_\delta)dx\, dt = \sum_{j=1}^7 I_j,
\end{equation}
where,
\begin{align*}
&I_1=\int_0^T\psi\int_\Omega (a\rho^\gamma + \delta\rho^\beta)dx \Medint_\Omega\ln(1+\rho_\delta)dx\, dt,\\
&I_2=\int_0^T\int_\Omega \psi \mathbb{S}^\delta : \nabla B\left[\ln(1+\rho_\delta)-\Medint_\Omega\ln(1+\rho_\delta)dx\right] dx\, dt\\
&I_3=-\int_0^T\int_\Omega\psi_t\rho_\delta u_\delta\cdot B\left[\ln(1+\rho_\delta)-\Medint_\Omega\ln(1+\rho_\delta)dx\right] dx\, dt\\
&I_4=-\int_0^T\int_\Omega\psi \rho_\delta u_\delta\otimes u_\delta : \nabla B\left[\ln(1+\rho_\delta)-\Medint_\Omega\ln(1+\rho_\delta)dx\right] dx\, dt\\
&I_5=\int_0^T\int_\Omega\psi\rho_\delta u_\delta\cdot B\Bigg[ \Big( \ln(1+\rho_\delta) - \frac{\rho_\delta}{1+\rho_\delta} \Big)\div u_\delta \\
&\qquad\qquad\qquad\qquad\qquad\qquad - \Medint_\Omega \Big( \ln(1+\rho_\delta) - \frac{\rho_\delta}{1+\rho_\delta} \Big)\div u_\delta\, dx \Bigg] dx\, dt\\
&I_6= \int_0^T\int_\Omega \psi\rho_\delta u_\delta\cdot B\left[\div\Big(\ln(1+\rho_\delta)u_\delta\Big)\right]dx\, dt\\
&I_7=\sum_{j=1}^N\int_0^T\int_\Omega \psi\Big(\nabla c_j^{(\delta)}+z_j c_j^{(\delta)}\nabla\Phi^{(\delta)}\Big)\cdot B\left[ \Medint_\Omega \ln(1+\rho_\delta)dx - \ln(1+\rho_\delta) \right]dx\, dt.
\end{align*}
Here $\mathbb{S}_\delta$ is given by \eqref{e1.6} with $u=u_\delta$.

As pointed out before, if $\rho_{0,\delta}^\gamma + \delta\rho_{0,\delta}^\beta$ is bounded in $L^1(\Omega)$, then the initial energy $E_\delta(0)$ is bounded uniformly with respect to $\delta$. Thus, based on the energy inequality \eqref{e4.9}, just as in the proof of Lemma~5.1 in \cite{HW}, the integrals $I_1$, ..., $I_6$ may be bounded by a positive constant which depends on $\|\psi\|_{L^\infty}$ and on $\|\frac{d}{dt}\psi\|_{L^1}$, but does not depend on $\delta$. Moreover, in view of \eqref{e4.17}, \eqref{e4.17'} and \eqref{e4.21} we have that $I_7$ is also bounded uniformly with respect to $\delta$ by a constant which depends on $\|\psi\|_{L^\infty}$.  

Therefore, we conclude that there is a constant $C=C(\|\psi\|_{L^\infty},\|\frac{d}{dt}\psi\|_{L^1})$ such that
\[
\int_0^T\int_\Omega \psi(a\rho_\delta^\gamma + \delta\rho_\delta^\beta)\ln(1+\rho_\delta)dx\, dt\le C,
\]
and taking $\psi=\psi_n$ where $\|\psi_n\|_{L^\infty}$ and $\|\frac{d}{dt}\psi\|_{L^1}$ are bounded and such that $\psi_n\to 1_{(0,T)}$, as $n\to\infty$, we obtain \eqref{e4.16}.
\end{proof}

Having estimate \eqref{e4.16} at hand, we can conclude that
\begin{equation}
\lim_{\delta\to 0}\int_0^T\int_\Omega \delta\rho_\delta^\beta dx\, dt =0.
\end{equation}
In particular,
\begin{equation}
\delta\rho_\delta^\beta\to 0, \text{ in $\mathcal{D}'((0,T)\times\Omega)$, as $\delta\to 0$.}
\end{equation}
This is shown in \cite[section~5]{HW}, by a clever application of the H\"{o}lder inequality in the Orlicz space associated to the function $s\mapsto (1+s)\ln(1+s)-s$.

\subsection{Convergence of the approximate solutions}

Now we consider the limit as $\delta \to 0$ in order to find a solution of the PNPNS system \eqref{e1.1}-\eqref{e1.13}, with $\kappa=D_j=1$. As before, we remark that this last constraint on the physical constants $\kappa$, and $D_j$ is not at all essential and its only purpose is to simplify the notations. 

Let $(\rho_0,m_0,c_j^0)$ satisfy \eqref{e1.initial} and let us consider a sequence of approximate initial data $(\rho_{0\delta},m_{0\delta},c_j^{0\delta})$ such that
\begin{enumerate}
\item[(i)] $\rho_{0\delta}$ is smooth and satisfies 
\begin{align}
&\delta < \rho_{0\delta} < \delta^{-1/2\beta},\\
&\partial_\nu\rho_{0\delta}|_{\partial\Omega}=0,\\
&\rho_{0\delta}\to \rho_0 \text{ in $L^1(\Omega)$, as $\delta \to 0$, and}\\
& |\{x\in\Omega : \rho_{0\delta}(x)<\rho_0(x)\}|\to 0, \text{ as $\delta\to 0$.}
\end{align}
\item[(ii)]
\[
m_{0\delta}(x)=\begin{cases}
m_0(x), &\text{if $\rho_{0\delta}(x)>\rho_0(x)$,}\\
0, &\text{if $\rho_{0\delta}(x)\le\rho_0(x)$.}
\end{cases}
\]
\item[(iii)] $c_j^{0\delta}$ is nonnegative and bounded and $c_j^{0\delta}\to c_j^0$ and $c_j^{0\delta}\ln(c_j^{0\delta})\to c_j^0\ln(c_j^0)$ in  $L^1(\Omega)$, as $\delta\to 0$.
\end{enumerate}
Let $(\rho_{\delta},u_{\delta},c_j^{({\delta})},\Phi^{(\delta)})$ be the corresponding solution of \eqref{e4.1}-\eqref{e4.8} provided by Proposition~\ref{P:4.1}. From the estimates \eqref{e4.9} and \eqref{e4.16} we have that, up to a subsequence, we have, as $\delta\to 0$ that
\begin{equation}
\rho_\delta\to \rho \text{ in $C([0,T];L_{weak}^\gamma(\Omega))$,}
\end{equation}
\begin{equation}
u_\delta\rightharpoonup u \text{ weakly in $L^2(0,T;H_0^1(\Omega))$.}
\end{equation}
Also, by Lemma~\ref{l4.1} and Proposition~2.1 in \cite{F}, we have that
\begin{equation}
\rho_\delta^\gamma\to \overline{\rho_\delta^\gamma} \text{ weakly in $L^1((0,T)\times\Omega)$.}
\end{equation}
Moreover, due to Theorem~\ref{T1.2} (i.e. Lemma~\ref{l3.4}) we have that
\begin{align}
&c_j^{(\delta)}\to c_j , \text{ strongly in $L^1(0,T;L^p(\Omega))$ for $1\le p< 3$,}\\
&\nabla c_j^{(\delta)}\rightharpoonup \nabla c_j \text{ weakly in $L^2(0,T;L^1(\Omega))\cap L^1(0,T;L^q(\Omega))$, for $1\le q< 3/2$,}\label{e4.34}\\
&\nabla \Phi^{(\delta)}\rightharpoonup \nabla \Phi \text{ weakly-* in $L^\infty(0,T; L^2(\Omega))$,}\\
&\Phi^{(\delta)}\to \Phi \text{ strongly in $C([0,T];L^p(\Omega))$ for $1\le p < 6$,}\\
&c_j^{(\delta)}\nabla\Phi^{(\delta)}\rightharpoonup c_j\nabla\Phi \text{ weakly in $L^{r_1}((0,T)\times\Omega)$,}\label{e4.37}\\
&c_j^{(\delta)}u_\delta\rightharpoonup c_j u \text{ weakly in $L^{r_2}((0,T)\times\Omega)$,}
\end{align}
for some $r_1,r_2 > 1$, and the limit functions $u$, $c_1,...,c_N$ and $\Phi$ are a weak solution of the PNP subsystem.

By virtue of the momentum equation \eqref{e4.2} we have that
\begin{equation}
\rho_\delta u_\delta \to \rho u \text{ in $C([0,T];L_{weak}^{\frac{2\gamma}{\gamma+1}}(\Omega))$.} 
\end{equation}
Also,
\begin{equation}
\rho_\delta u_\delta\otimes u_\delta\to \rho u\otimes u, \text{ in $\mathcal{D}'((0,T)\times \Omega)$.}
\end{equation}

As a consequence, letting $\delta\to 0$ we have that the limit functions satisfy the following system in the sense of distributions over $(0,T)\times\Omega$:
\begin{align}
&\partial_t\rho+\div(\rho u)=0,\label{e4.100}\\ 
&\partial_t(\rho u) + \div(\rho u\otimes u)+\nabla (a\overline{\rho^\gamma}) = \div\mathbb{S}-\sum_{j=1}^N\nabla c_j-\sum_{j=1}^N z_j c_j \nabla\Phi,\label{e4.200}\\
&\partial_tc_j+\div(c_ju)=\div\left( \nabla c_j+z_jc_j\nabla\Phi \right),\label{e4.300}\\
&-\Delta\Phi=\sum_{j=1}^Nz_j c_j ,\label{e4.400}
\end{align}  
with the corresponding initial and boundary conditions:
\begin{equation}\label{e4.500}
(\rho,\rho u, c_1,...,c_N)|_{t=0}=(\rho_{0}, m_{0},c_1^{(0)},...,c_N^{(0)}), \quad \text{in $\Omega$},
\end{equation}
and 
\begin{align}
&u|_{\partial \Omega}=0,\label{e4.600}\\
&\left(\partial_\nu c_j  - c_j\partial_\nu\Phi\right)|_{\partial \Omega} = 0,\label{e4.700}\\
&(\partial_\nu\Phi + \tau\Phi)|_{\partial \Omega}=V.\label{e4.800}
\end{align}
Thus, the only thing left to conclude the proof of Theorem~\ref{T1.1} is to show the strong convergence of the fluid  density so that $a\overline{\rho^\gamma}=a\rho^\gamma$.

\subsection{Strong convergence of density}

Let $T_k$, $k\in \mathbb{N}$, be the cut-off function given by
\[
T_k(z)=kT\Big(\frac{z}{k}\Big), \text{ for $z\in\mathbb{R}$,}
\]  
where $T\in C^\infty(\mathbb{R})$ is a concave function such that
\[
T(z)=\begin{cases}
z, & z\le 1,\\
2, & z>3.
\end{cases}
\]
Since $\rho_\delta$ and $u_\delta$ satisfy the continuity equation \eqref{e4.1} in the sense of renormalized solutions, we have that
\begin{equation}\label{e4.49}
T_k(\rho_\delta)_t+ \div(T_k(\rho_\delta)u_\delta)+ \Big( T_k'(\rho_\delta)\rho_\delta-T_k(\rho_\delta)\Big)\div u_\delta = 0,
\end{equation}
in the sense of distributions. Passing to the limit as $\delta\to 0$ we have
\begin{equation}
\overline{T_k(\rho)}_t+ \div(\overline{T_k(\rho)}u)+ \overline{\Big( T_k'(\rho)\rho-T_k(\rho)\Big)\div u} = 0,
\end{equation}
also in the sense of distributions, where the overline stands for a weak limit of the sequence indexed by $\delta$. 
Note that $\overline{T_k(\rho)u}=\overline{T_k(\rho)}\, u$, due to \eqref{e4.49} we have that $T_k(\rho_\delta)\to \overline{T_k(\rho)}$ in $C([0,T];L_{weak}^\gamma(\Omega))$, and using the fact that $L^\gamma(\Omega)$ is compactly embedded in $H^{-1}(\Omega)$.

Now as in \cite{F,FNP} we define the operator $\mathcal{A}$ by its coordinates
\[
\mathcal{A}_j[v]:=\Delta^{-1}\partial_{x_j}v,\qquad j=1,2,3,
\]
where $\Delta^{-1}$ denotes the inverse Laplacian in $\mathbb{R}^3$. Equivalently, $\mathcal{A}_j$ may be defined through its Fourier symbol
\[
\mathcal{A}_j=\mathcal{F}^{-1}\left[ \frac{-i\xi_j}{|\xi|^2}\mathcal{F}[v] \right].
\]
As shown in \cite{F}, $\mathcal{A}$ satisfies 
\begin{align*}
&\|\mathcal{A}_j v\|_{W^{1,s}(\Omega)}\le C(s,\Omega)\|v\|_{L^s(\mathbb{R}^3)},\text{ for $1< s <\infty$},\\
&\|\mathcal{A}_j v\|_{L^q(\Omega)}\le C(q,s,\Omega)\|v\|_{L^s(\mathbb{R}^3)},\text{ for finite $q$, provided $\frac{1}{q}\ge \frac{1}{s}-\frac{1}{3}$,}\\
&\|\mathcal{A}_j v\|_{L^\infty(\Omega)}\le C(s,\Omega)\|v\|_{L^s(\mathbb{R}^3)},\text{ if $s> 3$.}
\end{align*}

Now we consider the function $\varphi$ given by its coordinates
\[
\varphi_j^\delta(t,x)=\zeta(t)\eta(x)\mathcal{A}_j[\xi T_k(\rho_\delta)], \quad j=1,2,3,
\]
where $\zeta\in C_0^\infty(0,T)$ and $\eta,\xi\in C_0^\infty(\Omega)$, and use it as a test function in the momentum equation \eqref{e4.2} to obtain the following identity
\begin{equation}\label{e4.51}
\int_0^T\int_\Omega \zeta\eta\xi\Big( a\rho_\delta^\gamma + \delta\rho_\delta^\beta - (\lambda+2\mu)\div u_\delta \Big)T_k(\rho_\delta)dx\, ds = \sum_{i=1}^8J_i^\delta,
\end{equation}
where
\begin{align*}
J_1^\delta &= \int_0^T\int_\Omega \zeta\mathbb{S}_\delta \nabla\eta \cdot \mathcal{A}[\xi T_k(\rho_\delta)] \, dx\, ds,\\
J_2^\delta &= -\int_0^T\int_\Omega \eta(a\rho_\delta^\gamma +\delta\rho_\delta^\beta )\nabla\eta\cdot \mathcal{A}[\xi T_k(\rho_\delta)] \, dx\, ds,\\
J_3^\delta &= \sum_{i=1}^N\int_0^T\int_\Omega  \, \zeta\eta  \mathcal{A}[\xi T_k(\rho_\delta)]\cdot (\nabla c_j^{\delta}+ z_j c_j^{\delta} \nabla\Phi^{\delta})  dx\, ds,\\
J_4^\delta &= -\int_0^T\int_\Omega \zeta ([\rho_\delta u_\delta\otimes u_\delta]\nabla\eta)\cdot \mathcal{A}[\xi T_k(\rho_\delta)] \, dx\, ds,\\
J_5^\delta &= -\int_0^T\int_\Omega \zeta\eta\rho_\delta u_\delta\cdot \mathcal{A}[ T_k(\rho_\delta)\nabla\xi\cdot u_\delta] \, dx\, ds,\\
J_6^\delta &= -\int_0^T\int_\Omega \partial_t\zeta\eta\rho_\delta u_\delta\cdot\mathcal{A}[\xi T_k(\rho_\delta)]   \, dx\, ds,\\
J_7^\delta &= -\int_0^T\int_\Omega \zeta\eta\rho_\delta u_\delta\cdot\mathcal{A}[\xi T_k(\rho_\delta)\div u_\delta]   \, dx\, ds,\\
J_8^\delta &= -\int_0^T\int_\Omega 2\mu \zeta \xi T_k(\rho_\delta)\left[ (\nabla\Delta^{-1}\nabla):(u_\delta\otimes \nabla\eta)+u_\delta\cdot\nabla\eta \right]  \, dx\, ds.
\end{align*}
Similarly, taking $\varphi$, given by
\[
\varphi_j(t,x)=\zeta(t)\eta(x)\mathcal{A}_j[\xi \overline{T_k(\rho)}], \quad j=1,2,3,
\]
as a test function in \eqref{e4.200}, we obtain
\begin{equation}\label{e4.52}
\int_0^T\int_\Omega \zeta\eta\xi\Big( a\overline{\rho^\gamma}  - (\lambda+2\mu)\div u \Big)\overline{T_k(\rho)}dx\, ds = \sum_{i=1}^8J_i,
\end{equation}
where
\begin{align*}
J_1 &= \int_0^T\int_\Omega \zeta\mathbb{S}\nabla\eta \cdot \mathcal{A}[\xi \overline{T_k(\rho)}] \, dx\, ds,\\
J_2 &= -\int_0^T\int_\Omega a \eta\, \overline{\rho^\gamma}\, \nabla\eta\cdot \mathcal{A}[\xi \overline{T_k(\rho)}] \, dx\, ds,\\
J_3 &= \sum_{i=1}^N\int_0^T\int_\Omega  \, \zeta\eta  \mathcal{A}[\xi \overline{T_k(\rho)}]\cdot (\nabla c_j+ z_j c_j \nabla\Phi)  dx\, ds,\\
J_4 &= -\int_0^T\int_\Omega \zeta ([\rho u\otimes u]\nabla\eta)\cdot \mathcal{A}[\xi \overline{T_k(\rho)}] \, dx\, ds,\\
J_5 &= -\int_0^T\int_\Omega \zeta\eta\, \rho u\cdot \mathcal{A}[ \overline{T_k(\rho)}\nabla\xi\cdot u] \, dx\, ds,\\
J_6 &= -\int_0^T\int_\Omega \partial_t\zeta\eta\, \rho u\cdot\mathcal{A}[\xi \overline{T_k(\rho)}]   \, dx\, ds,\\
J_7 &= -\int_0^T\int_\Omega \zeta\eta\rho u\cdot\mathcal{A}[\xi \overline{T_k(\rho)\div u}]   \, dx\, ds,\\
J_8 &= -\int_0^T\int_\Omega 2\mu \zeta \xi \overline{T_k(\rho)}\left[ (\nabla\Delta^{-1}\nabla):(u\otimes \nabla\eta)+u\cdot\nabla\eta \right]  \, dx\, ds.
\end{align*}
As in \cite{F}, we have that all the terms on the right-hand-side of \eqref{e4.51} converge to their counterpart in \eqref{e4.52}. The only difference when compared to the context of \cite{F} is the form of the external force, which, in our present situation, corresponds to $J_3^\delta$, and which depends explicitly on the ion densities and on the self consistent potential that they generate. However, \eqref{e4.34} and \eqref{e4.37} combined with the properties of the operator $\mathcal{A}$, imply that $J_3^\delta$ converges to $J_3$. In summary, we obtain the following.

\begin{lemma}\label{l4.2}
There is a subsequence $\delta_n\to 0$, such that for any $\zeta\in C_0^\infty(0,T)$ and $\eta,\xi\in C_0^\infty(\Omega)$ we have
\begin{multline}
\lim_{\delta_n\to 0}\int_0^T\int_\Omega \zeta\eta\xi\Big( a\rho_{\delta_n}^\gamma + {\delta_n}\rho_{\delta_n}^\beta - (\lambda+2\mu)\div u_{\delta_n} \Big)T_k(\rho_{\delta_n})dx\, ds \\
= \int_0^T\int_\Omega \zeta\eta\xi\Big( a\overline{\rho^\gamma}  - (\lambda+2\mu)\div u \Big)\overline{T_k(\rho)}dx\, ds.
\end{multline}
In particular, we have that
\begin{equation}\label{e4.540}
\overline{a\rho^\gamma T_k(\rho)}-\overline{a\rho^\gamma}\, \overline{T_k(\rho)} = (\lambda+2\mu)\Big( \overline{T_k(\rho)\div u} - \overline{T_k(\rho)}\, \div u \Big), \text{ in $(0,T)\times K$},
\end{equation}
for any compact $K\subset \Omega$.
\end{lemma}

Now, with Lemma~\ref{l4.2} at hand, we may invoke the general reasoning from Chapter~12 of \cite{FKP} in order to conclude that, up to a subsequence, the densities $\rho_\delta$ are strongly convergent. Indeed, as shown in Section~12.2.2 of \cite{FKP}, Lemma~\ref{l4.2} guarantees the existence of a constant $C>0$ such that
\begin{equation}\label{e4.54}
\sup_{k\in\mathbb{N}}\limsup_{\delta_n\to 0}\int_0^T\int_\Omega |T_k(\rho_{\delta_n})-T_k(\rho)|^{\gamma+1}dx\, ds\le C.
\end{equation}
Moreover, \eqref{e4.54} combined with lemma~10 in \cite{FKP} implies that the limit functions $\rho$ and $u$ solve the continuity equation \eqref{e4.100} in the sense of renormalized solutions, that is, the following equation 
\begin{equation}\label{e4.renorm}
b(\rho)_t + \div(b(\rho u) + \Big(b'(\rho)\rho - b(\rho)\Big)\div u = 0,
\end{equation}
is satisfied in the sense of distributions, for any $b\in C^1(0,\infty)\cap C[0,\infty)$, such that
\begin{equation}\label{e4.condrenorm}
|b'(z)z|\le cz^{\frac{\gamma}{2}},\text{ for $z$ larger that some positive constant $z_0$,}
\end{equation}
cf. Remark~\ref{r1.2}.

Next, for each $k$, we consider the function 
\[
L_k(z)=z\int_1^z\frac{T_k(r)}{r^2}dr, \quad z\ge 0,
\]
which is a convex function and can be written as 
\[
L_k(z)=c_kz + b_k(z),
\]
for some function $b$ that satisfies \eqref{e4.condrenorm}. Note also that $zL_k'(z)-L_k(z)=T_k(z)$. Then, since $(\rho_\delta, u_\delta)$ satisfy the continuity equation in the sense of renormalized solutions, we have  that the following equations are satisfied in the sense of distributions:
\begin{equation}
L_k(\rho_\delta)_t + \div(L_k(\rho_\delta u_\delta) + T_k(\rho_\delta)\div u_\delta = 0.
\end{equation}
Integrating in $\Omega$ and letting $k\to\infty$,
\[
\frac{d}{dt}\int_\Omega \overline{L_k(\rho)}\, dx + \int_\Omega \overline{T_k(\rho)\div u}\, dx = 0, \quad \text{a.a. }t\in (0,T).
\]
Similarly, as $(\rho, u)$ also satisfy the continuity equation in the sense of distributions, we also find the identity
\[
\frac{d}{dt}\int_\Omega L_k(\rho)\, dx + \int_\Omega T_k(\rho)\div u\, dx = 0, \quad \text{a.a. }t\in (0,T).
\]

Consequently,
\begin{multline*}
\int_\Omega \Big(\overline{L_k(\rho)}-L_k(\rho)\Big)(t,x)\, dx =
- \int_0^t\int_\Omega \Big(\overline{T_k(\rho)\div u}- \overline{T_k(\rho)}\div u\Big)\, dx \, ds\\
\quad + \int_0^t\int_\Omega\Big(T_k(\rho)\div u - \overline{T_k(\rho)}\div u\Big)\, dx\, ds\\
\end{multline*}
and by \eqref{e4.540},
\begin{equation}\label{e4.59}
\int_\Omega \Big(\overline{L_k(\rho)}-L_k(\rho)\Big)(t,x)\, dx \le \int_0^t\int_\Omega\Big( T_k(\rho)\div u -\overline{T_k(\rho)}\div u \Big)\, dx\, ds.
\end{equation}
Regarding the integral on the right-hand-side of the inequality, we see that
\begin{align*}
\int_0^t\int_\Omega\Big|T_k(\rho)\div u - &\overline{T_k(\rho)}\div u\Big|\, dx\, ds \\
&\le \|\div u\|_{L^2((0,T)\times\Omega)}\| T_k(\rho) - \overline{T_k(\rho)} \|_{L^2(0,T)\times\Omega)}\\
    &\le C\| T_k(\rho) - \overline{T_k(\rho)} \|_{L^1(0,T)\times\Omega)}^\omega \| T_k(\rho) - \overline{T_k(\rho)} \|_{L^{\gamma+1}(0,T)\times\Omega)}^{1-\omega},
\end{align*}
for some $\omega\in (0,1)$. Noting that the weak lower semicontinuity of the norm yields
\begin{align*}
\| T_k(\rho) - \overline{T_k(\rho)} \|_{L^1(0,T)\times\Omega)}&\le \liminf_{\delta\to 0}\| T_k(\rho) - T_k(\rho_\delta) \|_{L^1(0,T)\times\Omega)}\\
   &\le k^{1-\gamma}\sup_\delta  \|\rho_\delta\|_{L^\gamma((0,T)\times\Omega)}\\
   &\le Ck^{1-\gamma},
\end{align*}
and using \eqref{e4.54}, we may take the limit as $k\to\infty$ in \eqref{e4.59} to obtain
\begin{equation}\label{e4.60}
\int_\Omega \Big(\overline{\rho\ln(\rho)}- \rho\ln(\rho)\Big)(t,x)\, dx\le 0.
\end{equation}
Finally, since the function $z\to z\ln(z)$ is continuous and convex, by Theorem~2.11 in \cite{F} we conclude that
\begin{equation}\label{e4.61}
\rho_{\delta_n}\to \rho\text{ in $L^1((0,T)\times\Omega)$}.
\end{equation} 

Thus, $\overline{a\rho^\gamma}=a\rho^\gamma$ implying that equation \eqref{e4.200} is exactly \eqref{e1.2}, which concludes the proof of Theorem~\ref{T1.1}.

%

\section{Large-Time Behavior}\label{S50}

This section is dedicated to the proof of Theorem~\ref{T1.ltb}, regarding the large-time behavior of the global finite energy weak solutions of the compressible PNPNS system. We begin by analyzing the convergence of a sequence of time shifts of the solutions, inspired by the arguments of Feireisl and Petzeltov\'{a} in \cite{FP} for the large-time behavior of weak solutions of the compressible Navier-Stokes equations.

For a sequence $\tau_n\to \infty$, we consider the sequence of shifts
\begin{equation*}
\begin{cases}
\rho_n(t,x)=\rho(\tau_n+t,x),\\
u_n(t,x)=u(\tau_n+t,x),\\
c_j^{(n)}(t,x)=c_j(\tau_n+t,x),\\
\Phi^{(n)}(t,x)=\Phi(\tau_n+t,x).
\end{cases}
\end{equation*}
By the energy inequality and estimate \eqref{e1.T_Phi}, we have that
\begin{align*}
 &\rho_n \text{ is bounded in $L^\infty(0,1;L^\gamma(\Omega))$},\\
&\sqrt{\rho_n}u_n \text{ is bounded in $L^\infty(0,1;L^2(\Omega))$},\\
&\rho_n u_n \text{ is bounded in $L^\infty(0,1;L^{\frac{2\gamma}{\gamma+1}}(\Omega))$},\\
&\sqrt{c_j^{(n)}} \text{ is bounded in $L^2(0,1;H^1(\Omega))$},\\
&\Phi^{(n)} \text{ is bounded in $L^\infty(0,1;H^1(\Omega))$, }\\
&\sqrt{c_j^{(n)}}\nabla\Phi^{(n)} \text{ is bounded in $L^2((0,T)\times\Omega)$}.
\end{align*}
Moreover, 
\begin{equation}\label{e50.1}
\lim_{n\to \infty} \int_0^1\left(\|\nabla u_n\|_{L^2(\Omega)}^2+\sum_{j=1}^N\left\|\nabla \sqrt{c_j^{(n)}}+z_j\sqrt{c_j^{(n)}}\nabla\Phi^{(n)}\right\|_{L^2(\Omega)}^2 \right)dt =0,
\end{equation}
and by the Poncaré inequality we have that
\[
\lim_{n\to\infty}\int_0^1\|u_n\|_{L^2(\Omega)}^2 dt = 0.
\]
In particular, we may apply Theorem~\ref{T1.2} to conclude that the sequence $(c_1^{(n)},...,c_N^{(n)},\Phi^{(n)})$ converges to a solution $(c_1^{(s)},...,c_N^{(s)},\Phi^{(s)})$ of the PNP subsystem with $u=0$.

Now, as in the proof of Lemma~\ref{l3.6} we see that $\sqrt{c_j^{(n)}+1}$ satisfies the following equation in $\mathcal{D}'((0,1)\times\Omega)$
\begin{multline}\label{e50.2}
2\partial_t\sqrt{c_j^{(n)}+1} =-\div\left(\Bigg(\frac{c_j^{(n)}}{c_j^{(n)}+1}\Bigg)^{1/2} \sqrt{c_j^{(n)}} u_n \right) - \Bigg(\frac{c_j^{(n)}}{(c_j^{(n)}+1)}\Bigg)^{3/2} u_n\cdot\nabla \sqrt{c_j^{(n)}}\\
	 + \div\left( \Bigg(\frac{c_j^{(n)}}{c_j^{(n)}+1}\Bigg)^{1/2}\left(2\nabla \sqrt{c_j^{(n)}} + z_j \sqrt{c_j^{(n)}}\nabla\Phi^{(n)}\right) \right) \\
	+  \frac{c_j^{(n)}}{(c_j^{(n)}+1)^{3/2}}\nabla\sqrt{c_j^{n}}\cdot\left(\nabla\sqrt{c_j^{(n)}} + z_j\sqrt{c_j^{(n)}}\nabla\sqrt{c_j^{(n)}}\cdot\nabla\Phi^{(n)}\right).
\end{multline}
Note that all the terms on the right-hand-side of \eqref{e50.2} converge to zero in the sense of distributions, by virtue of \eqref{e50.1}. We conclude that $\partial_t\sqrt{c_j^{(s)}+1} =0$ in the sense of distributions, which implies that ${c_j^{(s)}+1}$, and therefore also $c_j^{(s)}$, does not depend on $t$. As a consequence, the equation \eqref{e1.3} for the limiting functions reduces to
\begin{equation}\label{e50.5}
\div\left(D_j \nabla c_j^{(s)} + D_jz_j c_j^{(s)}\nabla\Phi^{(s)}\right)=0.
\end{equation}

Next, we claim that 
\begin{equation}\label{e50.6}
\int_\Omega \frac{D_j}{c_j^{(s)}}\left|\nabla c_j^{(s)}+z_jc_j\nabla\Phi^{(s)}\right|^2 dx = 0.
\end{equation}
Indeed, for $\delta>0$, we may use the function $\ln(c_j^{s}+\delta)+z_j\Phi$ as a test function in \eqref{e50.5} and, recalling that the functions $c_j^{(s)}$ (and therefore also $\Phi^{(s)}$ are independent of $t$), we obtain that
\begin{align*}
0&=\int_\Omega D_j\left(\nabla c_j^{(s)}+z_jc_j^{(s)}\nabla\Phi^{(s)}\right)\cdot\left(\frac{\nabla c_j^{(s)}}{c_j^{(s)}+\delta}+z_j\nabla\Phi^{(s)}\right) dx\\
 &=\int_\Omega \frac{D_j}{c_j^{(s)}+\delta}|\nabla c_j^{(s)}+z_j c_j^{(s)}\nabla\Phi^{(s)}|^2 dx \\
 &\qquad + \int_\Omega D_j\left(\nabla c_j^{(s)}+z_j c_j^{(s)}\nabla\Phi^{(s)}\right)\cdot z_j\left(1-\frac{c_j^{(s)}}{c_j^{(s)}+\delta}\right)\nabla\Phi^{(s)}dx.
\end{align*}
Then, 
\begin{align*}
&\int_\Omega \frac{D_j}{c_j^{(s)}+\delta}|\nabla c_j^{(s)}+z_j c_j^{(s)}\nabla\Phi^{(s)}|^2 dx \\
&\qquad =  -\int_\Omega D_j\left(\nabla c_j^{(s)}+z_j c_j^{(s)}\nabla\Phi^{(s)}\right)\cdot z_j\left(1-\frac{c_j^{(s)}}{c_j^{(s)}+\delta}\right)\nabla\Phi^{(s)}dx\\
&\qquad \le \frac{1}{2}\int_\Omega \frac{D_j}{c_j^{(s)}+\delta}|\nabla c_j^{(s)}+z_j c_j^{(s)}\nabla\Phi^{(s)}|^2 dx + C\int_\Omega \left(1-\frac{c_j^{(s)}}{c_j^{(s)}+\delta}\right)^2 (c_j^{(s)}+\delta)|\nabla\Phi^{(s)}|^2 dx,
\end{align*}
and we arrive at
\begin{equation}\label{e50.800}
\int_\Omega \frac{D_j}{c_j^{(s)}+\delta}|\nabla c_j^{(s)}+z_j c_j^{(s)}\nabla\Phi^{(s)}|^2 dx \le C\int_\Omega \left(1-\frac{c_j^{(s)}}{c_j^{(s)}+\delta}\right)^2 (c_j^{(s)}+1)|\nabla\Phi^{(s)}|^2 dx.
\end{equation}
Now, on the one hand, since we know that $(c_j^{(s)}+1)|\nabla \Phi|^2$ is integrable, the right-hand-side of \eqref{e50.800} converges to $0$ by the dominated convergence theorem. On the other hand, the left-hand-side of \eqref{e50.800} converges to 
\[
\int_\Omega \frac{D_j}{c_j^{(s)}}\left|\nabla c_j^{(s)}+z_jc_j\nabla\Phi^{(s)}\right|^2 dx
\]
by the monotone convergence theorem,  proving \eqref{e50.6}.
%

We infer from \eqref{e50.6} that
\begin{equation}\label{e50.700}
\nabla c_j^{(s)}+z_j c_j\nabla\Phi^{(s)} = 0, \text{ a.e. in $\Omega$,}
\end{equation}
which implies that there is a positive constant $Z_j$ such that
\begin{equation}\label{e50.80}
Z_jc_j^{(s)}(x)=e^{-z_j\Phi^{(s)}(x)}, \text{ a.e. in $\Omega$.}
\end{equation}
Since $c_j^{(n)}\to c_j^{(s)}$ strongly in $L^1((0,1)\times\Omega)$ and using the fact that the initial masses $I_j^0:=\int_\Omega c_j^{(0)} dx$ are conserved in time, we conclude that
\[
Z_j=(I_j^0)^{-1} \int_\Omega e^{-z_j\Phi^{(s)}(x)} dx.
\]
In other words, the functions $(c_1^{(s)},...,c_N^{(s)},\Phi^{(s)})$ are a solution of the Poisson-Boltzmann equation:
\begin{equation}\label{e50.8}
-\kappa\Delta \Phi^{(s)}=\sum_{j=1}^Nz_j I_j^{(0)}\frac{e^{-z_j\Phi^{(s)}}}{\int_\Omega e^{-z_j\Phi^{(s)}}(x) dx},
\end{equation}
with
\begin{equation}\label{e50.9}
(\partial_\nu\Phi^{(s)}+\tau\Phi^{(s)})|_{\partial\Omega}=V.
\end{equation}

We point out that solutions of this boundary value problem are unique. The proof of this fact goes by the same lines of that of Theorem 10 in Appendix A of \cite{CI1} and we refer to it for the details. 

\begin{remark}
In \cite{CI1} the boundary condition for the Poisson-Boltzmann equation \eqref{e50.8} is a Dirichlet condition and not a Robin boundary condition as \eqref{e50.9}. However, the proof of uniqueness can be modified easily. Indeed, the proof consists of considering the equation satisfied by the difference of two solutions, $\Psi$, multiplying by $\Psi$ and integrating by parts to find that $\kappa\|\nabla \Psi\|_{L^2(\Omega)}^2\le 0$. In the case considered in \cite{CI1}, $\Psi$ satisfies a homogeneous Dirichlet boundary condition. In our case, $\Psi$ satisfies a homogeneous Robin boundary condition, i.e., $(\partial_\nu \Psi + \tau\Psi)|_{\partial\Omega}=0$. And in this case, arguing as in \cite{CI1}, we obtain that $\kappa\|\nabla \Psi\|_{L^2(\Omega)}^2+\kappa\tau \|\Psi\|_{L^2(\partial\Omega)}^2\le 0$, which also leads to uniqueness.
\end{remark}

In conclusion, we have that the shifts $c_j^{(n)}$ converge in $L^1((0,1)\times\Omega)$ to $c_j^{(s)}$, given by \eqref{e50.80}, where $\Phi^{(s)}$ is the unique solution of the Poisson-Boltzmann equation \eqref{e50.8} with \eqref{e50.9}, and this is true for any sequence $\tau_n\to\infty$.

Let us now analyze the convergence of the fluid density. Using \eqref{e50.1}, by the Sobolev inequality and H\"{o}lder's inequality, we have that
\[
\lim_{n\to\infty}\int_0^1\left( \|\rho_n|u_n|^2\|_{L^{\frac{3\gamma}{\gamma+3}}}+\|\rho_n |u_n|\|_{^{\frac{6\gamma}{\gamma+6}}} \right)dt =0.
\]
Due to the boundedness of $\rho_n$ in $L^\infty(0,1;L^\gamma(\Omega))$ there is a subsequence (not relabelled) and a function $\rho_s$ such that
\[
\rho_n\to \rho_s \quad\text{weakly in $L^\gamma((0,1)\times\Omega)$.}
\]
Thus, taking a test function of the form $\psi(t)\phi(x)$ in the continuity equation \eqref{e1.1} and integrating over $[\tau_n,\tau_n+1]$, we see that
\[
\int_0^1\int_\Omega \rho_s\phi dx\psi'(t) dt = 0,
\]
which means that $\rho_s$ does not depend on $t$.

Also, by a similar procedure to that of Subsection~\ref{S4.1} (cf. \cite[Lemma~4.1]{FNP}), there is a constant $\vartheta>0$ such that
\[
\rho_n^{\gamma+\vartheta} \quad\text{is bounded in $L^1((0,1)\times\Omega)$,}
\]
and thus, we conclude that there is some function $\overline{p}$ such that
\begin{equation}\label{e50.2000}
a\rho^\gamma \to \overline{p}\quad\text{ weakly in $L^1((0,1)\times\Omega)$.}
\end{equation}
Now we may take the limit as $n\to\infty$ in the momentum equation in order to obtain that
\begin{equation}\label{e50.200}
\nabla\overline{p}=\sum_{j=1}^N\nabla c_j^{s}-\sum_{j=1}^Nz_j c_j^{(s)}\nabla\Phi^{(s)}, \quad\text{in $\mathcal{D}'((0,1)\times\Omega)$,}
\end{equation}
We  observe that \eqref{e50.700} and \eqref{e50.200}, implies that 
\[
\nabla\overline{p} = 0,
\]
which means that $\overline{p}$ is independent of $x$.

At this point, the arguments in \cite{FP} based on the $L^p$-version of the celebrated div-curl lemma and the use of the renormalized continuity equation \eqref{e1.renorm}, may be applied line by line to conclude that the convergence in \eqref{e50.2000} is actually strong. Hence $\overline{p}=a\rho_s^\gamma$ and we have that
\begin{equation}
\rho_n\to \rho_s\text{ strongly in $L^\gamma((0,1)\times\Omega)$.}
\end{equation}
We conclude that $\rho_s$ is a constant and, as in \cite{FP}, we have that this constant is the same for any sequence $\tau_n\to \infty$. 

Using the continuity equation \eqref{e1.1} we we may argue that
\[
\rho(t)\to \rho_s\text{ weakly in $L^\gamma(\Omega)$ as $t\to\infty$.}
\]
Indeed, for any sequence $t_n\to \infty$ we have that
\[
\int_{t_n}^{t_n+1} \int_\Omega \rho \phi dx dt - \int_\Omega \rho(t_n)\phi dx = \int_{t_n}^{t_n+1}\int_\Omega \rho u\cdot\nabla \phi dx dt \to 0, \text{ as $n\to\infty$}.
\]

In order to finish the proof, we need to show that $\displaystyle \lim_{t\to\infty}\rho(t)\to \rho_s$ and that $\displaystyle \lim_{t\to\infty}c_j(t)=c_j^{(s)}$ strongly. We do so by adapting the concluding argument from the proof of the analogous result in \cite{FP}.

Now, using equation \eqref{e1.3}, together with the fact that $u_s=0$ and that $\nabla c_j^{(s)}+z_jc_j^{(s)}\nabla\Phi^{(s)} =0$, we have that
\[
c_j(t)\to c_j^{(s)} \text{ weakly in $L^1(\Omega)\ln(L^1(\Omega))$}.
\]
Also, from the energy inequality \eqref{e1.energy} we have that the energy converges to a finite constant, when $t\to\infty$. Let
\[
E_\infty:=\overline{\lim_{t\to\infty}}E(t).
\]
Due to the strong convergence of $\rho_n$ to $\rho_s$ in $L^\gamma((0,1)\times\Omega)$, of $\rho_n |u_n|^2$ to $0$ in $L^1((0,1)\times \Omega)$, of $c_j^{(n)}$ to $c_j^{(s)}$ a.e. in $(0,1)\times\Omega$ and of $\Phi^{(n)}$ to $\Phi^{(s)}$ in $L^2((0,1)\times H^1(\Omega))$, as $n\to\infty$, we note that
\begin{align*}
E_\infty &= \int_\Omega\Bigg( \frac{a}{\gamma-1}\rho_s^\gamma + \sum_{j=1^N}(c_j^{(s)}\ln(c_j^{(s)})-c_j^{(s)}+1)+\frac{\kappa}{2}|\nabla\Phi^{(s)}|^2 \Bigg) dx + \frac{\kappa\tau}{2}\int_{\partial \Omega}|\Phi^{(s)}|^2 dS\\
  &\le \liminf_{t\to\infty}\Bigg[ \int_\Omega\Bigg( \frac{a}{\gamma-1}\rho(t)^\gamma + \sum_{j=1^N}(c_j(t)\ln(c_j(t))-c_j(t)+1)+\frac{\kappa}{2}|\nabla\Phi(t)|^2 \Bigg) dx \\
  &\qquad\qquad\qquad\qquad\qquad\qquad+ \frac{\kappa\tau}{2}\int_{\partial \Omega}|\Phi(t)|^2 dS\Bigg]\\
  &\le \overline{\lim_{t\to\infty}}\Bigg[ \int_\Omega\Bigg(\frac{1}{2}\rho(t)|u(t)|^2 +\frac{a}{\gamma-1}\rho(t)^\gamma + \sum_{j=1^N}(c_j(t)\ln(c_j(t))-c_j(t)+1)+\frac{\kappa}{2}|\nabla\Phi(t)|^2 \Bigg) dx \\
  &\qquad\qquad\qquad\qquad\qquad\qquad+ \frac{\kappa\tau}{2}\int_{\partial \Omega}|\Phi(t)|^2 dS\Bigg]\\
  &=\overline{\lim_{t\to\infty}}E(t)\\[5mm]
  &=E_\infty.
\end{align*}
Therefore, we have that
\begin{align}
\lim_{t\to\infty}F[\rho(t),c_1(t),...,c_N(t),\Phi(t)] = F[\rho_s,c_1^{(s)},...,c_N^{(s)},\Phi^{(s)}],
\end{align}
where, $F$ is the following functional 
\[
F[\rho_s,c_1,...,c_N,\Phi]:=\int_\Omega\Bigg( \frac{a}{\gamma-1}\rho^\gamma + \sum_{j=1^N}(c_j\ln(c_j)-c_j+1)+\frac{\kappa}{2}|\nabla\Phi|^2 \Bigg) dx + \frac{\kappa\tau}{2}\int_{\partial \Omega}|\Phi|^2 dS.
\]
By uniform convexity, we conclude that indeed, as $t\to\infty$, $\rho(t)\to \rho_s$ strongly in $L^\gamma(\Omega)$, $c_j(t)\to c_j^{(s)}$ strongly in $L^1(\Omega)$ and $\Phi(t)\to \Phi^{(s)}$ strongly in $H^1(\Omega)$ (cf. Theorem~2.11 in \cite{F}), while $u\to 0$ strongly in $L^2(\Omega)$, which concludes the proof of Theorem~\ref{T1.ltb}.

%

\section{The Incompressible Limit}\label{S5}

In this Section we discuss the incompressible limit for the PNPNS system as an application of Theorem~\ref{T1.2}. Our goal is to show that, in a large time scale, for small velocity, small ion concentrations and if the fluid density is close to a constant, the weak solutions of the compressible PNPNS system behave asymptotically as weak solutions of the corresponding incompressible system. 

More precisely, let $(\tilde{\rho},\tilde{u}, \tilde{c}_1,...,\tilde{c}_N,\tilde{\Phi})$ be a weak solution of the compressible PNPNS with certain physical coefficients, that is, of the system: 
\begin{align}
&\partial_t\tilde{\rho}+\div(\tilde{\rho} \tilde{u})=0,\label{e1.1t}\\
&\partial_t(\tilde{\rho} \tilde{u}) + \div(\tilde{\rho} \tilde{u}\otimes \tilde{u})+a\nabla\tilde{\rho}^\gamma  = \div\tilde{\mathbb{S}}-\sum_{j=1}^n\nabla\tilde{c}_j+\tilde{\kappa}\Delta\tilde{\Phi}\nabla\tilde{\Phi},\label{e1.2t}\\
&\partial_t \tilde{c}_{j}+\div(\tilde{c}_j\tilde{u})=\div\left( \tilde{D}_j\nabla\tilde{c}_j+\tilde{D}_jz_j\tilde{c}_j\nabla\tilde{\Phi} \right),\quad j=1,...,N\label{e1.3t}\\
&-\tilde{\kappa} \Delta\tilde{\Phi}=\sum_{j=1}^Nz_j\tilde{c}_j,\label{e1.5t}
\end{align}  
where 
\begin{equation*}
\tilde{\mathbb{S}}=\tilde{\lambda}(\div \tilde{u})I + \tilde{\mu}(\nabla \tilde{u} + (\nabla \tilde{u})^\perp).
\end{equation*}
 Then, for small $\varepsilon\in (0,1)$ we consider the scaling
\[
\tilde{\rho}=\rho(\varepsilon t,x),\quad \tilde{u}=\varepsilon u(\varepsilon t,x), \quad \tilde{c}_j=\varepsilon^2 c_j(\varepsilon t,x),\quad \tilde{\Phi}=\Phi(\varepsilon t,x).
\]
We also take the viscosity coefficients, the diffusion coefficients and the dielectric permittivity as
\[
\tilde{\mu}=\varepsilon\mu,\quad \tilde{\lambda}=\varepsilon \lambda, \quad \tilde{D}_j=\varepsilon D_{j},\quad \tilde{\kappa}=\varepsilon^2\kappa,
\]
where the scaled coefficients are such that
 $\mu>0$, $\lambda+\frac{2}{3}\mu\ge 0$, $D_j >0$ and $\kappa>0$.
Then, the scaled system reads as
\begin{align}
&\partial_t\rho_\varepsilon+\div(\rho_\varepsilon u_\varepsilon)=0,\label{e1.1c}\\
&\partial_t(\rho_\varepsilon u_\varepsilon) + \div(\rho_\varepsilon u_\varepsilon\otimes u_\varepsilon)+\frac{a}{\varepsilon^2}\nabla\rho_\varepsilon^\gamma  = \div\mathbb{S}_\varepsilon-\sum_{j=1}^n\nabla c_j^{(\varepsilon)}+\kappa\Delta\Phi^{(\varepsilon)}\nabla\Phi^{(\varepsilon)},\label{e1.2c}\\
&\partial_t c_{j}^{(\varepsilon)}+\div(c_j^{(\varepsilon)}u_\varepsilon)=\div\left( D_{j}\nabla c_j^{(\varepsilon)} +D_{j}z_jc_j^{(\varepsilon)}\nabla\Phi^{(\varepsilon)} \right),\quad j=1,...,N\label{e1.3c}\\
&-\kappa \Delta\Phi^{(\varepsilon)}=\sum_{j=1}^Nz_jc_j^{(\varepsilon)},\label{e1.5c}
\end{align}  
with
\begin{equation}
\mathbb{S}_\varepsilon=\lambda(\div u_\varepsilon)I + \mu(\nabla u_\varepsilon + (\nabla u_\varepsilon)^\perp).
\end{equation}
Accordingly, the corresponding initial and boundary conditions read as follows:
\begin{align}
&(\rho_\varepsilon,\rho_\varepsilon u_\varepsilon,c_j^{(\varepsilon)})(0,\cdot)=(\rho_{0,\varepsilon},m_{0,\varepsilon},c_{j}^{0,\varepsilon})(\cdot),\label{e1.9c}\\
&u_\varepsilon|_{\partial \Omega}=0,\label{e1.10c}\\
&D_{j}\left(\partial_\nu c_j^{(\varepsilon)} + z_jc_j^{(\varepsilon)}\partial_\nu\Phi^{(\varepsilon)}\right)|_{\partial \Omega} = 0,\quad j=1,...,N,\label{e1.11c}\\
&(\partial_\nu\Phi^{(\varepsilon)} + \tau\Phi^{(\varepsilon)})|_{\partial \Omega}=V.\label{e1.13c}
\end{align}

In contrast, the incompressible PNPNS system is
\begin{align}
&\div u=0,\label{e1.1i}\\
&\partial_t u + (u\cdot\nabla)u+\nabla \mathfrak{p}  = \mu\Delta u +\kappa\Delta\Phi\nabla\Phi,\label{e1.2i}\\
&\partial_t c_{j}+\div(c_ju)=\div\left( D_{j}\nabla c_j +D_{j}z_jc_j\nabla\Phi \right),\quad j=1,...,N\label{e1.3i}\\
&-\kappa \Delta\Phi=\sum_{j=1}^Nz_jc_j,\label{e1.5i}
\end{align}  
and we consider the following initial and boundary conditions
\begin{align}
&(u,c_j)(0,\cdot)=(u_0,c_{j}^0)(\cdot),\label{e1.9i}\\
&u|_{\partial \Omega}=0,\label{e1.10i}\\
&D_{j}\left(\partial_\nu c_j + z_jc_j\partial_\nu\Phi\right)|_{\partial \Omega} = 0,\quad j=1,...,N,\label{e1.11i}\\
&(\partial_\nu\Phi + \tau\Phi)|_{\partial \Omega}=V.\label{e1.13i}
\end{align}

Owning to the known results on the incompressible limit of the Navier-Stokes equations, it is reasonable to expect that as $\varepsilon\to 0$, the global weak solutions of \eqref{e1.1c}-\eqref{e1.13c} converge to the weak solutions of \eqref{e1.1i}-\eqref{e1.13i}, as $\varepsilon\to 0$, where the pressure $\mathfrak{p}$ is a ``limit" of 
\[
\frac{a}{\varepsilon^2}(\rho_\varepsilon^\gamma - 1) + \sum_{j=1}^Nc_j^{(\varepsilon)}.
\]
In the absence of the diluted ion species, the system \eqref{e1.1c}-\eqref{e1.5c} reduces to the compressible Navier-Stokes equations. The incompressible limit for the Navier-Stokes equations was proved in \cite{LiM1} for the equations posed on the whole space and on a periodic domain, and in \cite{DGLiM} on bounded domains with the no-slip boundary condition (see also \cite{LiM2,DG}). The case of a bounded domain is more difficult due to the subtle interactions between dissipative effects and wave propagation near the boundary. The method in \cite{DGLiM} to prove the weak convergence of the velocity relies on a spectral analysis of the semigroup generated by the dissipative wave operator together with Duhamel's principle, and the convergence is strong if the domain satisfies a certain geometrical condition.  Namely, consider the following overdetermined problem:
\begin{equation}\label{e1.32}
-\Delta \phi=\lambda \phi,\text{ in $\Omega$,}\qquad \partial_\nu\phi=0,\text{ on $\partial\Omega$,}\qquad\text{and $\phi$ is constant on $\partial\Omega$.}
\end{equation}
A solution to \eqref{e1.32} is said to be trivial if $\lambda=0$ and $\phi$ is a constant; and the domain $\Omega$ is said to satisfy assumption (H) if all solutions of \eqref{e1.32} are trivial (cf. \cite{DGLiM}). 

Here, we discuss the analogous result regarding the incompressible limit of the PNPNS equations. Let us recall that the existence of global weak solutions in the incompressible case, i.e. of the initial-boundary value problem \eqref{e1.1i}-\eqref{e1.13i}, has been proved in \cite{FS}.

Let $(\rho_\varepsilon, u_\varepsilon,c_j^{(\varepsilon)},\Phi^{(\varepsilon)})$ be the weak solution of \eqref{e1.1c}-\eqref{e1.13c} provided by Theorem~\ref{T1.1}. Then, in particular, we have that
\begin{multline}\label{e1.energyc}
E_\varepsilon(t) + \int_0^t\int_\Omega \Big(\mu|\nabla u_\varepsilon|^2 + (\lambda+\mu)(\div u_\varepsilon)^2 + \sum_{j=1}^ND_{j}c_j^{(\varepsilon)}\Big|\frac{\nabla c_j^{(\varepsilon)}}{c_j^{(\varepsilon)}}+ z_j\nabla\Phi^{(\varepsilon)}\Big|^2\Big)dx\, ds
\le E_\varepsilon(0),
\end{multline}
where
\begin{align*}
E_\varepsilon(t):=&\int_\Omega\Bigg( \rho_\varepsilon\left(\frac{1}{2} |u_\varepsilon|^2 + \frac{a}{\varepsilon^2(\gamma-1)}\rho_\varepsilon^{\gamma-1} \right)  
+ \sum_{j=1}^N(c_j^{(\varepsilon)}\ln c_j^{(\varepsilon)} - c_j^{(\varepsilon)} + 1) + \frac{\kappa}{2}|\nabla\Phi^{(\varepsilon)}|^2\Bigg) dx \\
&+\frac{\kappa\tau}{2}\int_{\partial\Omega}|\Phi^{(\varepsilon)}|^2dS.
\end{align*}
Due to mass conservation, the expression for the energy $E_\varepsilon$ in both sides of \eqref{e1.energyc} may be replaced by
\begin{multline*}
\tilde{E}_\varepsilon(t):=\int_\Omega\Bigg( \frac{1}{2} \rho_\varepsilon|u_\varepsilon|^2 + \frac{a}{\varepsilon^2(\gamma-1)}\Big(\rho_\varepsilon^{\gamma}-\gamma\rho_\varepsilon+(\gamma-1)  \Big)  \\
+ \sum_{j=1}^N(c_j^{(\varepsilon)}\ln c_j^{(\varepsilon)} - c_j^{(\varepsilon)} + 1) + \frac{\kappa}{2}|\nabla\Phi^{(\varepsilon)}|^2\Bigg) dx +\frac{\kappa\tau}{2}\int_{\partial\Omega}|\Phi^{(\varepsilon)}|^2dS.
\end{multline*}

Let us assume that the initial data $(\rho_{0,\varepsilon},m_{0,\varepsilon},c_{j}^{0,\varepsilon})$ satisfy \eqref{e1.initial}. Moreover, we assume that
\begin{equation}\label{e1.42}
\begin{cases}
\rho_{0,\varepsilon}^{-1/2}m_{0,\varepsilon}\rightharpoonup u_0 \text{ weakly in $L^2(\Omega)$,}\\[3mm]
c_j^{0,\varepsilon}\to c_j^0 \text{ in $L^1(\Omega)$,}
\end{cases}
\end{equation}
for some $u_0\in L^2(\Omega)$ and $c_j^0\in L^1(\Omega)$ with $c_j^0\ln c_j^0\in L^1(\Omega)$. Moreover, we assume that
\begin{multline}\label{e1.43}
\int_\Omega\Bigg( \frac{1}{2} \rho_{0,\varepsilon}^{-1/2}|m_{0\varepsilon}|^2 + \frac{a}{\varepsilon^2(\gamma-1)}\Big(\rho_{0,\varepsilon}^{\gamma}-\gamma\rho_{0,\varepsilon}+(\gamma-1)  \Big)  \\
+ \sum_{j=1}^N(c_j^{0,\varepsilon}\ln c_j^{0,\varepsilon} - c_j^{0,\varepsilon} + 1) + \frac{\kappa}{2}|\nabla\Phi^{0,\varepsilon}|^2\Bigg) dx +\frac{\kappa\tau}{2}\int_{\partial\Omega}|\Phi^{0,\varepsilon}|^2dS\le C,
\end{multline}
for some positive constant $C$, independent of $\varepsilon$. Assumption \eqref{e1.43} implies, in particular, that the initial density is of order $1+O(\varepsilon)$ since
\[
\rho_{0,\varepsilon}^{\gamma}-\gamma\rho_{0,\varepsilon}+(\gamma-1) = \rho_{0,\varepsilon}^\gamma - 1 - \gamma(\rho_0-1).
\]

Let us finally recall the definition of Leray's projectors $P$ and $Q$ defined in $L^2(\Omega)^3$ onto the space of divergence-free vector fields and onto the space of gradients, respectively, defined by
\[
P=I-Q, \quad Q=\nabla \Delta_N^{-1}\div,
\]
where $\Delta_N$ denotes the Laplace operator with Neumann boundary conditions, meaning that $f=\Delta_N^{-1}g$ if $\Delta f =g$ in $\Omega$ with $\partial_\nu f = 0$ and $\int_\Omega f dx =0$. In particular, for $f\in L^2(\Omega)^3$,
\[
f=Pf+Qf,\quad \text{with}\quad \div(Pf)=0,\quad \curl(Qf)=0.
\]

With these notations, we may state the following result.

\begin{theorem}\label{T1.3}
Let $(\rho_\varepsilon, u_\varepsilon,c_j^{(\varepsilon)},\Phi^{(\varepsilon)})$ be a sequence of finite energy weak solutions of the compressible PNPNS equations \eqref{e1.1c}-\eqref{e1.5c} subject to the initial and boundary conditions \eqref{e1.9c}-\eqref{e1.13c} in a smooth bounded domain $\Omega$ of $\mathbb{R}^3$, where the initial conditions satisfy \eqref{e1.42} and \eqref{e1.43}. Then, up to a subsequence, as $\varepsilon\to 0$, $(\rho_\varepsilon, u_\varepsilon,c_j^{(\varepsilon)},\Phi^{(\varepsilon)})$ converges to a weak solution $(u,\mathfrak{p},c_j,\Phi)$ of the incompressible PNPNS equations \eqref{e1.1i}-\eqref{e1.5i} with \eqref{e1.10i}-\eqref{e1.13i} and with initial data $u|_{t=0}=Pu_0$ and $c_j|_{t=0}=c_j^{0}$. More precisely,
\begin{align*}
&\rho_\varepsilon\to 1 \text{ in $C(0,T;L^\gamma(\Omega))$,}\\
&u_\varepsilon\to u \text{ weakly in $L^2((0,T)\times\Omega)$ and strongly if $\Omega$ satisfies condition (H),}\\
&c_{j}^{(\varepsilon)}\to c_j \text{ strongly in $L^1(0,T;L^p(\Omega))$ for $1\le p<3$,}\\
&\nabla c_j^{(\varepsilon)}\rightharpoonup \nabla c_j \text{ weakly in $L^2(0,T;L^1(\Omega))\cap L^1(0,T;L^q(\Omega))$, for $1\le q<3/2$,}\\
&\nabla\Phi^{(\varepsilon)}\rightharpoonup \nabla\Phi \text{ weakly-* in $L^\infty(0,T;L^2(\Omega))$,}\\
&\Phi^{(\varepsilon)}\to \Phi \text{ strongly in $C([0,T];L^p(\Omega))$, for $1\le p< 6$.}
\end{align*}

Moreover, there are $r_1, r_2>1$ such that 
\begin{align*}
&c_j^{(\varepsilon)}\nabla\Phi^{(\varepsilon)}\rightharpoonup c_j\nabla\Phi \text{ weakly in $L^{r_1}((0,T)\times\Omega)$, and}\\
&c_j^{(\varepsilon)}u_\varepsilon\rightharpoonup c_j u \text{ weakly in $L^{r_2}((0,T)\times\Omega)$.}
\end{align*}
\end{theorem}

Let us point out that the energy inequality \eqref{e1.energyc} and assumption \eqref{e1.43} imply that, up to a subsequence, $u_\varepsilon$ converges weakly in $L^2(0,T;H^1(\Omega))$ to some limit $u$. Then, Theorem~\ref{T1.2} readily implies the convergence of $(c_j^{(\varepsilon)},\Phi^{(\varepsilon)})$ in the sense stated in Theorem~\ref{T1.3} and that the limit functions are a solution to the initial-boundary value problem for the Poisson-Nernst-Planck subsystem \eqref{e1.3i}, \eqref{e1.5i}, with \eqref{e1.11i}, \eqref{e1.13i}. Moreover, the forcing terms in the momentum equation related to the evolution of the ion densities are suitably bounded and converge nicely, which means that the proof of Theorem~\ref{T1.3} reduces to studying the incompressible limit of the Navier-Stokes equations with a well behaved forcing term, and the proof in \cite{DGLiM} may be adapted accordingly.

For completeness and for the convenience of the reader we include below a sketch of the proof of Theorem~\ref{T1.3}, following \cite{DGLiM} and focusing on the relevant modifications that have to be made in our context.

\begin{proof} Let $(\rho_\varepsilon, u_\varepsilon,c_j^{(\varepsilon)},\Phi^{(\varepsilon)})$ be a finite energy weak solution of \eqref{e1.1c}-\eqref{e1.13c}. By \eqref{e1.energyc} and \eqref{e1.43} we have, in particular, that
\begin{align}
&\sqrt{\rho_\ve}u_\ve \text{ is bounded in $L^\infty(0,T;L^2(\Omega))$,}\\
&\frac{a}{\ve^2(\gamma-1)}(\rho_\ve^\gamma - \gamma\rho_\ve+\gamma-1)\text{ is bounded in $L^\infty(0,T;L^1(\Omega))$,}\label{e5.2}\\
&\nabla u_\ve \text{ is bounded in $L^2(0,T;L^2(\Omega))$,}\label{e5.3}\\
&\sqrt{c_j^{(\varepsilon)}}\text{ is bounded in $L^2(0,T;H^1(\Omega))$, $j=1,...,N,$}\\
&\Phi^{(\ve)} \text{ is bounded in $L^\infty(0,T;H^1(\Omega))$,}\\
&\sqrt{c_j^{(\ve)}}\nabla\Phi^{(\ve)} \text{ is bounded in $L^2((0,T)\times\Omega)$, $j=1,...,N$.}
\end{align}
Reasoning as in \cite{LiM1}, it follows from \eqref{e5.2} that 
%
\begin{equation}\label{e5.12}
\rho_\ve\to 1\text{ in $C([0,T];L^\gamma(\Omega))$,}\quad \text{as $\ve\to 0$,}
\end{equation}
and also, using the continuity equation \eqref{e1.1c}, that
\[
\div u_\ve \rightharpoonup 0\text{ weakly in $L^2(0,T;L^2(\Omega))$,}
\]
for all $T>0$.

Now, by virtue of \eqref{e5.3}, we may assume that, up to a subsequence, $u_\ve\rightharpoonup u$ for some $u\in L^2(0,T;H_0^1(\Omega))$. Then, by Theorem~\ref{T1.2} we have that there are $c_j$, $j=1,...,N$ and $\Phi$ such that,
\begin{align*}
&c_{j}^{(\varepsilon)}\to c_j \text{ strongly in $L^1(0,T;L^p(\Omega))$ for $1\le p<3$,}\\
&\nabla c_j^{(\varepsilon)}\rightharpoonup \nabla c_j \text{ weakly in $L^2(0,T;L^1(\Omega))\cap L^1(0,T;L^q(\Omega))$, for $1\le q<3/2$,}\\
&\nabla\Phi^{(\varepsilon)}\rightharpoonup \nabla\Phi \text{ weakly-* in $L^\infty(0,T;L^2(\Omega))$,}\\
&\Phi^{(\varepsilon)}\to \Phi \text{ strongly in $C([0,T];L^p(\Omega))$, for $1\le p< 6$.}
\end{align*}
Moreover, there are $r_1, r_2>1$ such that,
\begin{equation}\label{e5.13}
\begin{split}
&c_j^{(\varepsilon)}\nabla\Phi^{(\varepsilon)}\rightharpoonup c_j\nabla\Phi \text{ weakly in $L^{r_1}((0,T)\times\Omega)$, and}\\
&c_j^{(\varepsilon)}u_\varepsilon\rightharpoonup c_j u \text{ weakly in $L^{r_2}((0,T)\times\Omega)$,}
\end{split}
\end{equation}
and the limit functions $u$, $(c_1,...,c_N,\Phi)$ are a weak solution of \eqref{e1.3i}-\eqref{e1.5i} and \eqref{e1.11i}-\eqref{e1.13i}.

Next, we consider the projection of the continuity equation \eqref{e1.2c} onto the space of divergence-free vector fields:
\begin{equation}
\partial_t P(\rho_\varepsilon u_\varepsilon) + P(\div (\rho_\varepsilon u_\varepsilon\otimes u_\varepsilon))  = \mu_\ve\Delta Pu_\ve- P\left(\Big(\sum_{j=1}^Nz_j c_j^{(\ve)}\Big)\nabla\Phi^{(\varepsilon)}\right),
\end{equation}
which implies that $\partial_t P(\rho_\ve u_\ve)$ is bounded in $L^1(0,T;H^{-1}(\Omega))+L^2(0,T;H^{-1}(\Omega))+L^{r_1}((0,T)\times(\Omega))$. Then, as in \cite{LiM1} we may apply Lemma~5.1 in \cite{L2} in order to conclude that $P(\rho_\ve u_\ve)\cdot P(u_\ve)$ converges to $|u|^2$ in the sense of distributions. Then, we conclude that $P (u_\ve)$ converges in $L^2(0,T;L^2(\Omega))$ to $u=P (u)$, since $P(u_\ve)\rightharpoonup u$ weakly in $L^2(0,T;L^2(\Omega))$ and 
\[
\left|\int_0^T\int_\Omega (|Pu_\ve|^2 - P(\rho_\ve u_\ve)\cdot P(u_\ve))dx dt\right| \le C\|\rho_\ve - 1\|_{C([0,T];L^\gamma(\Omega))}\|u_\ve\|_{L^2(0,T;L^{2\gamma/(\gamma-1)}(\Omega))}.
\]

At this point, we have to show that the limit functions solve the incompressible PNPNS system \eqref{e1.1i}-\eqref{e1.5i} with \eqref{e1.10i}-\eqref{e1.13i}; particularly, that the momentum equation \eqref{e1.2i} is satisfied. As in \cite{DGLiM}, the main difficulty lies in the convergence of the term $\div(\rho_\ve u_\ve\otimes u_\ve)$, which amounts to analyzing the convergence of the gradient part of the velocity $Q(u_\ve)$. Indeed, we already have strong convergence of $P(u_\ve)$ to $u$, strong convergence of $\rho_\ve$ to $1$ and also the convergence of the forcing terms related to the evolution of the ion densities, cf. \eqref{e5.13}. Then, in order to conclude, it is enough to show that $\div (Q(\rho_\ve u_\ve)\otimes Q(u_\ve))$ converges weakly to a gradient, which can then be incorporated into the pressure. 

Here, the argument consists of splitting $Q(u_\ve)=Q_1(u_\ve)+Q_2(u_\ve)$, where
\[
Q_1(u_\ve)=\sum_{k\in I}\left\langle Q(u_\ve), \frac{\nabla\Psi_{k,0}}{\lambda_{k,0}}\right\rangle \frac{\nabla\Psi_{k,0}}{\lambda_{k,0}}, \quad Q_2(u_\ve)=\sum_{k\in J}\left\langle Q(u_\ve), \frac{\nabla\Psi_{k,0}}{\lambda_{k,0}}\right\rangle \frac{\nabla\Psi_{k,0}}{\lambda_{k,0}},
\]
for certain sets of indexes $I$ and $J$, such that $I\cup J=\mathbb{N}$, where $(\lambda_{k,0}^2)_{k\in\mathbb{N}}$ ($\lambda_{k,0}>0$) are the non-decreasing sequence of eigenvalues and $(\Psi_{k,0})_{k\in\mathbb{N}}$ the corresponding eigenvectors in $L^2(\Omega)$ of the Laplace operator with homogeneous Neumann boundary condition, $-\Delta_N$, so that
\[
-\Delta \Psi_{k,0}=\lambda_{k,0}^2\Psi_{k,0} \text{ in $\Omega$,}\quad \partial_\nu \Psi_{k,0}=0\text{ on $\partial\Omega$}.
\] 
In order to define the sets $I$ and $J$, we consider the wave operators $A_0$ and $A_\ve$ in $\mathcal{D}'(\Omega)\times \mathcal{D}'(\Omega)^3$ given by
\[
A_0\begin{pmatrix} \Psi\\ m \end{pmatrix} = \begin{pmatrix}\div m\\ \nabla \psi \end{pmatrix},
\]
and
\[
A_\ve\begin{pmatrix} \Psi\\ m \end{pmatrix} = \begin{pmatrix} \div m\\ \nabla \psi \end{pmatrix} + \ve \begin{pmatrix} 0\\ \mu \Delta m + (\lambda + \mu)\nabla\div m\end{pmatrix}.
\]
The eigenvalues $A_0$ are $(\pm i\lambda_{k,0})_{k\ge 1}$ and the corresponding eigenvectors are
\[
\phi_{k,0}^{\pm}=\begin{pmatrix}\Psi_{k,0}\\ m_{k,0}^\pm=\pm\frac{\nabla\Psi_{k,0} }{i\lambda_{k,0}} \end{pmatrix},
\]
i.e.
\[
A_0 \phi_{k,0}^{\pm}=\pm i \lambda_{k,0}\phi_{k,0}^\pm\text{ in $\Omega$,}\quad m_{k,0}^\pm\cdot \nu = 0 \text{ on $\partial\Omega$.}
\]

We also recall the following lemma from \cite{DGLiM}.

\begin{lemma}\label{l5.1}
Let $\Omega$ be a $C^2$ bounded domain of $\mathbb{R}^d$ and let $k\ge 1$ and $M\ge 0$. Then, there exist approximate eigenvalues $i\lambda_{k,\ve,M}^\pm$ and eigenvectors $\phi_{k,\ve,M}^\pm=(\Psi_{k,\ve,M}^\pm, m_{k,\ve,M}^\pm)\perp$ of $A_\ve$ such that
\begin{equation}\label{e5.170}
A_\ve \phi_{k,\ve,M}^\pm = i\lambda_{k,\ve,M}^\pm\phi_{k,\ve,M}^\pm+R_{k,\ve,M}^\pm,
\end{equation} 
with
\begin{equation}\label{e5.17}
i\lambda_{k,\ve,M}^\pm=\pm i\lambda_{k,0}+i\lambda_{k,1}^\pm\sqrt{\ve}+O(\ve),\quad \text{where $Re(i\lambda_{k,1}^\pm)\le 0$,}
\end{equation}
and for all $p\in [1,\infty]$, we have
\begin{equation}
\|R_{k,\ve,M}^\pm\|_{L^p(\Omega)}\le C_p(\sqrt{\ve})^{M+1/p}, \quad\text{and}\quad \|\phi_{k,\ve,M}^\pm-\phi_{k,0}^\pm\|_{L^p(\Omega)}\le C_p(\sqrt{\ve})^{1/p}.
\end{equation}
\end{lemma}

\begin{remark}
From the construction in \cite{DGLiM}, we have 
\[
i\lambda_{k,1}^\pm = -\frac{1\pm i}{2}\sqrt{\frac{\mu}{2\lambda_{k,0}^3}}\int_{\partial\Omega}|\Psi_{k,0}|^2 dS.
\]
In particular, if $Re(i\lambda_{k,1}^\pm)= 0$, then we actually have that i$\lambda_{k,1}^\pm=0$. In this case, $m_{k,0}$ vanishes on $\partial \Omega$.
\end{remark}

The next step is to define $I\subset \mathbb{N}$ to be the set of indices such that $Re(i\lambda_{k,1}^\pm)< 0$ and $J=\mathbb{N}\setminus I$; and the goal is to prove that $Q_1(u_\ve)\to 0$ in $L^2((0,T)\times\Omega)$ and that $\curl \div(Q_2(m_\ve)\otimes Q_2(u_\ve))$ converges to $0$ in the sense of distributions, which means that $\div(Q_2(m_\ve)\otimes Q_2(u_\ve))$ converges to a gradient in the sense of distributions. 

We denote the density fluctuation by 
\[
\Psi_\ve:=\frac{\rho_\ve -1}{\ve},
\]
 the momentum by 
\[
 m_\ve:=\rho_\ve u_\ve
 \]
and also we define $\phi_\ve = (\Psi_\ve, m_\ve)^\perp$.
As in \cite{DGLiM}, the proof is reduced to showing that $b_{k,\ve}^\pm=\langle\phi_\ve,\phi_{k,\ve,2}\rangle$ converges to $0$ in $L^2(0,T)$, for each $k\in I$ and analyzing its oscillations  when $k\in J$.

Now, denoting by $A_\ve^*$ the adjoint of $A_\ve$, it follows from the continuity and the momentum equations \eqref{e1.1c}, \eqref{e1.2c} that  $\phi_\ve$ solves the following equation:
\begin{equation}\label{e5.19}
\partial_t\phi_\ve -\frac{A_\ve^*\phi_\ve}{\ve}=\begin{pmatrix} 0\\ g_\ve \end{pmatrix}
\end{equation}
where
\begin{equation}
g_\ve=-\ve\mu\Delta(u_\ve\Psi_\ve)-\ve(\lambda+\mu)\nabla\div(u_\ve\Psi_\ve)-\div(m_\ve\otimes u_\ve)-\pi_\ve-\sum_{j=1}^Nz_jc_j^{(\ve)}\nabla\Phi^{(\ve)},
\end{equation}
and
\[
\pi_\ve :=  \frac{a\Big(\rho_\ve^\gamma - 1 -\gamma(\rho_\ve -1)\Big)}{\ve^2} + \sum_{j=1}^N c_j^{(\ve)}.
\]
Note that $\pi_\ve$ is bounded in $L^\infty(0,T;L^1(\Omega))$.
Then, we take the scalar product of equation \eqref{e5.19} with $\phi_{k,\ve,2}^\pm$ in order to obtain
\begin{equation}\label{e5.21}
\frac{d}{dt}b_{k,\ve}^\pm(t) -\overline{\frac{i\lambda_{k,\ve,2}^\pm}{\ve}}b_{k,\ve}^\pm(t)=d_{k,\ve}^\pm(t),
\end{equation}
where $d_{k,\ve}^\pm(t)=\langle g_\ve, m_{k,\ve,2}^\pm\rangle + \ve^{-1}\langle\phi_\ve,R_{k,\ve,2}^\pm\rangle$.
We conclude, by Duhamel's formula that
\begin{equation}
b_{k,\ve}^\pm=b_{k,\ve}^\pm(0)e^{\overline{i\lambda_{k,\ve,2}^\pm}t/\ve}+\int_0^td_{k,\ve}^\pm(s)e^{\overline{i\lambda_{k,\ve,2}^\pm}(t-s)/\ve}ds.
\end{equation}

Here, the argument is split in two cases.\\

\noindent{\em\bf The case $k\in I$:} 

In this case, $b_{k,\ve}^\pm$ may be estimated using relation \eqref{e5.17} with $M=2$. First, we see that
\[
\|e^{\overline{i\lambda_{k,\ve,2}^\pm}t/\ve}b_{k,\ve}^\pm(0)\|_{L^2(0,T)}\le C\|e^{Re(\overline{i\lambda_{k,1}^\pm})t/\sqrt{\ve}}b_{k,\ve}^\pm(0)\|_{L^2(0,T)}\le C\ve^{1/4}.
\]

Next, we see that \eqref{e5.17} implies, in particular, that for any $a\in L^q(0,T)$ and $1\le p,q\le\infty$ with $p^{-1}+q^{-1}=1$ we have
\begin{equation}\label{e5.24}
\left|\int_0^t e^{\overline{i\lambda_{k,\ve,2}^\pm}(t-s)/\ve} a(t) ds\right|\le \int_0^te^{Re(\overline{i\lambda_{k,1}^\pm})(t-s)/\ve}|a(s)|ds\le C\|a\|_{L^q(0,T)}(\sqrt{\ve})^{1/p}.
\end{equation}
Then, we write $|d_{k,\ve}^\pm(t)|\le d_1(t)+d_2(t)+d_3(t)+d_4(t)+d_5(t)$, where
\begin{align*}
&d_1=\left|\int_\Omega(m_\ve\otimes u_\ve)(t)\cdot\nabla m_{k,\ve,2}^\pm dx\right|,\\
&d_2=\left|\int_\Omega\pi_\ve(t)\div m_{k,\ve,2}^\pm dx\right|,\\
&d_3=\ve\left|\int_\Omega (\mu\Delta + (\lambda+\mu)\nabla\div)m_{k,\ve,2}^\pm\cdot (u_\ve\Psi_\ve)(t) dx \right|,\\
&d_4=\ve^{-1}|\langle \phi_\ve,R_{k,\ve,2}\rangle|,\\
&d_5=\sum_{j=1}^N\left|\int_\Omega z_j(c_j^{(\ve)}\nabla\Phi^{(\ve)})(t)\cdot m_{k,\ve,2}^\pm dx\right|.
\end{align*}
Exactly as in \cite{DGLiM} we have the following bounds
\begin{align*}
&d_1(t)\le C\|u_\ve^1\|_{L^\infty(0,T;L^2(\Omega))}\|\nabla u_\ve\|_{L^2(\Omega)} + C\ve^{1/2}\|\nabla u\|_{L^2(\Omega)}^2+C\ve^{1/2}\|\nabla u_\ve\|_{L^2(\Omega)},\\
&d_2(t)\le C\|\pi_\ve\|_{L^\infty(0,T;L^1(\Omega))}(\|\Psi_{k,\ve,2}^\pm\|_{L^\infty(\Omega)}+\|R_{k,\ve,2}^\pm\|_{L^\infty(\Omega)})\le C,\\
&d_3(t)\le C\|\nabla u\|_{L^2(\Omega)},\\
&d_4(t)\le  C\ve^{1/2 - 1/2\vartheta},
\end{align*}
where $\vartheta=\min\{2,\gamma\}$.

Finally, recalling that $c_j^{(\ve)}\nabla\Phi^{(\ve)}$ is bounded in $L^{r_1}((0,T)\times\Omega)$ (cf. \eqref{e5.13}), we estimate
\[
d_5(t)\le \|\nabla m_{k,\ve,2}^\pm\|_{L^\infty(\Omega)}\|c_j^{(\ve)}\nabla\Phi^{(\ve)}\|_{L^{r_2}(\Omega)}.
\]
As a consequence, we may apply \eqref{e5.24} repeatedly in order to conclude that $b_{k,\ve}^\pm$ converges strongly to $0$ in $L^2(0,T)$. \\

\noindent{\em\bf The case $k\in J$:} 

In this case, we have that $\lambda_{k,1}^\pm=0$. From \eqref{e5.21}, we have that $e^{\pm i\lambda_{k,0}t/\ve}b_{k,\ve}^\pm$ is bounded in $L^2(0,T)$ and that $\frac{d}{dt}(e^{\pm i\lambda_{k,0}t/\ve}b_{k,\ve}^\pm)$ is bounded in $\sqrt{\ve}L^1(0,T)+L^p(0,T)$ for some $p>1$. We conclude that, up to a subsequence, $e^{\pm i\lambda_{k,0}t/\ve}b_{k,\ve}^\pm$ converges strongly in $L^2(0,T)$ to some $b_{osc}^\pm$.

At this point, we see that the particular structure of the forcing terms due to the interactions with the ion species does not play any additional role in the weak convergence to a gradient of the interaction terms $$\div\left( \sum_{k,j\in J,\, j,k\le K}\rho_\ve b_{k,\ve}^\pm b_{l,\ve}^\pm \frac{\nabla\Psi_{k,0}}{\lambda_{k,0}}\otimes \frac{\nabla\Psi_{l,0}}{\lambda_{l,0}}\right)$$ for any given $K>0$, and that the arguments of \cite{DGLiM,G,LiM1} (cf. \cite{HW2,WY}) may be carried out line by line in order to conclude that $\div(\rho_\ve Q_2(u_\ve)\otimes Q_2(u_\ve))$ converges to a gradient in the sense of distributions, which completes the proof of Theorem~\ref{T1.3}.
\end{proof}



\appendix

\section{Energy Estimates for Other Bboundary Conditions}\label{A1}

As mentioned in Section~\ref{S:2}, other boundary conditions for the model might be shown to provide good estimates, starting from the energy equation \eqref{e.energydif}. For instance, if instead of \eqref{e1.13} one considers the following Dirichlet boundary condition on the potential (which has been considered in the incompressible case in e.g. \cite{CI1}),
\begin{equation}\tag{\ref{e1.13}*}\label{e1.13**}
\Phi|_{\partial\Omega}=V,
\end{equation}
then, it is still possible to recover an integral energy inequality. Indeed, in this case the only difference is the treatment of the term $\div(\kappa\Phi\nabla\Phi_{t})$, which can be handled as follows. We write $\Phi=\Phi_0+\Phi_1$ where $\Phi_0$ is the solution of the stationary problem 
\begin{align}
&-\kappa\Delta \Phi_0 = 0,\label{e.phi0}\\
&\Phi_0|_{\partial \Omega}=V,\label{e.phi0data}
\end{align}
and $\Phi_1$ is the solution of the problem 
\begin{align}
&-\kappa\Delta\Phi_1 = \sum_{j=1}^Nz_jc_j,\\
&\Phi_1|_{\partial \Omega}=0.\label{e.Phi1b}
\end{align} 
In light of this decomposition, using \eqref{e.Phi1b}, noting that $\Phi_0$ does not depend on $t$ and also recalling \eqref{e1.5}, we see that
\begin{align*}
\int_\Omega \div(\kappa\Phi\nabla\Phi_t)dx &=  \int_\Omega \div(\kappa\Phi_0\nabla\Phi_t)dx\\
    &=\frac{d}{dt} \int_\Omega \left(\kappa\nabla\Phi_0\cdot\nabla\Phi + \Phi_0\sum_{j=1}^Nz_jc_j \right)dx.
\end{align*}

Thus, we conclude that
\begin{multline}\label{e.energyint*}
\frac{d}{dt}\tilde{E}(t) + \int_\Omega \Big(\mu|\nabla u|^2 + (\lambda+\mu)(\div u)^2 + \sum_{j=1}^ND_j c_j|\nabla (\sigma_j'(c_j)+ z_j\Phi)|^2\Big)dx = 0,
\end{multline}
where,
\begin{equation}
\tilde{E}(t)=\int_\Omega\left( \rho\left(\frac{1}{2} |u|^2 + \frak{e}(\rho)\right) + \sum_{j=1}^N\sigma_j(c_j) + \frac{\kappa}{2}|\nabla\Phi|^2 +\kappa\nabla\Phi_0\cdot\nabla\Phi + \Phi_0\sum_{j=1}^Nz_jc_j \right)dx.
\end{equation}

Note that the function $\Phi_0$ is smooth and depends only on the given boundary data $V(x)$ (cf. \eqref{e.phi0} and \eqref{e.phi0data}). Moreover, the blocking boundary conditions \eqref{e1.11} imply that for all $t>0$ and $j=1,...,N$,
\begin{equation}
\int_\Omega c_j(t,x)dx=\int_\Omega c_j^0(x)dx.
\end{equation}
Putting these observation together with the fact that $c_j$ are nonnegative 
 we see that 
\begin{multline*}
\int_\Omega\left( \kappa\nabla\Phi_0\cdot\nabla\Phi+ \Phi_0\sum_{j=1}^Nz_jc_j \right)dx \\
\ge -\frac{\kappa}{4}\int_\Omega|\nabla\Phi|^2dx - \kappa\int_\Omega|\nabla\Phi_0 |^2 dx - N\|\Phi_0\|_{L^\infty(\Omega)}\max_j{|z_j|}\|c_j^0\|_{L^1(\Omega)}.
\end{multline*}

In conclusion, 
\begin{equation}
\tilde{E}(t)\ge \int_\Omega\left( \rho\left(\frac{1}{2} |u|^2 + \frak{e}(\rho)\right) + \sum_{j=1}^N\sigma_j(c_j) + \frac{\kappa}{4}|\nabla\Phi|^2 \right)dx - C_0,
\end{equation}
for some finite constant that depends only on the function $V(x)$ and on the initial ion masses, which means that \eqref{e.energyint} provides satisfactory {\it a priori} estimates upon integration in $t$.

\bigskip
	
\section*{Acknowledgments}
D.~Marroquin 
acknowledges the support from CNPq, through grant proc. 310021/2023-5 and by the Coordenação de Aperfeiçoamento de Pessoal de Nível Superior - Projeto CAPES - Print UFRJ - 2668/2018/88881.311616/2018-01.
 D. Wang was supported in part by National Science Foundation  grants  DMS-1907519 and DMS-2219384.

\end{document}